\documentclass[a4paper, 11pt]{amsart}

\usepackage[margin=1in,marginparwidth=0.8in, marginparsep=0.1in]{geometry}

\usepackage{graphicx}
\usepackage{marginnote}

\usepackage[british,UKenglish,USenglish,english,american]{babel}

\usepackage{newlfont}
\usepackage{amssymb}
\usepackage{amsmath}
\usepackage{amsfonts}
\usepackage{latexsym}
\usepackage{amsthm}
\usepackage{mathrsfs}
\usepackage{stmaryrd}
\usepackage[all,cmtip]{xy}
\usepackage{caption}
\usepackage{enumerate}

\UseRawInputEncoding

\usepackage{amsmath,amssymb,amscd,xypic}
\usepackage{amsfonts}
\usepackage[leqno]{amsmath}
\usepackage{mathrsfs} 
\usepackage{enumerate}

\usepackage[usenames,dvipsnames]{xcolor}

\usepackage{layout}
\usepackage{fullpage}

\usepackage[backref=page]{hyperref}

\hypersetup{
 colorlinks,
 citecolor=Green,
 linkcolor=blue,
 urlcolor=Blue}

\theoremstyle{plain}                    
\newtheorem{teo}{Theorem}[subsection]     
\newtheorem{theoremalpha}{Theorem}

\newtheorem{prop}[teo]{Proposition}

\newtheorem{cor}[teo]{Corollary}       
\newtheorem{lem}[teo]{Lemma}            
\newtheorem{deflem}[teo]{Definition/Lemma} 
\theoremstyle{definition}               
\newtheorem{notations}[teo]{}
\newtheorem{defin}[teo]{Definition}
\newtheorem{assumption}[teo]{Assumption} 

\theoremstyle{remark}
\newtheorem{rmk}[teo]{Remark}

\newenvironment{sis}{\left\{\begin{aligned}}{\end{aligned}\right.}

\newcommand{\longhookrightarrow}{\lhook\joinrel\longrightarrow}

\newcommand{\bbR}{{\mathbb R}}

\newcommand{\bbN}{{\mathbb N}}

\newcommand{\bbQ}{{\mathbb Q}}
\newcommand{\bbZ}{{\mathbb Z}}

\newcommand{\bbG}{{\mathbb G}}

\newcommand{\bbE}{{\mathbb E}}
\newcommand{\bbF}{{\mathbb F}}

\newcommand{\cC}{{\mathcal C}}

\newcommand{\cE}{{\mathcal E}}

\renewcommand{\cL}{{\mathcal L}}

\newcommand{\cM}{{\mathcal M}}

\newcommand{\cO}{{\mathcal O}}

\newcommand{\un}{\underline}
\newcommand{\ov}{\overline}
\newcommand{\wt}{\widetilde}
\newcommand{\wh}{\widehat}

\renewcommand{\rm}{\mathrm}
\newcommand{\scr}{\mathscr}

\numberwithin{equation}{subsection}

\newcommand{\mt}{\mathcal}

\DeclareMathOperator{\Pic}{Pic}
\DeclareMathOperator{\Aut}{Aut}

\DeclareMathOperator{\RPic}{RPic}
\DeclareMathOperator{\NS}{NS}
\DeclareMathOperator{\Hom}{Hom}

\DeclareMathOperator{\coker}{coker}
\renewcommand{\Im}{\text{Im}}
\DeclareMathOperator{\End}{End}
\DeclareMathOperator{\id}{id}
\DeclareMathOperator{\res}{res}
\DeclareMathOperator{\red}{red}

\DeclareMathOperator{\rk}{rk}

\DeclareMathOperator{\Gm}{\mathbb{G}_{m}}

\DeclareMathOperator{\PGL}{PGL}
\DeclareMathOperator{\PSL}{PSL}
\DeclareMathOperator{\GL}{GL}
\DeclareMathOperator{\SL}{SL}
\renewcommand{\div}{\mathrm{div}}

\DeclareMathOperator{\SO}{SO}
\DeclareMathOperator{\CSO}{CSO}
\DeclareMathOperator{\PSO}{PSO}
\DeclareMathOperator{\Spin}{Spin}
\DeclareMathOperator{\CSpin}{CSpin}
\DeclareMathOperator{\Sp}{Sp}
\DeclareMathOperator{\CSp}{CSp}
\DeclareMathOperator{\PSp}{PSp}

\DeclareMathOperator{\lcm}{lcm}
\DeclareMathOperator{\ord}{ord}
\DeclareMathOperator{\Cl}{Cl}

\newcommand{\Mg}{\mathcal M_{g,n}}
\newcommand{\Cg}{\mathcal C_{g,n}}

\newcommand{\bg}[1]{\mathrm{Bun}_{#1,g,n}}

\DeclareMathOperator{\Bun}{\mathrm{Bun}}

\newcommand{\bgr}[1]{\mathfrak{Bun}_{#1,g,n}}
\DeclareMathOperator{\Bunr}{\mathfrak{Bun}}

\DeclareMathOperator{\Bil}{Bil}

\DeclareMathOperator{\Sym}{Sym}
\DeclareMathOperator{\ev}{ev}

\renewcommand{\ss}{\mathrm{ss}}
\newcommand{\ab}{\mathrm{ab}}

\newcommand{\ad}{\mathrm{ad}}
\renewcommand{\sc}{\mathrm{sc}}

\newcommand{\g}{\mathfrak g}

\newcommand{\roo}{\mathrm{roots}}
\newcommand{\coroo}{\mathrm{coroots}}
\newcommand{\wei}{\mathrm{weights}}
\newcommand{\cowei}{\mathrm{coweights}}

\DeclareMathOperator{\w}{wt}
\DeclareMathOperator{\obs}{obs}

\renewcommand\arraystretch{1.5}

\title{The Picard group of the universal moduli stack of principal bundles on pointed smooth curves II}

\author{Roberto Fringuelli}
\address{Roberto Fringuelli, Dipartimento di Matematica, Universit\`a di Roma ``Tor Vergata'', Via della Ricerca Scientifica 1, I-00133 Roma, Italy}
\email{fringuel@mat.uniroma2.it}

\author{Filippo Viviani}
\address{Filippo Viviani,
Dipartimento di Matematica e Fisica 
Universit\`a Roma Tre 
Largo San Leonardo Murialdo  
I-00146 Roma  Italy }
\email{viviani@mat.uniroma3.it}

\begin{document}

\begin{abstract}
In this paper, which is a sequel of  \cite{FV1}, we  investigate, for any reductive group $G$ over an algebraically closed field $k$,  the Picard group of the universal moduli stack $\bg{G}$ of  $G$-bundles over $n$-pointed  smooth projective curves of genus $g$. In particular, we give new functorial presentations of the  Picard group of $\bg{G}$, we study the restriction homomorphism onto the Picard group of the moduli stack of principal $G$-bundles over a fixed smooth curve, we determine the Picard group of the rigidification of $\bg{G}$ by the center of $G$ as well as the image of the obstruction homomorphism of the associated gerbe. As a consequence, we compute the divisor class group of the moduli space of semistable $G$-bundles over $n$-pointed  smooth projective curves of genus $g$. 
\end{abstract}


\maketitle

\tableofcontents

\section{Introduction}

The aim of this paper, which is a sequel of the paper \cite{FV1}, is to study 
 the Picard group of the \emph{universal moduli stack of (principal) $G$-bundles $\bg{G}$}, which parametrizes  $G$-bundles, where $G$ is a connected and smooth linear algebraic group  $k=\ov k$,  over families of (connected, smooth and projective) $k$-curves of genus $g\geq 0$ endowed with $n\geq 0$ pairwise disjoint ordered sections. We refer the reader  to \cite{FV1} for the motivation behind this investigation as well as for its relationship with previous results in the literature.


Recall (see Theorem \ref{T:propBunG}) that the stack $\bg{G}$ is an algebraic stack, locally of finite type and smooth over the moduli stack $\Mg$ of $n$-marked curves of genus $g$ and  its connected  components (which are integral and smooth over $k$) are in functorial bijection with the fundamental group $\pi_1(G)$. We will denote the connected components and the restriction of the forgetful morphism by 
$$\Phi_G^{\delta}: \bg{G}^{\delta}\to \Mg \quad \text{ for any } \delta\in \pi_1(G). $$


We proved in \cite[Thm. A]{FV1} that if $\red: G\twoheadrightarrow G^{\red}$ is the  reductive quotient of $G$, i.e. the quotient of $G$ by its unipotent radical, then for any $\delta\in \pi_1(G)\xrightarrow[\cong]{\pi_1(\red)}\pi_1(G^{\red})$ the pull-back homomorphism 
$$
\red_\#^*:\Pic(\bg{G^{\red}}^{\delta})\xrightarrow{\cong} \Pic(\bg{G}^{\delta}) 
$$
is an isomorphism.   Hence, throughout this paper, we will restrict to the case of a \emph{reductive group} $G$. We fix a maximal torus $\iota: T_G\hookrightarrow G$ and let $\scr W_G$ be the Weyl group of $G$.

Since the Picard group of $\Mg$ is well-known up to torsion (and completely known if $\mathrm{char}(k)\neq 2$ by \cite{FV2}) and  the pull-back morphism 
$$(\Phi_G^{\delta})^*:\Pic(\Mg)\to \Pic(\bg{G}^{\delta})$$
is injective since $\Phi_G^{\delta}$  is fpqc and cohomologically flat in degree zero (see Theorem \ref{T:propBunG} for the definition), we can focus our attention on the relative Picard group 
$$
\RPic(\bg{G}^{\delta}):=\Pic(\bg{G}^{\delta})/(\Phi_G^{\delta})^*(\Pic(\Mg)).
$$

Throughout this paper, we will mainly restrict to the case of \emph{positive genus}; the case $g=0$ is easier, see Remarks \ref{R:res-g0} and \ref{R:weight-g0}. 

The relative Picard group $\RPic(\bg{G}^{\delta})$ was described in \cite[Thm. C]{FV1}: it is generated by the image of a functorial transgression homomorphism 
$$\tau_{G}^{\delta}:\Sym^2(\Lambda^*(T_G))^{\scr W_G}\cong \Bil^{s,\ev}(\Lambda(T_G))^{W_G}\hookrightarrow \RPic\Big(\bg{G}^{\delta}\Big),$$
where $\Bil^{s,\ev}(\Lambda(T_G))^{W_G}$ is the lattice of $\scr W_G$-invariant even symmetric bilinear forms on the lattice $\Lambda(T_G)$ of cocharacters of a maximal torus $T_G$ in $G$, and by the image of the pull-back  homomorphism 
$$\ab_\#^*: \RPic\left(\bg{G^{\ab}}^{\delta^{\ab}}\right)\hookrightarrow \RPic\Big(\bg{G}^{\delta}\Big),$$
where $\ab:G\twoheadrightarrow G^{\ab}$ is the maximal abelian quotient and $\delta^{\ab}:=\pi_1(\ab)(\delta)\in \pi_1(G^{\ab})$.
Moreover, the image of $\ab_\#^*$ coincides with the subgroup generated by the tautological line bundles, see \cite[Thm. B]{FV1}.


The first result of this paper is a new description of $\RPic(\bg{G}^{\delta})$ in terms of three functorial exact sequences.

\begin{theoremalpha}\label{T:thmA}(see Corollary \ref{C:Pic-red}, Theorem \ref{T:gG}, Theorem \ref{T:oG+gG})
Assume that $g\geq 1$. Let  $G$ be a reductive group and fix $\delta\in \pi_1(G)$.
Then the relative Picard group of $\bg{G}^{\delta}$ sits into the following functorial commutative diagram with exact rows
\begin{equation}\label{E:diagthmA}
\xymatrix{
0 \ar[r] &  \Lambda^*(G^{\ab})\otimes H_{g,n}  \ar[r]^{j_G^{\delta}} \ar@{^{(}->}[d] & \RPic(\bg{G}^{\delta})\ar[r]^{\omega_G^{\delta}\oplus \gamma_G^{\delta}}  \ar@{=}[d]  & \NS(\bg{G}^{\delta}) \ar@{->>}[d]^{\res_{G}^{\NS}} \\
0  \ar[r] &  \Lambda^*(G^{\ab})\otimes \wh H_{g,n}   \ar[r]^{i_G^{\delta}} \ar@{^{(}->}[d] &  \RPic(\bg{G}^{\delta})  \ar[r]^{\gamma_G^{\delta}} \ar@{=}[d] & \Bil^{s,\scr D-\ev}(\Lambda(T_G))^{\scr W_G}  \ar[r] \ar@{->>}[d]^{\res_{\scr D}} & 0\\
0  \ar[r] & \RPic(\bg{G^{\ab}}^{\delta^{\ab}})  \ar[r]^{\ab_\#^*}  & \RPic(\bg{G}^{\delta})   \ar[r]^(0.4){\theta_G^{\delta}} &  \Bil^{s,\ev}(\Lambda(T_{\scr D(G)})\vert \Lambda(T_{G^{\ss}}))^{\scr W_G}  \ar[r] & 0,\\
}
\end{equation}
where the left vertical morphism are injective and the right vertical morphisms are surjective.

Moreover, the image of $\omega_G^{\delta}\oplus \gamma_G^{\delta}$ is equal to 
\begin{equation*}
 \Im(\omega_G^{\delta}\oplus \gamma_G^{\delta})=
 \begin{cases} 
\NS(\bg{G}^{\delta}) & \text{ if } n\geq 1,\\
  \left\{([\chi], b)\in \NS(\bg{G}^{\delta})\: : \: 
 \begin{aligned}
 &  \left[\chi(x)-b(\delta\otimes x)\right]+(g-1)b(x\otimes x) \\
 & \text{ is divisible by } 2g-2,  \text{ for any } x\in \Lambda(T_G) 
 \end{aligned}
 \right\} & \text{ if } n=0. 
\end{cases} 
 \end{equation*}
\end{theoremalpha}

In the above diagram: 
\begin{itemize}
\item $\Lambda^*(G^{\ab})$ is the lattice of characters of $G^{\ab}$; 
\item the groups $\wh H_{g,n}$ and $H_{g,n}$ are defined by 
$$\wh H_{g,n}:=
\begin{cases} 
\bbZ\oplus \bbZ^n  & \text{ if } g\geq 2,\\
 \bbZ^n& \text{ if } g=1.\\
\end{cases}
\text{ and } 
H_{g,n}:=
\begin{cases} 
\{(m,\zeta)\in \bbZ\oplus \bbZ^n\: : \: (2g-2)m+|\zeta|=0\} & \text{ if } g\geq 2,\\
\{\zeta\in \bbZ^n\: : \: |\zeta|=0\} & \text{ if } g=1.\\
\end{cases}
$$
\item  $\Bil^{s,\scr D-\ev}(\Lambda(T_G))^{\scr W_G}$ is the lattice of $\scr W_G$-invariant symmetric bilinear forms on $\Lambda(T_G)$ which are even on $\Lambda(T_{\scr D(G)})\otimes \Lambda(T_{\scr D(G)})$, 
where $T_{\scr D(G)}$ is the maximal torus of the derived subgroup $\scr D(G)$ of $G$, see Corollary \ref{C:forms-LTG}; 
\item  $\Bil^{s,\ev}(\Lambda(T_{\scr D(G)})\vert \Lambda(T_{G^{\ss}}))^{\scr W_G}$ is the lattice of $\scr W_G$-invariant even symmetric bilinear forms on $\Lambda(T_{\scr D(G)})$ which are integral on $\Lambda(T_{\scr D(G)})\otimes \Lambda(T_{G^{\ss}})$, where $T_{G^{\ss}}$ is the maximal torus of the semisimplification $G^{\ss}$ of $G$, see Proposition \ref{P:forms-LTG}; 
\item  $\NS(\bg{G}^{\delta})\subset \Lambda^*(T_G)/\Lambda^*(T_{G^{\ad}})\oplus \Bil^{s,\scr D-\ev}(\Lambda(T_G))^{\scr W_G}$, where $G^{\ad}=G/\scr Z(G)$ with $\scr Z(G)$ center of $G$, is  introduced in Definition \ref{D:NS-BunG} and further studied in Definition/Lemma \ref{D:funcNS} and Proposition \ref{P:seqNS}.  
\item the homomorphisms $j_G^{\delta}$, $\omega_G^{\delta}$, $\gamma_G^{\delta}$,  $i_G^{\delta}$ and $\theta_G^{\delta}$ are defined in, respectively, 
Theorem \ref{T:oG+gG}, Definition/Lemma \ref{D:wtG}, Definition/Lemma \ref{D:gammaG}, Theorem \ref{T:gG} and Corollary \ref{C:Pic-red}.
\end{itemize}

The second result of this paper is a description of the restriction homomorphism
\begin{equation}\label{E:resINT}
\res_G^{\delta}(C): \RPic(\bg{G}^{\delta})\to \Pic(\Bun_G^{\delta}(C))
\end{equation}
for any $(C,p_1,\ldots, p_n)\in \Mg(k)$, where $\mathrm{Bun}_{G}(C)$ is the $k$-stack of \emph{$G$-bundles on $C$}, i.e. the stack over $k$ whose $S$-points 
$\mathrm{Bun}_{G}(C)(S)$ is the groupoid of $G$-bundles on $C_S:=C\times_k S$ for any $k$-scheme $S$, and it is canonically identified with the fiber of $\Phi_{G,g,n}^{\delta}$ over  $(C,p_1,\ldots,p_n)\in \Mg(k)$. 
The Picard group $\Pic(\Bun_G^{\delta}(C))$ has been described by Biswas-Hoffman in \cite{BH10}, see Theorem \ref{T:Pic-BunGC} and Proposition \ref{P:NS-BunGC}.


\begin{theoremalpha}\label{T:thmB}(see Theorem \ref{T:resPic}, Theorem \ref{T:resNS}, Corollary \ref{C:resNS-coker}) 
Assume that $g\geq 1$ and let $(C,p_1,\ldots, p_n)\in \Mg(k)$ be a geometric point. Let $G$ be a reductive group 
 and fix $\delta\in \pi_1(G)$. 
 \begin{enumerate}
\item The restriction homomorphism \eqref{E:resINT}  sits into the following functorial commutative diagram with exact rows
\begin{equation}\label{E:thmB1}
\xymatrix{
0 \ar[r] &  \Lambda^*(G^{\ab})\otimes H_{g,n}  \ar[r]^{j_G^{\delta}} \ar[d]^{\res_G^{\delta}(C)^o}& \RPic(\bg{G}^{\delta})\ar[r]^{\omega_G^{\delta}\oplus \gamma_G^{\delta}} \ar[d]^{\res_G^{\delta}(C)}&  \NS(\bg{G}^{\delta}) \ar[d]^{\res_G^{\delta}(C)^{\NS}}\\
 0 \ar[r] & \Hom(\pi_1(G),J_C(k))\ar[r]^(0.55){j_G^{\delta}(C)} &  \Pic(\Bun_G^{\delta}(C))\ar[r]^{c_G^{\delta}(C)} &  \NS(\Bun_G^{\delta}(C))\ar[r] & 0,
}
\end{equation}
\item The composition $\ov{\res_G^{\delta}(C)}:=c_G^{\delta}(C)\circ \res_G^{\delta}(C)$  sits into the following functorial commutative diagram with exact rows
\begin{equation}\label{E:thmB2}
\xymatrix{ 
0 \ar[r] &   \RPic\left(\bg{G^{\ab}}^{\delta^{\ab}}\right) \ar[r]^{\ab_\#^*} \ar[d]^{\ov{\res_{G^{\ab}}^{\delta^{\ab}}(C)}} & \RPic\Big(\bg{G}^{\delta}\Big)  \ar[r]^(0.4){\theta_G^{\delta}} \ar[d]^{\ov{\res_G^{\delta}(C)}}&  \Bil^{s,\ev}(\Lambda(T_{\scr D(G)})\vert \Lambda(T_{G^{\ss}}))^{\scr W_G}  \ar@{^{(}->}[d]^{r_G} \ar[r] & 0,\\
0 \ar[r] &  \NS(\Bun_{G^{\ab}}^{\delta^{\ab}}(C))\ar[r]^(0.55){\ab^{*,\NS}(C)}&  \NS(\Bun_{G}^{\delta}(C)) \ar[r]^(0.3){p} &   \Bil^{s, \sc-\ev}(\Lambda(T_{\scr D(G)})\vert \Lambda(T_{G^{\ss}}))^{\scr W_G} \ar[r] & 0
} 
\end{equation} In particular, if  $\id_{J_C}: \bbZ\xrightarrow{\cong} \End(J_C)$ is an isomorphism (which is true if $k$ is uncountable and $(C,p_1,\ldots, p_n)\in \Mg(k)$ is very general),  then   $\coker(\ov{\res_G^{\delta}(C)})$ sits into a canonical short exact sequence
\begin{equation*}
0\to \coker(\omega_{G^{\ab}}^{\delta^{\ab}}\oplus \gamma_{G^{\ab}}^{\delta^{\ab}})\to \coker(\ov{\res_G^{\delta}(C)})\to  \coker(r_G)\to 0,
\end{equation*}
where 
$$
\coker(\omega_{G^{\ab}}^{\delta^{\ab}}\oplus \gamma_{G^{\ab}}^{\delta^{\ab}})=
\begin{cases}
0 & \text{ if } n>0,\\
\left(\frac{\bbZ}{(2g-2)\bbZ}\right)^{\dim G^{\ab}} & \text{ if } n=0.  
\end{cases}
$$
\end{enumerate}
\end{theoremalpha}
The group $\Bil^{s,\ev}(\Lambda(T_{\scr D(G)})\vert \Lambda(T_{G^{\ss}}))^{\scr W_G}$ has been defined after Theorem \ref{T:thmA} while 
\begin{itemize}
\item $\Bil^{s,\sc-\ev}(\Lambda(T_{\scr D(G)})\vert \Lambda(T_{G^{\ss}}))^{\scr W_G}$ is the lattice of $\scr W_G$-invariant symmetric bilinear forms on $\Lambda(T_{\scr D(G)})$ which are integral on $\Lambda(T_{\scr D(G)})\otimes \Lambda(T_{G^{\ss}})$ and even on $\Lambda(T_{G^{\sc}})\otimes \Lambda(T_{G^{\sc}})$, where $T_{G^{\sc}}$ is the maximal torus of the simply-connected cover $G^{\sc}$ of the semisimplification $G^{\sc}$ of $G$, see Definition-Lemma \ref{D:evGtilde}.
\end{itemize}
It follows that the cokernel of $r_G$ is a $2$-elementary abelian group of rank bounded by the number of the simple factors of the semisimple part $\g^{\ss}$ of the Lie algebra $\g$ of $G$ (see Definition/Lemma \ref{D:evGtilde} and Corollary \ref{C:rank-Bil}). In Section \ref{Sec:Examples}, we compute $\coker(r_G)$ for all reductive groups $G$ such that $\g^{\ss}$ is a simple Lie algebra. 


The third result of this paper is the computation of the relative Picard group of the rigidification 
\begin{equation*}\label{E:rigidINT}
\nu_G^{\delta}:\bg{G}^{\delta} \to \bg{G}^{\delta}\fatslash \scr Z(G):=\bgr{G}^{\delta},
\end{equation*}
of the stack $\bg{G}^{\delta}$ by the center $\scr Z(G)$ of $G$, which acts functorially on any $G$-bundle. 
This is also closely related to the divisor class group of the adequate moduli space $M_{G,g,n}^{\delta, ss}$ of the locus $\bg{G}^{\delta, ss}\subseteq \bg{G}^{\delta}$ parametrizing semistable $G$-bundles over $n$-marked curves of genus $g$, or equivalently of its image  $\bgr{G}^{\delta, ss}$ in the $\scr Z(G)$-rigidification $\bgr{G}^{\delta}$. See \S\ref{Sec:Mss} for a discussion of the properties of the loci 
$\bg{G}^{\delta, ss}$ and $\bgr{G}^{\delta, ss}$, as well as for the state of the art on the existence of the adequate moduli space $M_{G,g,n}^{\delta, ss}$. 

\begin{theoremalpha}\label{T:thmC}(see Theorems \ref{T:Pic-rig} and \ref{T:Cl-Pic})
Assume that $g\geq 1$. Let $G$ be a reductive group and fix $\delta\in \pi_1(G)$. 
\begin{enumerate}
\item The relative Picard group of the rigidification 
$\bgr{G}^{\delta}:= \bg{G}^{\delta}\fatslash \scr Z(G)$ sits into the following exact sequence 
\begin{equation}\label{E:thmC1}
0 \to \Lambda^*(G^{\ab})\otimes H_{g,n}  \xrightarrow{\ov{j_G^{\delta}}} \RPic(\bgr{G}^{\delta})\xrightarrow{\ov{\gamma_G^{\delta}}} \NS(\bgr{G}^{\delta}),
\end{equation}
and the image of $\ov{\gamma_G^{\delta}}$ is equal to 
\begin{equation}\label{E:thmC2}
 \Im(\ov{\gamma_G^{\delta}})=
 \begin{cases} 
\NS(\bgr{G}^{\delta}) & \text{ if } n\geq 1,\\
  \left\{b\in \NS(\bgr{G}^{\delta})\: : \: 
 \begin{aligned}
 & 2g-2\vert b(\delta\otimes x)+(g-1)b(x\otimes x)\\
& \text{ for any } x\in \Lambda(T_G)
 \end{aligned}
 \right\} & \text{ if } n=0. 
\end{cases} 
 \end{equation}
 \item Assume that there exists an adequate moduli space $\pi: \bgr{G}^{\delta, ss}\to M_{G,g,n}^{\delta, ss}$ (e.g. ${\rm char}(k)=0$ and $n=0$ or $n>2g+2$).
 Suppose that  $g+n\geq 3$ (i.e. $\mathcal M_{g,n}$ is generically a variety) and that one of the following conditions hold
\begin{enumerate}[(i)]
	\item $G$ is a torus,
	\item $G$ is not a torus, $\operatorname{char}(k)>0$, $g\geq 4$,
	\item $G$ is not a torus, $\operatorname{char}(k)=0$, $g\geq 2$, with the exception of the case $g=2$ and $G$ having a non-trivial homomorphism into $PGL_2$.
\end{enumerate}
Then there are isomorphisms
 $$
\Cl(M_{G,g,n}^{\delta, ss}) \xrightarrow{\cong}  \Pic(\bgr{G}^{\delta,ss})\xrightarrow[\res^{-1}]{\cong} \Pic(\bgr{G}^{\delta}),
 $$
 where $\res$ is the restriction homomorphism and first isomorphism is obtained by pull-back along $\pi$. 
 \end{enumerate}
\end{theoremalpha}
The group $\NS(\bgr{G}^{\delta})$ is introduced in Definition \ref{D:NS-rig} and further studied in Proposition \ref{P:seqNSrig}. From the above Theorem \ref{T:thmC}, one easily recovers \cite[Thm. B(i) and 1.5]{MV} (see also \cite{Kou91}) if $G=\Gm$, $n=0$, $g\geq 2$ and $\mathrm{char}(k)=0$; \cite[Thm. B(i) and Thm. A.2]{Fri18} (see also \cite{Kou93}) if $G=\GL_r$, $n=0$, $g\geq 2$ and $\mathrm{char}(k)=0$. 

The final result of the paper deals with the triviality of the $\scr Z(G)$-gerbe $\nu_G^{\delta}$.
From the Leray spectral sequence associated to the $\scr Z(G)$-gerbe $\nu_G^{\delta}$ and the multiplicative group $\Gm$, we get the  exact sequence 
\begin{equation}\label{E:LerayINT}
 \Pic(\bgr{G}^{\delta})\stackrel{(\nu_G^{\delta})^*}{\hookrightarrow} \Pic(\bg{G}^{\delta}) \xrightarrow{\w_G^{\delta}} \Lambda^*(\scr Z(G))  \xrightarrow{\obs_G^{\delta}} H^2(\bgr{G}^{\delta}, \Gm)\xrightarrow{(\nu_G^{\delta})^*}  H^2(\bg{G}^{\delta}, \Gm).
\end{equation}
For a geometric interpretation of the \emph{weight homomorphism} $\w_G^{\delta}$ and of the \emph{obstruction homomorphism} $\obs_G^{\delta}$, see \S \ref{S:rig}.
In particular, $\Im(\obs_G^{\delta})\cong \coker(\w_G^{\delta})$ is an obstruction to the triviality of the $\scr Z(G)$-gerbe $\nu_G^{\delta}$.

\begin{theoremalpha}\label{T:thmD}(see Theorem \ref{T:coker-om})
Assume that $g\geq 1$. Let $G$ be a reductive group and fix $\delta\in \pi_1(G)$.
 \begin{enumerate}
\item \label{T:thmD1} If $n>0$ then 
 \begin{equation*}
\coker(\w_G^{\delta})\xrightarrow{\cong} \coker(\ev_{\scr D(G)}^{\delta}).
 \end{equation*}
\item \label{T:thmD2} If $n=0$ then the cokernel of $\w_G^{\delta}$ sits in an exact sequence 
 \begin{equation*}
 0\to \coker(\ov{\gamma_G^{\delta}})\xrightarrow{\partial_G^{\delta}} \Hom\left(\Lambda(G^{\ab}), \frac{\bbZ}{(2g-2)\bbZ}\right)\xrightarrow{\wt{\Lambda_{\ab}^*}} \coker(\w_G^{\delta})\xrightarrow{ \wt{\Lambda_{\scr D}^*}} \coker(\ev_{\scr D(G)}^{\delta})\to 0.
  \end{equation*}
\end{enumerate}
\end{theoremalpha}
For a definition of the homomorphism $\ev_{\scr D(G)}^{\delta}$, see \S\ref{Sec:int-forms}. 
In Section \ref{Sec:Examples}, we compute $ \coker(\ev_{\scr D(G)}^{\delta})$ for all reductive groups $G$ such that $\g^{\ss}$ is a simple Lie algebra, together with its quotient $ \coker(\wt \ev_{\scr D(G)}^{\delta})$ (see Definition/Lemma \ref{D:evGtilde}\eqref{D:evGtilde2}), which is an obstruction to the triviality of the $\scr Z(G)$-gerbe 
\begin{equation*}
\nu_G^{\delta}(C): \Bun_G^{\delta}(C)\to \Bunr_G^{\delta}(C):=\Bun_G^{\delta}(C)\fatslash \scr Z(G),
\end{equation*}
for any $(C,p_1,\ldots, p_n)\in \Mg(k)$, as shown by Biswas-Hoffmann \cite{BH12}, see Theorem \ref{T:weightC}.

From Theorem \ref{T:thmD}, one easily recovers \cite[Thm. 6.4]{MV} if $G=\Gm$, $n=0$, $g\geq 3$ and $\mathrm{char}(k)=0$; \cite[Cor. 3.3.2(i)]{Fri18}  if $G=\GL_r$, $n=0$, $g\geq 3$ and $\mathrm{char}(k)=0$; \cite[Thm. B(i)]{FPbr} if $G=\GL_r$, $g\geq 3$ and $\mathrm{char}(k)=0$. 

The computation of the image of the obstruction homomoprhism $\obs_G^{\delta}$ carried over in Theorem \ref{T:thmD} will be a crucial ingredient in our upcoming work \cite{FV3}, where we will compute the (cohomological) Brauer groups of $\bg{G}^{\delta}$, $\bgr{G}^{\delta}$ and $M_{G,g,n}^{\delta, ss}$, extending the work of Pirisi and the first author \cite{FPbr} from   $G=\GL_r$ to an arbitrary reductive group $G$. This is also closely related to the works of Biswas-Hogadi \cite{BHg} and Biswas-Holla \cite{BH13}, where the (cohomological) Brauer group of $\Bun_G(C)$ (and of its good moduli space) has been computed for a fixed curve $C$ and a complex semisimple group $G$.

\subsection*{Notations}

\begin{notations}
We denote by $k=\ov k$ an algebraically closed field of arbitrary characteristic. All the schemes and algebraic stacks that we will appear in this paper will be locally of finite type over $k$
(hence locally Noetherian). 
\end{notations}

\begin{notations}\label{N:famcur}

A \emph{curve} is a connected, smooth and projective scheme of dimension one over $k$.  The genus of a curve $C$ is $g(C):=\dim H^0(C,\omega_C)$.

A \emph{family of curves} $\pi: \cC\to S$ is a proper and flat morphism of stacks whose geometric fibers are curves. If all the geometric fibers of $\pi$ have the same genus $g$, then we say that $\pi:\cC\to S$ is a family of curves of genus $g$ (or a family of curves with relative genus $g$) and we set $g(\cC/S):=g$. We will denote by $\omega_{\pi}$ the relative canonical line bundle of $\pi$. Note that any family of curves $\pi:\cC\to S$ with $S$ connected is a family of genus $g$ curves for some $g\geq 0$.

\end{notations}

\begin{notations}
Given two integers $g,n\geq 0$, we will denote by $\Mg$ the stack (over $k$) whose fiber over a scheme $S$ is the groupoid of families $(\pi:\cC\to S,\un\sigma=\{\sigma_1,\ldots, \sigma_n\})$ of  \emph{$n$-pointed curves of genus $g$} over $S$, i.e. $\pi:\cC\to S$ is a family of curves of  genus $g$ and $\{\sigma_1,\ldots, \sigma_n\}$ are (ordered) sections of $\pi$ that are fiberwise disjoint.

It is well known that the stack $\Mg$ is an irreducible algebraic stack, smooth and separated over $k$, and of dimension $3g-3+n$. 
Moreover,  $\Mg$ is a DM(=Deligne-Mumford) stack if and only if $3g-3+n>0$.

We will denote by $(\pi_{g,n}=\pi:\Cg\to \Mg, \un \sigma)$ the universal $n$-pointed curve over $\Mg$. 

\end{notations}

\begin{notations}
A linear algebraic group over $k$ is a group scheme of finite type over $k$ that can be realized as a closed algebraic subgroup of $\GL_n$, or equivalently it is an affine group scheme of finite type over $k$. 
We will be dealing almost always with linear algebraic groups that are smooth (which is always the case if $\mathrm{char}(k)=0$) and connected. 

Given a linear  algebraic group $G$,  a principal $G$-bundle over an algebraic  stack $S$ is a $G$-torsor over $S$, where $G$ acts on the right. 
\end{notations}

\begin{notations}In the paper, we introduce several groups and morphisms. To help the reader, we make a table of the main objects together with a reference to their definitions.

\newpage

	\begin{tabular}{|l|l|}
\textbf{Symbol} & \textbf{Definition} \\
\hline
$\scr D(G)$,  $\scr R(G)$, $G^{\rm{ss}}$, $G^{\ab}$, $G^{\sc}$, $G^{\ad}$ & \eqref{E:cross}, \eqref{E:tower-ss} \\
\hline
Weil group  $\scr W_G$, Fundamental group $\pi_1(G)$, Center  $\scr Z(G)$ & \eqref{E:Weil}, \eqref{E:seq-pi1}, \eqref{E:cent-cross} \\
\hline
Maximal tori $T_G$, $T_{\scr D(G)}$,  $T_{G^{\rm{ss}}}$, $T_{G^{\sc}}$, $T_{G^{\ad}}$ & \eqref{E:crosstori} \\
\hline
Cocharacter lattices $\Lambda(-)$ and character lattices $\Lambda^*(-)$ & \eqref{E:tori-cocar}, \eqref{E:tori-car} \\
\hline
$\Lambda(T_G)_{\bbQ}^{\ab}$, $\Lambda(T_G)_{\bbQ}^{\ss}$, $\Lambda^*(T_G)_{\bbQ}^{\ab}$, $\Lambda^*(T_G)_{\bbQ}^{\ss}$& \eqref{E:dec-L}, \eqref{E:dec-L*}\\
\hline
$\Bil^{s}(-)$ and $\Bil^{s,\ev}(-)$ & \eqref{D:intform}\\
\hline
$\Bil^{s, \ev}(\Lambda(T_G))^{\scr W_G}\subseteq  \Bil^{s}(\Lambda(T_G))^{\scr W_G}$ & Section \ref{Sec:int-forms}\\
\hline
$\Bil^{s,(\ev)}(\Lambda(G^{\ab})) \xrightarrow{B_{\ab}^*}   \Bil^{s,(\ev)}(\Lambda(T_G))^{\scr W_G}$ &  Proposition \ref{P:forms-LTG} \\
\hline
$ \Bil^{s,(\ev)}(\Lambda(T_G))^{\scr W_G}  \xrightarrow{\res_{\scr D}} \Bil^{s,(\ev)}(\Lambda(T_{\scr D(G)})\vert \Lambda(T_{G^{\ss}}))^{\scr W_G} $& Proposition \ref{P:forms-LTG} \\
\hline
$\Bil^{s}(\Lambda(G^{\ab})) \xrightarrow{B_{\ab}^*}   \Bil^{s,\scr D-\ev}(\Lambda(T_G))^{\scr W_G}  $ &  Corollary \ref{C:forms-LTG}\\
\hline
$   \Bil^{s,\scr D-\ev}(\Lambda(T_G))^{\scr W_G}  \xrightarrow{\res_{\scr D}} \Bil^{s,\ev}(\Lambda(T_{\scr D(G)})\vert \Lambda(T_{G^{\ss}}))^{\scr W_G} $& Corollary \ref{C:forms-LTG} \\
\hline
$\ev_{G}^{\delta}:\Bil^{s,\scr D-\ev}(\Lambda(T_G))^{\scr W_G}\to \Lambda^*(T_G)/\Lambda^*(T_{G^{\ad}})$ & Definition/Lemma \ref{D:evG}\\
\hline
 $\Bil^{s,\sc-\ev}(\Lambda(T_{\scr D(G)})\vert \Lambda(T_{G^{\ss}}))^{\scr W_G}$ &  Definition/Lemma \ref{D:evGtilde}\\
 \hline
$\Bil^{s,\ev}(\Lambda(T_{\scr D(G)})\vert \Lambda(T_{G^{\ss}}))^{\scr W_G}\stackrel{r_G}{\hookrightarrow} \Bil^{s,\sc-\ev}(\Lambda(T_{\scr D(G)})\vert \Lambda(T_{G^{\ss}}))^{\scr W_G}$ & Definition/Lemma \ref{D:evGtilde}\eqref{D:evGtilde1}\\
\hline
 $\widetilde{\ev}_{\scr D(G)}^{\delta}:\Bil^{s,\sc-\ev}(\Lambda(T_{\scr D(G)})\vert \Lambda(T_{G^{\ss}}))^{\scr W_G}\to \Lambda^*(T_{\scr D(G)})/\Lambda^*(T_{G^{\ad}})$ & Definition/Lemma \ref{D:evGtilde}\eqref{D:evGtilde2}\\
 \hline
$\tau_G^{\delta}:\Sym^2(\Lambda^*(T_G))^{\scr W_G}\cong \Bil^{s,\ev}(\Lambda(T_G))^{\scr W_G}\to \RPic(\bg{G}^{\delta})$ & Theorem \ref{T:Pic-red} \\
\hline
$\theta_G^{\delta}:\RPic(\bg{G}^{\delta})\to\Bil^{s,\ev}(\Lambda(T_{\scr D(G)})\vert \Lambda(T_{G^{\ss}}))^{\scr W_G}$ & Corollary \ref{C:Pic-red}\\
\hline
$\gamma_G^{\delta}:\RPic(\bg{G}^{\delta})\to \Bil^{s,\scr D-\ev }(\Lambda(T_G))^{\scr W_G}$ & Definition/Lemma \ref{D:gammaG}\\
\hline
$\wh H_{g,n}$ and $i_G^{\delta}:\Lambda^*(G^{\ab})\otimes \wh H_{g,n}\to\RPic(\bg{G}^{\delta})$ & Theorem \ref{T:gG}\\
\hline
$\omega_G^{\delta}:\RPic(\bg{G}^{\delta})\to \Lambda^*(T_G)/\Lambda^*(T_{G^{\ad}})$ & Definition/Lemma \ref{D:wtG}\\
\hline
$\NS(\bg{G}^{\delta})\subseteq \Lambda^*(T_G)/\Lambda^*(T_{G^{\ad}})\oplus\Bil^{s,\scr D-\ev}(\Lambda(T_G))^{\scr W_G}$ & Definition \ref{D:NS-BunG}\\
\hline
$\phi^{*,\NS}$ & Definition/Lemma \ref{D:funcNS} \\
\hline
$\res_G^{\NS}: \NS(\bg{G}^{\delta})\to  \Bil^{s,\scr D-\ev}(\Lambda(T_G))^{\scr W_G}$ & Proposition \ref{P:seqNS}  \\
 \hline
$H_{g,n}$ and  $j_G^{\delta}:\Lambda^*(G^{\ab})\otimes  H_{g,n}\to\RPic(\bg{G}^{\delta})$ & Theorem \ref{T:oG+gG}\\
 \hline
 $\Hom(\pi_1(G),J_C(k))\stackrel{j_G^{\delta}(C)}{\hookrightarrow} \Pic(\Bun_G^{\delta}(C))\stackrel{c_G^{\delta}(C)}{\twoheadrightarrow} \NS(\Bun_G^{\delta}(C))$ & Theorem \ref{T:Pic-BunGC}\\
 \hline
 $p:\NS(\Bun_G^{\delta}(C))\to \Bil^{s,\ev}(\Lambda(T_{G^{sc}}))^{\scr W_G}$ & Proposition \ref{P:NS-BunGC}\\
 \hline
$\res_G^{\delta}(C)$, $\res_G^{\delta}(C)^o$, $\res_G^{\delta}(C)^{\NS}$,  $\ov{\res_G^{\delta}(C)}$ &  Theorem \ref{T:resPic}, \eqref{E:reshomNS}\\
\hline
$ \Pic(\bg{G}^{\delta}) \xrightarrow{\w_G^{\delta}} \Lambda^*(\scr Z(G))  \xrightarrow{\obs_G^{\delta}} H^2(\bgr{G}^{\delta}, \Gm)$ & \eqref{E:seqLeray} \\
 \hline
$ \Pic(\Bun_G^{\delta}(C)) \xrightarrow{\w_G^{\delta}(C)} \Lambda^*(\scr Z(G))  \xrightarrow{\obs_G^{\delta}(C)} H^2(\Bunr_G^{\delta}(C), \Gm)$ & \eqref{E:seqLerayC} \\
\hline
$\w_G^{\delta}(C):\Pic(\Bun_G^{\delta}(C))\xrightarrow{c_G^{\delta}(C)} \NS(\Bun_G^{\delta}(C))  \xrightarrow{\ov{\w}_G^{\delta}(C)} \frac{\Lambda^*(T_G)}{\Lambda^*(T_{G^{\ad}})}$
& Theorem \ref{T:weightC}\\
\hline
$\NS(\bgr{G}^{\delta})\subset \Bil^{s,\scr D-\ev}(\Lambda(T_G))^{\scr W_G}$ &  Definition \ref{D:NS-rig}\\
\hline
$\NS(\bgr{G}^{\delta})\stackrel{\nu_G^{\delta, \NS}}{\hookrightarrow} \NS(\bg{G}^{\delta}) \xrightarrow{\omega_G^{\delta,\NS}} \Lambda^*(T_G)/\Lambda^*(T_{G^{\ad}})$ & Proposition \ref{P:seqNSrig}\\
\hline
$\Lambda^*(G^{\ab})\otimes H_{g,n}  \stackrel{\ov{j_G^{\delta}}}{\hookrightarrow} \RPic(\bgr{G}^{\delta})\xrightarrow{\ov{\gamma_G^{\delta}}} \NS(\bgr{G}^{\delta})$ &  Theorem \ref{T:Pic-rig}\\
 \hline
$ \partial_G^{\delta}: \coker(\ov{\gamma_G^{\delta}})  \longrightarrow \Hom\left(\Lambda(G^{\ab}), \frac{\bbZ}{(2g-2)\bbZ}\right)$ & Theorem \ref{T:coker-om}\eqref{T:coker-om2}\\
\hline
\end{tabular}

\end{notations}

\section{Preliminaries}\label{S:prel}

\subsection{Reductive groups}\label{red-grps}

In this subsection, we will collect some result on the structure of reductive groups, that will be used in what follows. 

A \textbf{reductive group} (over $k=\ov k$) is a smooth and connected linear algebraic group (over $k$) which does not contain non-trivial connected normal unipotent algebraic subgroups. To any reductive group $G$, we can associate a cross-like diagram of reductive groups 
\begin{equation}\label{E:cross}
\xymatrix{
& \scr D(G) \ar@{^{(}->}[d]^{\scr D} \ar@{->>}[dr] & \\
\scr R(G) \ar@{^{(}->}[r]_{\scr R}  \ar@{->>}[dr] & G \ar@{->>}[d]^{\ab}  \ar@{->>}[r]_{\ss}  & G^{\rm{ss}} \\
& G^{\ab} & \\
}
\end{equation}
where 
	\begin{itemize}
		\item   $\mathscr D(G):=[G,G]$ is the derived subgroup of $G$;
		\item   $G^{\mathrm{ab}}:=G/\mathscr D(G)$ is called the abelianization of $G$;
		\item   $\mathscr R(G)$ is the radical subgroup of $G$, which is equal (since $G$ is reductive) to the connected component  $\scr Z(G)^o$ of the center $\scr Z(G)$;
		\item   $G^{\rm{ss}}:=G/\mathscr R(G)$ is called the semisimplification of $G$.
	\end{itemize}		
In the above diagram, the horizontal and vertical lines are short exact sequences of reductive groups, the morphisms $\scr D(G)\twoheadrightarrow G^{\ss}$ and $\scr R(G)\twoheadrightarrow G^{\ab}$ are central isogenies of, respectively, semisimple groups and tori with the same kernel which is equal to the finite   multiplicative (algebraic) group
\begin{equation}\label{E:mu}
\mu:=\scr D(G)\cap \scr R(G)\subset G.
\end{equation}

Since the two semisimple groups $\scr D(G)$ and $G^{\ss}$ are isogenous, they share the same simply-connected cover, that we will denote by $G^{\sc}$, and the same adjoint quotient, that we will denote by $G^{\ad}$.  Hence we have the following tower of central isogenies of semisimple groups
\begin{equation}\label{E:tower-ss}
G^{\sc}\twoheadrightarrow \scr D(G) \twoheadrightarrow G^{\ss} \twoheadrightarrow G^{\ad}.
\end{equation}

The Lie algebra $\g$ of $G$ splits as 
\begin{equation}\label{E:Lie-split}
\g=\g^{\ab}\oplus \g^{\ss},
\end{equation}
where $\g^{\ab}$ is the abelian Lie algebra of the tori $\scr R(G)$ and $G^{\ab}$, whose dimension is called the abelian rank of $G$, and $\g^{\ss}$ is the semisimple Lie algebra of each of the  semisimple groups in \eqref{E:tower-ss}, whose rank is called the semisimple rank of $G$. The semisimple Lie algebra $\g^{\ss}$ decomposes as a direct sum of simple Lie algebras of classical type (i.e. type $A_n$, $B_n$, $C_n$, $D_n$, $E_6$, $E_7$, $E_8$, $F_4$ or  $G_2$). If $G$ is a semisimple group such that its Lie algebra $\g=\g^{ss}$ is simple, then $G$ is said to be almost-simple.

\begin{rmk}\label{R:functcross}
It follows from the universal property of the maximal abelian quotient $G^{\ab}$ and from the universal property of the universal cover $G^{\sc}$ that the morphisms
$$G^{\sc}\twoheadrightarrow \scr D(G)\stackrel{\scr D}{\hookrightarrow} G\stackrel{\ab}{\twoheadrightarrow} G^{\ab},$$
are covariantly functorial with respect to  homomorphisms of reductive groups. 

On the other hand, the morphisms
$$
\scr R(G)\stackrel{\scr R}{\hookrightarrow} G\stackrel{\ss}{\twoheadrightarrow} G^{\ss}\twoheadrightarrow G^{\ad}
$$
are not functorial with respect to arbitrary homomorphisms of reductive groups, e.g. the inclusion of a maximal torus $T\hookrightarrow G$ does not factor, in general,  through $\scr R(G)$  or, equivalently, does not map to zero in $G^{\ss}$ or $G^{\ad}$. 
\end{rmk}

Recall now that all maximal tori of $G$ are conjugate and let us fix one such \emph{maximal torus}, that we call $T_G$. We will denote by $B_G$ a Borel subgroup of $G$ that contains $T_G$ and by $\scr N(T_G)$ the normalizer of $T_G$  in $G$, so that 
\begin{equation}\label{E:Weil} 
\scr W_G:=\mathscr N(T_G)/T_G
\end{equation}
is the Weyl group of $G$.

The maximal torus $T_G$ induces compatible maximal tori of every semisimple group appearing in \eqref{E:tower-ss}, that we will call, respectively, $T_{G^{\sc}}$,  $T_{\scr D(G)}$, $T_{G^{\ss}}$ and $T_{G^{\ad}}$. These tori fit into the following commutative diagram:
\begin{equation}\label{E:crosstori}
\xymatrix{
T_{G^{\sc}} \ar@{->>}[dr] & & & \\
& T_{\scr D(G)} \ar@{^{(}->}[d] \ar@{->>}[dr] & & \\
\scr R(G) \ar@{^{(}->}[r]  \ar@{->>}[dr] & T_G \ar@{->>}[d]  \ar@{->>}[r]  & T_{G^{\rm{ss}}} \ar@{->>}[dr]  & \\
& G^{\ab} & & T_{G^{\ad}}\\
}
\end{equation}
where the horizontal and vertical lines are short exact sequences of tori, and the diagonal arrows are (central) isogenies of tori. Using the canonical realization \eqref{E:Weil} of the Weyl group  (and the similar ones for the semisimple groups in \eqref{E:tower-ss}), diagram \eqref{E:crosstori} induces canonical isomorphisms of Weyl groups 
\begin{equation}\label{E:iso-Weyl}
\scr W_{G^{\sc}}\cong \scr W_{\scr D(G)}\cong \scr W_G \cong \scr W_{G^{\ss}}\cong \scr W_{G^{\ad}}. 
\end{equation} 


By taking the \emph{cocharacter lattices} $\Lambda(-):=\text{Hom}(\Gm,-)$ of the tori in the diagram \eqref{E:crosstori}, we get the following $\scr W_G$-equivariant commutative diagram of lattices 
\begin{equation}\label{E:tori-cocar}
\xymatrix{
\Lambda(T_{G^{\sc}}) \ar@{^{(}->}[dr] & & & \\
& \Lambda(T_{\scr D(G)}) \ar@{^{(}->}[d]^{\Lambda_{\scr D}} \ar@{^{(}->}[dr] & & \\
\Lambda(\scr R(G)) \ar@{^{(}->}[r]_{\Lambda_{\scr R}}  \ar@{^{(}->}[dr] & \Lambda(T_G) \ar@{->>}[d]^{\Lambda_\ab}  \ar@{->>}[r]_{\Lambda_\ss}  & \Lambda(T_{G^{\rm{ss}}}) \ar@{^{(}->}[dr]  & \\
& \Lambda(G^{\ab}) & & \Lambda(T_{G^{\ad}}),\\
}
\end{equation}
where the horizontal and vertical lines are short exact sequences and  the diagonal arrows are finite index inclusions.
 The above diagram induces a canonical splitting 
\begin{equation}\label{E:dec-L}
\Lambda(T_G)_{\bbQ}=\Lambda(T_G)_{\bbQ}^{\ab}\oplus \Lambda(T_G)_{\bbQ}^{\ss},
\end{equation}
where $\Lambda(T_G)_{\bbQ}^{\ab}$ and $\Lambda(T_G)_{\bbQ}^{\ss}$ are the unique subgroups of $\Lambda(T_G)_{\bbQ}$ such that 
\begin{equation}\label{E:dec-Lbis}
\begin{aligned}
&\Lambda(\scr R(G))=\Lambda(T_G)\cap \Lambda(T_G)_{\bbQ}^{\ab} & \quad \text{ and } \quad & \Lambda(T_{\scr D(G)})=\Lambda(T_G)\cap \Lambda(T_G)_{\bbQ}^{\ss}, \\
&\Lambda(G^{\ab})=p_1(\Lambda(T_G))  & \quad  \text{ and } \quad & \Lambda(T_{G^{\ss}})=p_2(\Lambda(T_G)),
\end{aligned}
\end{equation}
where $p_1$ and $p_2$ are the two projections onto the two factors of \eqref{E:dec-L}.

	
In a similar way, by taking the \emph{character lattices} $\Lambda^*(-):=\text{Hom}(-,\Gm)$ (which we will interpret as spaces of integral functionals on $\Lambda(-)$), we get the following $\scr W_G$-equivariant commutative diagram of lattices 
\begin{equation}\label{E:tori-car}
\xymatrix{
\Lambda^*(T_{G^{\sc}})  & & & \\
& \Lambda^*(T_{\scr D(G)})  \ar@{_{(}->}[ul]& & \\
\Lambda^*(\scr R(G)) & \Lambda^*(T_G) \ar@{->>}[u]_{\Lambda^*_{\scr D}}  \ar@{->>}[l]^{\Lambda^*_{\scr R}} & \Lambda^*(T_{G^{\rm{ss}}}) \ar@{_{(}->}[ul]   \ar@{_{(}->}[l]^{\Lambda^*_\ss} & \\
& \Lambda^*(G^{\ab}) \ar@{_{(}->}[ul] \ar@{^{(}->}[u]_{\Lambda^*_\ab}& & \Lambda^*(T_{G^{\ad}}) \ar@{_{(}->}[ul],\\
}
\end{equation}
where the horizontal and vertical lines are short exact sequences and  the diagonal arrows are finite index inclusions.  The above diagram induces a canonical splitting 
\begin{equation}\label{E:dec-L*}
\Lambda^*(T_G)_{\bbQ}=\Lambda^*(T_G)_{\bbQ}^{\ab}\oplus \Lambda^*(T_G)_{\bbQ}^{\ss},
\end{equation}
where $\Lambda^*(T_G)_{\bbQ}^{\ab}$ and $\Lambda^*(T_G)_{\bbQ}^{\ss}$ are the unique subgroups of $\Lambda^*(T_G)_{\bbQ}$ such that 
\begin{equation}\label{E:dec-L*bis}
\begin{aligned}
&\Lambda^*(G^{\ab})=\Lambda^*(T_G)\cap \Lambda^*(T_G)_{\bbQ}^{\ab} & \quad \text{ and } \quad &\Lambda^*(T_{G^{\ss}})=\Lambda^*(T_G)\cap \Lambda^*(T_G)_{\bbQ}^{\ss}, \\
&\Lambda^*(\scr R(G))=p_1(\Lambda^*(T_G))  & \quad  \text{ and } \quad &    \Lambda^*(T_{\scr D(G)})=p_2(\Lambda^*(T_G)),
\end{aligned}
\end{equation}
where $p_1$ and $p_2$ are the two projections onto the two factors of \eqref{E:dec-L*}.

\begin{rmk}\label{R:lat-Lie}
There are natural identifications
$$
\begin{aligned}
& \Lambda(T_{G^{\sc}})\cong \Lambda_{\coroo}(\g^{\ss}) \quad \text{ and } \quad & \Lambda(T_{G^{\ad}})\cong \Lambda_{\cowei}(\g^{\ss}), \\
& \Lambda^*(T_{G^{\sc}})\cong \Lambda_{\wei}(\g^{\ss}) \quad \text{ and }\quad  & \Lambda^*(T_{G^{\ad}})\cong \Lambda_{\roo}(\g^{\ss}), 
\end{aligned}
$$
where $\Lambda_{\roo}(\g^{\ss})$ (resp. $\Lambda_{\coroo}(\g^{\ss}) $) is the lattice of roots (resp. coroots) of the  semisimple Lie algebra $\g^{\ss}$, and $\Lambda_{\wei}(\g^{\ss}) $ (resp. $\Lambda_{\cowei}(\g^{\ss})$) is the lattice of weights (resp. coweights) of $\g^{\ss}$. 
The two diagrams \eqref{E:tori-cocar} and \eqref{E:tori-car}, together with the root system of the semisimple Lie algebra $\g^{\ss}$, are equivalent to the root data of the reductive group $G$ (see \cite[\S 19]{Mil}), and hence they determine completely the reductive group $G$. 
\end{rmk}

\begin{rmk}\label{R:funclatt}
It follows from Remark \ref{R:functcross} that the morphism of lattices
\begin{equation}\label{E:funclatt}
 \Lambda(T_{G^{\sc}})\hookrightarrow \Lambda(T_{\scr D(G)}) \stackrel{\Lambda_{\scr D}}{\hookrightarrow} \Lambda(T_G) \stackrel{\Lambda_{\ab}}{\twoheadrightarrow} \Lambda(G^{\ab}) \left(\text{resp.  } 
  \Lambda^*(G^{\ab}) \stackrel{\Lambda^*_{\ab}}{\hookrightarrow}   \Lambda^*(T_G)  \stackrel{\Lambda^*_{\scr D}}{\twoheadrightarrow}\Lambda^*(T_{\scr D(G)})\hookrightarrow  \Lambda^*(T_{G^{\sc}})\right)
\end{equation}
are covariantly (resp. contravariantly) functorial with respect to  homomorphisms of reductive groups $\phi:G\to H$ provided that we choose (and this is always possible) the  maximal tori $T_G$ and $T_H$ of, respectively, $G$ and $H$ in such a way that $\phi(T_G)\subseteq T_H$. 

On the other hand, the morphisms of lattices appearing in \eqref{E:tori-cocar} and in \eqref{E:tori-car} and different from the ones in  \eqref{E:funclatt} are not functorial. 
\end{rmk}



The \textbf{fundamental group} $\pi_1(G)$ of $G$ is canonically isomorphic to $\Lambda(T_G)/\Lambda(T_{G^{\sc}})$ and it fits into the following (covariantly functorial) short exact sequence of finitely generated abelian groups 
\begin{equation}\label{E:seq-pi1}
\pi_1(\scr D(G)) =\frac{\Lambda(T_{\scr D(G)})}{\Lambda(T_{G^{\sc}})}\hookrightarrow \pi_1(G)=\frac{\Lambda(T_G)}{\Lambda(T_{G^{\sc}})}\twoheadrightarrow \pi_1(G^{\ab})=\Lambda(G^{\ab}),
\end{equation}
where the first term is  the torsion subgroup of $\pi_1(G)$ and the last term is  the torsion-free quotient of $\pi_1(G)$. 

Let us now describe the   \textbf{center} $\scr Z(G)$ of $G$, which is a multiplicative (algebraic) group. By taking the centers of the algebraic groups appearing in the cross-like diagram \eqref{E:cross}, we obtain the following (non functorial) cross-like diagram of multiplicative  groups 
\begin{equation}\label{E:cent-cross}
\xymatrix{
& \scr Z(\scr D(G)) \ar@{^{(}->}[d] \ar@{->>}[dr] & \\
\scr R(G)=\scr Z(G)^o \ar@{^{(}->}[r]  \ar@{->>}[dr] & \scr Z(G) \ar@{->>}[d]  \ar@{->>}[r]  & \scr Z(G^{\rm{ss}})=\pi_0(\scr Z(G)) \\
& G^{\ab} & \\
}
\end{equation}
where the horizontal and vertical lines are short exact sequences, the upper right diagonal morphism is an isogeny of finite multiplicative groups and the lower left diagonal arrow is an isogeny of tori. 
By passing to the character groups $\Lambda^*(-)=\Hom(-,\Gm)$ of the multiplicative  groups appearing in \eqref{E:cent-cross}, we get the following (non functorial) diagram of finitely generated abelian groups
\begin{equation}\label{E:mult-car}
\xymatrix{
&\Lambda^*(\scr Z(\scr D(G)))=\frac{\Lambda^*(T_{\scr D(G)})}{\Lambda^*(T_{G^{\ad}}))} &  \\
\Lambda^*(\scr R(G)) & \Lambda^*(\scr Z(G))=\frac{\Lambda^*(T_G)}{\Lambda^*(T_{G^{\ad}})} \ar@{->>}[u]^{\ov{\Lambda_{\scr D}^*}} \ar@{->>}[l]  & \Lambda^*(\scr Z(G^{\ss}))=\frac{\Lambda^*(T_{G^{\rm{ss}}})}{\Lambda^*(T_{G^{\ad}})} \ar@{_{(}->}[ul]   \ar@{_{(}->}[l]  \\
& \Lambda^*(G^{\ab}) \ar@{_{(}->}[ul] \ar@{^{(}->}[u]^{\ov{\Lambda_{\ab}^*}}& ,\\
}
\end{equation}
where the horizontal  line is the canonical decomposition of  $\Lambda^*(\scr Z(G))$ into its torsion subgroup and torsion-free quotient, the vertical line is exact,   the lower left diagonal arrow is a finite inclusion of lattices, the upper right diagonal arrow is an inclusion of finite abelian groups.

\subsection{Integral bilinear (even) symmetric forms on $\Lambda(T_G)$}\label{Sec:int-forms} 

In this subsection, we will prove some results on ($\scr W_G$-invariant)  integral bilinear (even) symmetric forms on the lattice $\Lambda(T_G)$. 

Given a lattice $\Lambda$ of rank $r$ (i.e. $\Lambda\cong \bbZ^r$), we denote the lattice of integral bilinear (resp. even) symmetric forms on $\Lambda$ by
\begin{equation}\label{D:intform}
\begin{aligned}
& \Bil^s(\Lambda):=\left\{b:\Lambda\otimes\Lambda\to \bbZ \quad \text{ such that } b \: \text{ is symmetric}\right\},\\
& \Bil^{s,\ev}(\Lambda):=\{b\in \Bil^s(\Lambda): \: b(x,x) \: \text{ is even for any } x\in \Lambda\}.
\end{aligned}
\end{equation}
Given $b \in \Bil^{s}(\Lambda)$, we will denote by $b^{\bbQ}: \Lambda_{\bbQ}\otimes \Lambda_{\bbQ}\to \bbQ$ the rational extension of $b$ to the $\bbQ$-vector space $\Lambda_{\bbQ}:=\Lambda\otimes_{\bbZ}\bbQ$.  

The above lattices \eqref{D:intform}, which  are contravariantly functorial with respect to morphisms of lattices,  can be described in terms of the  dual lattice $\Lambda^*:=\Hom(\Lambda, \bbZ)$ in the following way. Consider the following lattices (of rank $\binom{r+1}{2}$)
\begin{equation}\label{D:dual-latt}
 (\Lambda^*\otimes \Lambda^*)^s\subset \Lambda^*\otimes \Lambda^* \quad \text{ and } \quad \Sym^2(\Lambda^*):=\frac{\Lambda^*\otimes \Lambda^*}{\langle \chi\otimes \mu-\mu\otimes \chi\rangle},
\end{equation}
where  $(\Lambda^*\otimes \Lambda^*)^s$ is the subspace of symmetric tensors of  $\Lambda^*\otimes \Lambda^*$, i.e. tensors that are invariant under the involution $\chi\otimes \mu \mapsto \mu\otimes \chi$.  We will denote the elements of $\Sym^2(\Lambda^*)$ by $\chi\cdot \mu:=[\chi\otimes \mu]$. 

The lattices in \eqref{D:intform} are isomorphic to the ones in \eqref{D:dual-latt} via the following isomorphisms
\begin{equation}\label{E:iso-b}
\begin{aligned}
(\Lambda^*\otimes \Lambda^*)^s& \stackrel{\cong}{\longrightarrow}  \Bil^s(\Lambda) \\
\chi\otimes \mu & \mapsto  (\chi\otimes \mu)(x\otimes y):=\chi(x)\mu(y),
\end{aligned}
\text{ and } \quad 
\begin{aligned}
\Sym^2 \Lambda^* & \stackrel{\cong}{\longrightarrow}  \Bil^{s,\ev}(\Lambda)\\
\chi\cdot\mu& \mapsto (\chi\cdot \mu)(x\otimes y):=\chi(x)\mu(y)+\mu(x)\chi(y).
\end{aligned}
\end{equation}
In terms of the isomorphisms \eqref{E:iso-b}, the inclusion $\Bil^{s,\ev}(\Lambda)\subset \Bil^s(\Lambda)$ corresponds to the injective morphism 
\begin{equation}
\begin{aligned}
\psi: \Sym^2(\Lambda^*) & \hookrightarrow (\Lambda^*\otimes \Lambda^*)^s, \\
\chi\cdot \mu & \mapsto \chi\otimes \mu+\mu\otimes \chi.
\end{aligned}
\end{equation}
If we fix a basis $\{\chi_i\}_{i=1}^r$ of $\Lambda^*$, then 
\begin{equation}\label{E:basis-form}
\begin{aligned}
& \left\{\{\chi_i\otimes \chi_i\}_i\cup\{\chi_i\otimes \chi_j+\chi_j\otimes \chi_i\}_{i<j}\right\} \text{ is a basis of } (\Lambda^*\otimes \Lambda^*)^s, \\
& \{\chi_i\cdot \chi_j\}_{i\leq j} \text{ is a basis of }  \Sym^2(\Lambda^*).
\end{aligned}
\end{equation}
 Using the above basis, it follows that 
\begin{equation}\label{E:cokerpsi}
\coker(\psi)=(\bbZ/2\bbZ)^r.
\end{equation}

We now come back to the setting of \S\ref{red-grps}. Let $G$ be a reductive  group with maximal torus $T_G\subset G$  and consider the natural action of the Weyl group  $\scr W_G$ on $\Lambda(T_G)$ and on $\Lambda^*(T_G)$. 
We now want to describe the lattices 
$$\Bil^s(\Lambda(T_G))^{\scr W_G}\cong   \left((\Lambda^*(T_G)\otimes \Lambda^*(T_G))^s\right)^{\scr W_G}  \quad \text{ and } \quad \Bil^{s,\ev}(\Lambda(T_G))^{\scr W_G}\cong \Sym^2(\Lambda^*(T_G))^{\scr W_G} $$
of $\scr W_G$-invariant integral bilinear (resp. even) symmetric forms on $\Lambda(T_G)$.

\begin{prop}\label{P:forms-LTG}
With the above notation, we have an exact sequence of lattices
\begin{equation}\label{E:seq-LTG}
\begin{aligned}
& 0\to \Bil^{s,(\ev)}(\Lambda(G^{\ab})) \xrightarrow{B_{\ab}^*}   \Bil^{s,(\ev)}(\Lambda(T_G))^{\scr W_G}  \xrightarrow{\res_{\scr D}} \Bil^{s,(\ev)}(\Lambda(T_{\scr D(G)})\vert \Lambda(T_{G^{\ss}}))^{\scr W_G}\to 0,\\
& \hspace{3cm} b(-\otimes -)  \mapsto b(\Lambda_{\ab}(-)\otimes \Lambda_{\ab}(-))  \\
 &  \hspace{6cm} b  \mapsto b_{|\Lambda(T_{\scr D(G)})\otimes \Lambda(T_{\scr D(G)})}  \\
\end{aligned}
\end{equation} 
where $\Lambda_{\ab}:\Lambda(T_G) \to \Lambda(G^{\ab})$ is the homomorphism defined in \eqref{E:tori-cocar} and 
$$
\Bil^{s,(\ev)}(\Lambda(T_{\scr D(G)})\vert \Lambda(T_{G^{\ss}}))^{\scr W_G}:=\left\{b\in \Bil^{s,(\ev)}(\Lambda(T_{\scr D(G)}))^{\scr W_G}\: : \: b^{\bbQ}_{|\Lambda(T_{\scr D(G)})\otimes \Lambda(T_{G^{\ss}})} \: \text{ is integral.} \right\} 
$$
Moreover, the exact sequence \eqref{E:seq-LTG} is contravariant with respect to homomorphisms of reductive groups $\phi:H\to G$ such that $\phi(T_H)\subseteq T_G$. 
\end{prop}
The above notation  $\Bil^{s,(\ev)}$ means that the result applies by putting $\Bil^s$ everywhere or by putting $\Bil^{s, \ev}$ everywhere. 

In the proof of the above proposition, we will use the following

\begin{lem}\label{L:inv-char}
Let $G$ be a reductive  group with maximal torus $T_G\subset G$  and consider the natural action of the Weyl group  $\scr W_G$ on $\Lambda(T_G)$ and on $\Lambda^*(T_G)$. Then we have isomorphisms 
$$\Lambda_{\scr R}:\Lambda(\scr R(G))\xrightarrow{\cong} \Lambda(T_G)^{\scr W_G} \quad \text{ and } \quad \Lambda_{\ab}^*:\Lambda^*(G^{\ab})\xrightarrow{\cong} \Lambda^*(T_G)^{\scr W_G}.$$
\end{lem}
\begin{proof}
The second isomorphism is proved in \cite[Lemma 2.1.1]{FV1}. The proof of the first isomorphism is similar. 
\end{proof}

\begin{proof}[Proof of Proposition \ref{P:forms-LTG}]
Clearly, the morphism $B_{\ab}^*$ is injective and $\res_{\scr D}\circ B_{\ab}^*=0$. 

In order to complete the proof, we will need the following

\un{Claim:} If $b\in  \Bil^{s}(\Lambda(T_G))^{\scr W_G}$ then $b_{|\Lambda(\scr R(G))\otimes \Lambda(T_{\scr D(G)})}\equiv 0$. In particular,  $b^{\bbQ}_{|\Lambda(T_G)_{\bbQ}^{\ab}\otimes \Lambda(T_G)_{\bbQ}^{\ss}}\equiv 0$.

Indeed, for any $x\in \Lambda(\scr R(G))$, the restriction $b(x\otimes -):\Lambda(T_G)\to \bbZ$ is $\scr W_G$-invariant since $b$ is $\scr W_G$-invariant and $x$ is  fixed by $\scr W_G$ (because the action of $\scr W_G$ is trivial on $\Lambda(\scr R(G))$).  Hence, Lemma \ref{L:inv-char} implies that $b(x\otimes -)$ is the pull-back of an integral functional on $\Lambda(G^{\ab})$, or in other words that $b(x\otimes -)_{|\Lambda(T_{\scr D(G)})}\equiv 0$. Since this is true for any $x\in \Lambda(\scr R(G))$, we get that  $b_{|\Lambda(\scr R(G))\otimes \Lambda(T_{\scr D(G)})}\equiv 0$. The last assertion follows from the fact that $\Lambda(\scr R(G))_{\bbQ}=\Lambda(T_G)_{\bbQ}^{\ab}$ and $\Lambda(T_{\scr D(G)})_{\bbQ}=\Lambda(T_G)_{\bbQ}^{\ss}$. 

\vspace{0.1cm}
We now go back to proof of the Proposition.  Let us first prove that the sequence \eqref{E:seq-LTG} is exact in the middle, i.e. $\ker(\res_{\scr D})\subseteq \Im(B_{\ab}^*)$. Consider a form $b\in \Bil^{s, (\ev)}(\Lambda(T_G))^{\scr W_G}$ such that $\res_{\scr D}(b)= 0$. 
This assumption, together with the above claim, implies that $b_{|\Lambda(T_{\scr D(G)})\otimes \Lambda(T_G)}\equiv 0$, which implies that $b$ is the pull-back of an integral bilinear (resp. even) symmetric  form on $\Lambda(G^{\ab})$.

Let us now prove that the morphism $\res_{\scr D}$ is well-defined, i.e. that $\res_{\scr D}(b)$ is integral on $\Lambda(T_{\scr D(G)})\otimes \Lambda(T_{G^{\ss}})$ for every $b\in  \Bil^{s,(\ev)}(\Lambda(T_G))^{\scr W_G}$. 
For any element $y\in \Lambda(T_{\scr D(G)})$, consider the integral functional $b(y\otimes -):\Lambda(T_G)\to \bbZ$. By the claim, we have that $b(y\otimes -)_{|\Lambda(T_{\scr R(G)})}\equiv 0$, which implies that $b(y\otimes -)$ is the 
restriction of an integral functional on $\Lambda(T_{G^{\ss}})$. Since this is true for any $y\in \Lambda(T_{\scr D(G)})$, we deduce that $\res_{\scr D}(b)$ is integral on $\Lambda(T_{\scr D(G)})\otimes \Lambda(T_{G^{\ss}})$. 

In order to show that the sequence \eqref{E:seq-LTG} is exact, it remains to prove that $\res_{\scr D}$ is surjective. Let $\wt b\in \Bil^{s,(\ev)}(\Lambda(T_{\scr D(G)})\vert \Lambda(T_{G^{\ss}}))^{\scr W_G}$. Since $\wt b^{\bbQ}$ is integral on   $\Lambda(T_{\scr D(G)})\otimes \Lambda(T_{G^{\ss}})$, by composing $\wt b^{\bbQ}$ with the surjection   $\Lambda(T_G)\twoheadrightarrow \Lambda(T_{G^{\ss}})$ we get a $\scr W_G$-invariant   (resp. even) symmetric integral form
\begin{equation}\label{E:formbhat}
\wh b:\Lambda(T_{\scr D(G)})\otimes \Lambda(T_G)+\Lambda(T_G)\otimes \Lambda(T_{\scr D(G)})\to \bbZ.
\end{equation}
Now consider the $\scr W_G$-equivariant  short exact sequence of lattices
\begin{equation}\label{E:seqBil}
\begin{aligned}
&  \Bil^{s,(\ev)}(\Lambda(G^{\ab})) \hookrightarrow  \Bil^{s,(\ev)}(\Lambda(T_G)) \twoheadrightarrow  \Hom^{s,(\ev)}(\Lambda(T_{\scr D(G)})\otimes \Lambda(T_G)+\Lambda(T_G)\otimes \Lambda(T_{\scr D(G)}),\bbZ) \\
& \hspace{6cm} b \mapsto b_{| \Lambda(T_{\scr D(G)})\otimes \Lambda(T_G)+\Lambda(T_G)\otimes \Lambda(T_{\scr D(G)})}
\end{aligned}
\end{equation}
where $\Hom^{s,(\ev)}(\Lambda(T_{\scr D(G)})\otimes \Lambda(T_G)+\Lambda(T_G)\otimes \Lambda(T_{\scr D(G)}),\bbZ)$ is the lattice of (resp. even) symmetric integral forms on $\Lambda(T_{\scr D(G)})\otimes \Lambda(T_G)+\Lambda(T_G)\otimes \Lambda(T_{\scr D(G)})\subseteq \Lambda(T_G)\otimes \Lambda(T_G)$. Since the action of $\scr W_G$ is trivial, we have that 
\begin{equation}\label{E:vanH1}
H^1(\scr W_G,  \Bil^{s,(\ev)}(\Lambda(G^{\ab}))=\Hom(\scr W_G, \Bil^{s,(\ev)}(\Lambda(G^{\ab})),
\end{equation}
and the last group is zero since $\scr W_G$ is a finite group and  $\Bil^{s,(\ev)}(\Lambda(G^{\ab}))$ is torsion-free. 
By taking the long exact sequence in $\scr W_G$-cohomology associated to the exact sequence \eqref{E:seqBil} and using the vanishing $H^1(\scr W_G,  \Bil^{s,(\ev)}(\Lambda(G^{\ab}))=0$, we get a surjection  
$$ \Bil^{s,(\ev)}(\Lambda(T_G))^{\scr W_G} \twoheadrightarrow  \Hom^{s,(\ev)}(\Lambda(T_{\scr D(G)})\otimes \Lambda(T_G)+\Lambda(T_G)\otimes \Lambda(T_{\scr D(G)}),\bbZ)^{\scr W_G}.$$
 Hence, the form $\wh b$ of \eqref{E:formbhat} is the restriction of a form $b\in \Bil^{s,(\ev)}(\Lambda(T_G))^{\scr W_G}$. By construction we have that $\res_{\scr D}(b)=\wt b$, which concludes the proof of the surjectivity of $\res_{\scr D}$. 
 
 Finally, the (contravariant) functoriality of the exact sequence \eqref{E:seq-LTG} follows from the fact that  $\Bil^{s,(\ev)}(\Lambda(T_G))^{\scr W_G}$ is functorial by \cite[Lemma 4.3.1]{BH10} (and the discussion that follows) while  $\Bil^{s,(\ev)}(\Lambda(G^{\ab}))$ and the morphism $B_{\ab}^*$ are functorial by Remark \ref{R:funclatt}.
\end{proof}

By combining the two exact sequences of Proposition \ref{P:forms-LTG}, we get a new exact sequence 

\begin{cor}\label{C:forms-LTG}
With the above notation, we have an exact sequence of lattices
\begin{equation}\label{E:seq-LTG2}
 0\to \Bil^{s}(\Lambda(G^{\ab})) \xrightarrow{B_{\ab}^*}   \Bil^{s,\scr D-\ev}(\Lambda(T_G))^{\scr W_G}  \xrightarrow{\res_{\scr D}} \Bil^{s,\ev}(\Lambda(T_{\scr D(G)})\vert \Lambda(T_{G^{\ss}}))^{\scr W_G}\to 0,
\end{equation} 
where 
$$
 \Bil^{s,\scr D-\ev}(\Lambda(T_G))^{\scr W_G}:=\left\{b\in  \Bil^{s}(\Lambda(T_G))^{\scr W_G}\: : \res_{\scr D}(b) \text{ is even.} \right\}
$$
Moreover, the exact sequence \eqref{E:seq-LTG2} is contravariant with respect to homomorphisms of reductive groups $\phi:H\to G$ such that $\phi(T_H)\subseteq T_G$. 
\end{cor} 

\begin{proof}Consider the commutative diagram of abelian groups
	$$
	\xymatrix{
			\Bil^{s}(\Lambda(G^{\ab})) \ar@{^{(}->}[r]\ar@{=}[d]&   \res_{\scr D}^{-1}(\Bil^{s, \ev}(\Lambda(T_{\scr D(G)})\vert \Lambda(T_{G^{\ss}}))^{\scr W_G}) \ar@{^{(}->}[d]\ar@{->>}[r]& \Bil^{s,\ev}(\Lambda(T_{\scr D(G)})\vert \Lambda(T_{G^{\ss}}))^{\scr W_G}\ar@{^{(}->}[d]\\
			\Bil^{s}(\Lambda(G^{\ab})) \ar@{^{(}->}[r]^{B_{\ab}^*}&   \Bil^{s}(\Lambda(T_{G}))^{\scr W_G} \ar@{->>}[r]^{\res_{\scr D}}& \Bil^{s}(\Lambda(T_{\scr D(G)})\vert \Lambda(T_{G^{\ss}}))^{\scr W_G}
	}
	$$
where the rows are exact and the columns are the obvious inclusions. The bottom row is the non-even version the exact sequence \eqref{E:seq-LTG}. By definition, we have that
 $$
 \res_{\scr D}^{-1}(\Bil^{s, \ev}(\Lambda(T_{\scr D(G)})\vert \Lambda(T_{G^{\ss}}))^{\scr W_G})=\Bil^{s,\scr D-\ev}(\Lambda(T_G))^{\scr W_G},
 $$
and hence the top row is the required sequence \eqref{E:seq-LTG2}.
\end{proof}

\begin{cor}\label{C:rank-Bil}
The ranks of  $\Bil^{s,(\ev)}(\Lambda(T_G))^{\scr W_G}$ and of $ \Bil^{s,\scr D-\ev}(\Lambda(T_G))^{\scr W_G}$ are equal to 
$$\binom{\dim G^{\ab}+1}{2}+|\left\{\text{simple factors of } \g^{\ss} \right\}|. $$
\end{cor}
\begin{proof}
Since $\Bil^{s,\ev}(\Lambda(T_G))^{\scr W_G}\subseteq  \Bil^{s,\scr D-\ev}(\Lambda(T_G))^{\scr W_G}\subseteq \Bil^{s}(\Lambda(T_G))^{\scr W_G}$ are finite index inclusions, it is enough to prove the result for 
$\Bil^{s,\ev}(\Lambda(T_G))^{\scr W_G}$. Proposition \ref{P:forms-LTG} implies that 
$$\rk \Bil^{s,\ev}(\Lambda(T_G))^{\scr W_G}=\rk \Bil^{s,\ev}(\Lambda(G^{\ab}))+\rk \Bil^{s,\ev}(\Lambda(T_{\scr D(G)})\vert \Lambda(T_{G^{\ss}}))^{\scr W_G}=$$
$$=\binom{\dim G^{\ab}+1}{2}+\rk  \Bil^{s,\ev}(\Lambda(T_{\scr D(G)})\vert \Lambda(T_{G^{\ss}}))^{\scr W_G}.
$$ 
We conclude observing that we have finite index inclusions 
$$
\Bil^{s,\ev}(\Lambda(T_{\scr D(G)})\vert \Lambda(T_{G^{\ss}}))^{\scr W_G}\subseteq \Bil^{s,\ev}(\Lambda(T_{\scr D(G)}))^{\scr W_G}\subseteq \Bil^{s,\ev}(\Lambda(T_{G^{\sc}}))^{\scr W_G}
$$
and that the last lattice has rank equal to the number of simple factors of $\g^{\ss}$ (see e.g. \cite[Lemma 2.2.1]{FV1}).
\end{proof}

We now define evaluation homomorphisms from the exact sequence  \eqref{E:seq-LTG2}  onto the vertical exact sequence in \eqref{E:mult-car}.

\begin{deflem}\label{D:evG}
Fix the same notation as above. Let $\delta\in \pi_1(G)$ and set $\delta^{\ss}:=\pi_1(\ss)(\delta)\in \pi_1(G^{\ss})$ and $\delta^{\ab}:=\pi_1(\ab)(\delta)\in \pi_1(G^{\ab})$.
There is a (non functorial) commutative diagram with exact rows
$$
\xymatrix{
0\ar[r] & \Bil^{s}(\Lambda(G^{\ab})) \ar[r]^(0.4){B_{\ab}^*}  \ar[d]^{\ev_{G^{\ab}}^{\delta^{\ab}}} &  \Bil^{s,\scr D-\ev}(\Lambda(T_G))^{\scr W_G}  \ar[r]^(0.4){\res_{\scr D}} \ar[d]^{\ev_{G}^{\delta}}&  \Bil^{s,\ev}(\Lambda(T_{\scr D(G)})\vert \Lambda(T_{G^{\ss}}))^{\scr W_G} \ar[r] \ar[d]^{\ev_{\scr D(G)}^{\delta}}&  0,\\
0\ar[r] & \Lambda^*(G^{\ab}) \ar[r]_{\ov{\Lambda_{\ab}^*}}& \frac{\Lambda^*(T_G)}{\Lambda^*(T_{G^{\ad}})} \ar[r] _{\ov{\Lambda_{\scr D}^*}}&  \frac{\Lambda^*(T_{\scr D(G)})}{\Lambda^*(T_{G^{\ad}})} \ar[r] & 0
}
$$
where the vertical arrows, called \emph{evaluation homomorphisms}, are defined as follows:
\begin{enumerate}[(i)]
\item $\ev_{G^{\ab}}^{\delta^{\ab}}(b)=b(\delta^{\ab}\otimes -)$;
\item $ \ev_{G}^{\delta}(b)=b(\delta\otimes -):=[b(d\otimes -)]$, for some lifting $d\in \Lambda(T_{G})$ of $\delta \in \pi_1(G)$.
\item $ \ev_{\scr D(G)}^{\delta}(b)=b(\delta^{\ss}\otimes -):=[b^{\bbQ}(d^{\ss}\otimes -)_{|\Lambda(T_{\scr D(G)})}]$, for some lifting $d^{\ss}\in \Lambda(T_{G^{\ss}})$ of $\delta^{\ss} \in \pi_1(G^{\ss})$. 
\end{enumerate}
\end{deflem}
Note that the notation in the above Definition/Lemma is coherent since if $G$ is a torus then $\ev_G^{\delta}=\ev_{G^{\ab}}^{\delta^{\ab}}$ and if $G$ is semisimple then $\ev_G^{\delta}=\ev_{\scr D(G)}^{\delta}$. 

In order to prove that the last two evaluation homomorphisms are well-defined, we will need the following
 
\begin{lem}\label{L:intGad}
If $b\in \Bil^{s,\ev}(\Lambda(T_{G^{\sc}}))^{\scr W_{G}}$, then its rational extension $b^{\bbQ}$ is integral on $\Lambda(T_{G^{\sc}})\otimes \Lambda(T_{G^{\ad}})+ \Lambda(T_{G^{\ad}})\otimes \Lambda(T_{G^{\sc}})\subseteq \Lambda(T_{G^{\sc}})_{\bbQ}\otimes \Lambda(T_{G^{\sc}})_{\bbQ}$.  
\end{lem}
\begin{proof}
See \cite[Lemma 4.3.4]{BH10}.
\end{proof}

\begin{proof}[Proof of Definition/Lemma \ref{D:evG}]
The fact that the evaluation homomorphism $ \ev_{G}^{\delta}$ (resp.  $ \ev_{\scr D(G)}^{\delta}$) is well-defined follows from the fact that any two lifts of $\delta$ (resp. of $\delta^{\ss}$) differ by  an element $e\in \Lambda(T_{G^{\sc}})$, together with Lemma \ref{L:intGad} which implies that $b^{\bbQ}(e\otimes -)$ is integral on $\Lambda(T_{G^{\ad}})$.

The commutativity of the left square follows from the fact that if $b\in \Bil^{s}(\Lambda(G^{\ab}))$ then 
$$
\ev_G^{\delta}(B_{\ab}^*(b))=[B_{\ab}^*(b)(d\otimes -)]=[b(\delta^{\ab}\otimes \Lambda_{\ab}(-))]=\ov{\Lambda_{\ab}^*}(b(\delta^{\ab}\otimes -))=\ov{\Lambda_{\ab}^*}(\ev_{G^{\ab}}^{\delta^{\ab}}(b)),
$$
where we have used that any lift $d\in \Lambda(T_G)$ of $\delta$ satisfies $\Lambda_{\ab}(d)=\delta^{\ab}$. 
 
Next, observe that, by \eqref{E:dec-Lbis}, any lift $d\in \Lambda(T_G)$ of $\delta\in \pi_1(G)$ decomposes as $d=\delta^{\ab}+d^{\ss}$, where  $d^{\ss}:=p_2(d)\in \Lambda(T_{G^{\ss}})$ has the property that its class in $\pi_1(G^{\ss})$ coincides with $\delta^{\ss}$. Therefore,  the commutativity of the right square follows since  for any  $b\in \Bil^{s,\scr D-\ev}(\Lambda(T_G))^{\scr W_G}$ we have  that
 $$\ov{\Lambda_{\scr D}^*}(\ev_G^{\delta}(b))=[b(d\otimes -)_{|\Lambda(T_{\scr D(G)})}]
 =[b^{\bbQ}(d^{\ss}\otimes -)_{|\Lambda(T_{\scr D(G)})}]=\ev_{\scr D(G)}^{\delta}(\res_{\scr D}(b)), $$
 where we have used that $b^{\bbQ}(\delta^{\ab}\otimes -)_{|\Lambda(T_{\scr D(G)})}=0$ by the claim in the proof of Proposition \ref{P:forms-LTG}.
\end{proof}

\begin{rmk}\label{R:evGss}
The last evaluation homomorphism $\ev_{\scr D(G)}^{\delta}$ can be compared with the evaluation homomorphisms of the semisimple groups $\scr D(G)$ and $G^{\ss}$ in the following way.
First all, note that we have injective restriction homomorphisms
$$
\Bil^{s,(\ev)}(\Lambda(T_{G^{\ss}}))^{\scr W_{G}}\hookrightarrow \Bil^{s,(\ev)}(\Lambda(T_{\scr D(G)})\vert \Lambda(T_{G^{\ss}}))^{\scr W_G}\hookrightarrow \Bil^{s,(\ev)}(\Lambda(T_{\scr D(G)}))^{\scr W_{G}}.
$$
Then we have that:
\begin{enumerate}
\item For any $\delta\in \pi_1(G)$, the following diagram is commutative 
$$ 
\xymatrix{
\Bil^{s,\ev}(\Lambda(T_{G^{\ss}}))^{\scr W_{G}}\ar@{^{(}->}[r] \ar[d]^{\ev_{G^{\ss}}^{\delta^{\ss}}} & \Bil^{s,\ev}(\Lambda(T_{\scr D(G)})\vert \Lambda(T_{G^{\ss}}))^{\scr W_G} \ar[d]^{\ev_{\scr D(G)}^{\delta}}\\
\frac{\Lambda^*(T_{G^{\ss}})}{\Lambda^*(T_{G^{\ad}})} \ar@{^{(}->}[r] & \frac{\Lambda^*(T_{\scr D(G)})}{\Lambda^*(T_{G^{\ad}})}
}
$$ 
where $\delta^{\ss}$ is the image of $\delta$ in $\pi_1(G^{\ss})$.
\item For any  $\delta\in \pi_1(G)$ which is the image of a (necessarily unique) element $\delta^{\scr D}\in \pi_1(\scr D(G))$ (which happens precisely when $\delta$ is a torsion element of $\pi_1(G)$, see \eqref{E:seq-pi1}),
then we have the following commutative diagram 
$$ 
\xymatrix{
\Bil^{s,\ev}(\Lambda(T_{\scr D(G)})\vert \Lambda(T_{G^{\ss}}))^{\scr W_G} \ar@{^{(}->}[r] \ar[d]^{\ev_{\scr D(G)}^{\delta}} & \Bil^{s,\ev}(\Lambda(T_{G^{\ss}}))^{\scr W_{G}}\ar[d]^{\ev_{\scr D(G)}^{\delta^{\scr D}}}\\
 \frac{\Lambda^*(T_{\scr D(G)})}{\Lambda^*(T_{G^{\ad}})} \ar@{=}[r] & \frac{\Lambda^*(T_{\scr D(G)})}{\Lambda^*(T_{G^{\ad}})}
}
$$ 
\end{enumerate}
\end{rmk}

The homomorphism $\ev_{\scr D(G)}^{\delta}$ of Definition/Lemma \ref{D:evG} can be extended to a slightly larger lattice, as we now show.

\begin{deflem}\label{D:evGtilde}
Fix the same notation as above. Consider the lattice
$$
\Bil^{s, \sc-\ev}(\Lambda(T_{\scr D(G)})\vert \Lambda(T_{G^{\ss}}))^{\scr W_G}:=\{b\in \Bil^{s}(\Lambda(T_{\scr D(G)})\vert \Lambda(T_{G^{\ss}}))^{\scr W_G}\: : b_{\Lambda(T_{G^{\sc}})\otimes \Lambda(T_{G^{\sc}})} \: \text{ is even} \}.
$$
\begin{enumerate}[(i)]
\item \label{D:evGtilde1} The natural inclusion $r_G:\Bil^{s,\ev}(\Lambda(T_{\scr D(G)})\vert \Lambda(T_{G^{\ss}}))^{\scr W_G}\hookrightarrow \Bil^{s, \sc-\ev}(\Lambda(T_{\scr D(G)})\vert \Lambda(T_{G^{\ss}}))^{\scr W_G}$
has an elementary $2$-abelian cokernel.  
\item  \label{D:evGtilde2} For any  $\delta\in \pi_1(G)$, the evaluation homomorphism $\ev_{\scr D(G)}^{\delta}: \Bil^{s,\ev}(\Lambda(T_{\scr D(G)})\vert \Lambda(T_{G^{\ss}}))^{\scr W_G}\to \frac{\Lambda^*(T_{\scr D(G)})}{\Lambda^*(T_{G^{\ad}})}$ of Definition/Lemma \ref{D:evG} can be extended to a homomorphism 
\begin{equation*}
\begin{aligned}
\wt \ev_{\scr D(G)}^{\delta}: \Bil^{s,\sc-\ev}(\Lambda(T_{\scr D(G)})\vert \Lambda(T_{G^{\ss}}))^{\scr W_G} & \longrightarrow \frac{\Lambda^*(T_{\scr D(G)})}{\Lambda^*(T_{G^{\ad}})},\\
b & \mapsto b(\delta^{\ss}\otimes -):=[b^{\bbQ}(d^{\ss}\otimes -)_{|\Lambda(T_{\scr D(G)})}],\\
\end{aligned}
\end{equation*}
\end{enumerate}
where $d^{\ss}\in \Lambda(T_{G^{\ss}})$ is any lifting of $\delta^{\ss} \in \pi_1(G^{\ss})$. 
\end{deflem}
\begin{proof}
Part \eqref{D:evGtilde1} follows from the inclusions 
$$
\Bil^{s,\ev}(\Lambda(T_{\scr D(G)})\vert \Lambda(T_{G^{\ss}}))^{\scr W_G}\stackrel{r_G}{\hookrightarrow} \Bil^{s, \sc-\ev}(\Lambda(T_{\scr D(G)})\vert \Lambda(T_{G^{\ss}}))^{\scr W_G}\subseteq \Bil^{s}(\Lambda(T_{\scr D(G)})\vert \Lambda(T_{G^{\ss}}))^{\scr W_G}
$$
together with \eqref{E:cokerpsi}. 

Part \eqref{D:evGtilde2}: the fact that  $\wt \ev_{G}^{\delta}$  is well-defined follows from the fact that any two lifts of $\delta^{\ss}$ differ by  an element $e\in \Lambda(T_{G^{\sc}})$, together with Lemma \ref{L:intGad} which implies that $b^{\bbQ}(e\otimes -)$ is integral on $\Lambda(T_{G^{\ad}})$. The fact that $\wt \ev_{\scr D(G)}^{\delta}\circ r_G=\ev_{\scr D(G)}^{\delta}$ is obvious from the definitions. 
\end{proof}

\section{The universal moduli stack $\bg{G}$ and its Picard group}\label{S:BunG}

Let $G$ be a  \emph{reductive group}  over $k=\ov k$. 
We denote by $\bg{G}$ the \emph{universal moduli stack of $G$-bundles over $n$-marked  curves of genus $g$}. More precisely, for any scheme $S$, $\bg{G}(S)$ is the groupoid of triples $(\cC\to S,\un{\sigma}, E)$, where  $(\pi:\cC\to S,\un\sigma=\{\sigma_1,\ldots, \sigma_n\})$ is a family of  $n$-pointed curves of genus $g$ over $S$ and $E$ is a $G$-bundle on $\cC$. 
We will denote by  $(\pi:\mathcal C_{G,g,n}\to\bg{G},\un \sigma, \mathcal E)$ the universal family of $G$-bundles.

By definition, we have a forgetful surjective morphism 
\begin{equation}\label{E:PhiG}
\begin{aligned}
\Phi_G(=\Phi_{G,g,n}):\bg{G}& \longrightarrow \Mg\\
(\cC\to S,\un{\sigma}, E)& \mapsto (\cC\to S,\un{\sigma})
\end{aligned}
\end{equation}
 onto the moduli stack $\Mg$ of $n$-marked  curves of genus $g$. Note that the universal $n$-marked curve $(\mathcal C_{G,g,n}\to\bg{G},\un \sigma)$ over $\bg{G}$ is the pull-back of the universal $n$-marked curve $(\Cg\to\Mg,\un \sigma)$ over $\Mg$. 



Any morphism of  reductive groups $\phi:G\to H$ determines a morphism of stacks over $\Mg$
\begin{equation}\label{E:fun1}
\begin{array}{lccc}
\phi_\#(=\phi_{\#,g,n}):&\bg{G}&\longrightarrow&\bg{H}\\
&\Big(\cC\to S,\un \sigma, E\Big)&\longmapsto &\Big(\cC\to S,\un \sigma, (E\times H)/G\Big)
\end{array}
\end{equation}
where the (right) action of $G$ on $E\times H$ is $(p,h).g:=(p.g,\phi(g)^{-1}h)$. 

The fiber of $\Phi_{G,g,n}$ over a $n$-pointed curve $(C,p_1,\ldots,p_n)\in \Mg(k)$ is equal to the $k$-stack  $ \mathrm{Bun}_{G}(C)$ of \emph{$G$-bundles on $C$}, i.e. the stack over $k$ whose $S$-points 
$\mathrm{Bun}_{G}(C)(S)$ is the groupoid of $G$-bundles on $C_S:=C\times_k S$ for any $k$-scheme $S$. For any morphism of  reductive groups $\phi:G\to H$, the restriction of the morphism $\phi_{\#, g,n}$ to 
the fiber over $(C,p_1,\ldots, p_n)\in \Mg(k)$ gives rise to a morphism 
\begin{equation*}
 \phi_\#(C):\mathrm{Bun}_{G}(C)\to \mathrm{Bun}_{H}(C).
 \end{equation*}


We collect in the following theorem the geometric properties of   $\bg{G}$ and of the forgetful morphism $\Phi_{G,g,n}$.

\begin{teo}\label{T:propBunG}
Let $G$ be a reductive group.
\begin{enumerate}
\item \label{T:propBunG1} The morphism $\Phi_{G,g,n}$ is locally of finite presentation, smooth, with affine and finitely presented relative diagonal. 
\item \label{T:propBunG2} There is a functorial decomposition into connected components 
\begin{equation}\label{E:PhiGcomp}
\Phi_{G,g,n}: \coprod_{\delta\in \pi_1(G)} \bg{G}^{\delta}\stackrel{\Phi_{G,g,n}^{\delta}}{\longrightarrow} \Mg.
\end{equation}
Similarly, the fiber $\mathrm{Bun}_{G}(C)$ of $\Phi_{G,g,n}$ over $(C,p_1,\ldots, p_n)\in \Mg(k)$ admits a functorial decomposition into connected components 
$$
\mathrm{Bun}_{G}(C)=\coprod_{\delta\in \pi_1(G)} \mathrm{Bun}^{\delta}_{G}(C).
$$
\item \label{T:propBunG3} For each $\delta\in\pi_1(G)$, the stack $\bg{G}^{\delta}$ is smooth and integral of relative dimension over $\Mg$ equal to $(g-1)\dim G.$
\item \label{T:propBunG4} $\Phi_G^{\delta}: \bg{G}^\delta\to \Mg $ is of finite type (or equivalently quasi-compact) for any (or equivalently for some) $\delta\in \pi_1(G)$ if and only if $G$ is a torus.
\item \label{T:propBunG5} 
For any $\delta\in\pi_1(G)$ the morphism  $\Phi_G^{\delta}:\bg{G}^{\delta}\to\Mg$ is fpqc (i.e. faithfully flat and locally quasi-compact) and cohomologically flat in degree zero (i.e. the natural morphism $(\Phi_G^{\delta})^{\sharp}:\cO_{\Mg} \to  (\Phi_G^{\delta})_*(\cO_{\bg{G}^{\delta}})$ is a universal isomorphism).
\end{enumerate}
\end{teo} 
\begin{proof}
See \cite[Sec. 3]{FV1}.
\end{proof}

\subsection{The Picard group of $\bg{G}^{\delta}$}\label{S:Picard}

The aim of this subsection is to recall the results on the Picard group of $\bg{G}^{\delta}$ obtained in \cite{FV1}. We will focus on the case $g\geq 1$; the case $g=0$  is easier to dealt with and it is completely described in \cite[Thm. D]{FV1}.

Note that the Picard group of $\Mg$ is well-known up to torsion (and completely known if $\mathrm{char}(k)\neq 2$ by \cite{FV2}) and the pull-back morphism 
$$(\Phi_G^{\delta})^*:\Pic(\Mg)\to \Pic(\bg{G}^{\delta})$$
is injective since $\Phi_G^{\delta}$  is fpqc and cohomologically flat in degree zero by Theorem \ref{T:propBunG}\eqref{T:propBunG5}. Therefore, we can focus our attention onto the relative Picard group 
\begin{equation}\label{E:RelPic}
\RPic(\bg{G}^{\delta}):=\Pic(\bg{G}^{\delta})/(\Phi_G^{\delta})^*(\Pic(\Mg)).
\end{equation}

A first source of line bundles on $\bg{G}$ comes from the determinant of cohomology $d_{\pi}(-)$ and the Deligne pairing $\langle -,-\rangle_{\pi}$ of line bundles on the universal curve $\pi:\mathcal C_{G,g,n}\to\bg{G}$ (see \cite[Chap. XIII, Sec. 4, 5]{GAC2} for the definition and main properties of  $d_{\pi}(-)$ and $\langle -,-\rangle_{\pi}$).
To be more precise,  any character $\chi:G\to\Gm\in \Lambda^*(G):=\Hom(G, \Gm)$ gives rise to a morphism of stacks 
$$
\chi_\#:\bg{G}\to\bg{\Gm}
$$
and, by  pulling back via $\chi_{\#}$ the universal $\Gm$-bundle (i.e. line bundle) on the universal curve over $\bg{\Gm}$, we get a line bundle $\mathcal L_{\chi}$ on  $\mathcal C_{G,g,n}$. Then, using these line bundles $\mathcal L_{\chi}$  and the sections $\sigma_1,\ldots,\sigma_n$ of $\pi$, we define the following two types of line bundles, that we call \emph{tautological line bundles}, on $\bg{G}$ (and hence, by restriction, also on $\bg{G}^{\delta}$)
\begin{itemize}
\item $\mathscr L(\chi,\zeta):=d_{\pi}\big(\mt L_\chi(\zeta_1\cdot\sigma_1+\ldots+\zeta_n\cdot\sigma_n)\big)$,
\item $\langle (\chi,\zeta),(\chi',\zeta')\rangle:=\langle \mt L_\chi(\zeta_1\cdot\sigma_1+\ldots+\zeta_n\cdot\sigma_n), \mt L_{\chi'}(\zeta_1'\cdot\sigma_1+\ldots+\zeta_n'\cdot\sigma_n)\rangle_{\pi}$,
\end{itemize}
for $\chi, \chi'\in \text{Hom}(G,\Gm)$ and $\zeta=(\zeta_1,\ldots,\zeta_n), \zeta'=(\zeta'_1,\ldots,\zeta'_n)\in \bbZ^n$. 
From the standard relations between the Deligne pairing and the determinant of cohomology, we deduce that 
\begin{equation}\label{E:Del-det}
\begin{aligned}
& \langle-,-\rangle \text{ is bilinear on } \text{Hom}(G,\Gm)\times  \bbZ^n, \\
& \langle (\chi,\zeta),(\chi',\zeta')\rangle=\mathscr L(\chi+\chi',\zeta+\zeta')\otimes \mathscr L(\chi,\zeta)^{-1}\otimes \mathscr L(\chi',\zeta')^{-1}\otimes \mathscr L(0,0). \\
\end{aligned}
\end{equation}
See \cite[Sec. 3.5]{FV1}  for more details.

In the case of a torus $G=T$, the relative Picard group  $\RPic(\bg{T}^{\delta})$ is generated by tautological line bundles and the following theorem also clarifies the dependence relations among the  tautological line bundles.

\begin{teo}\label{T:Pic-torus}\cite[Thm. B]{FV1}
Assume that $g\geq 1$. Let $T$ be an algebraic torus and let $d\in \pi_1(T)$.
               The relative Picard group $\RPic(\bg{T}^d)$ is a free abelian group of finite rank generated by the tautological line bundles and sitting in the following  exact sequences
		\begin{equation}\label{seqT-INT}
		0\to\text{Sym}^2\left(\Lambda^*(T)\right)\oplus\left(\Lambda^*(T)\otimes\bbZ^n\right)\xrightarrow{\tau_T^d+\sigma_T^d}\RPic\left(\bg{T}^d\right)\xrightarrow{\rho_T^d} \Lambda^*(T)\to 0 \quad \text{ if } g\geq 2,
		\end{equation}
		\begin{equation}\label{seqT2-INT}
		0\to\text{Sym}^2\left(\Lambda^*(T)\right)\oplus\left(\Lambda^*(T)\otimes\bbZ^n\right)\xrightarrow{\tau_T^d+\sigma_T^d}\RPic\left(\mathrm{Bun}_{T,1,n}^d\right)\xrightarrow{\rho_T^d} \frac{\Lambda^*(T)}{2\Lambda^*(T)} \to 0 \:\:
		 \text{ if } g=1,
		\end{equation}
		where $\tau_T^d$ (called \emph{transgression map}) and $\sigma_T^d$ are defined by 
		$$
		\begin{array}{cclll}
		\tau_T^d(\chi\cdot\chi')&=&\langle(\chi,0),(\chi',0)\rangle, &\text{ for any } \chi, \chi'\in\Lambda^*(T),\\
		\sigma_T^d(\chi\otimes \zeta)&=&\langle(\chi,0),(0,\zeta)\rangle, &\text{ for any }\chi\in\Lambda^*(T) \text{ and }\zeta\in\bbZ^n,\\
		\end{array}
		$$
		and $\rho_T^d$ is the unique homomorphism such that
		$$ 
		\rho_T^d(\mathscr L(\chi,\zeta))=
		\begin{cases}
		\chi \in \Lambda^*(T) & \text{ if } g\geq 2,\\
		[\chi] \in \frac{\Lambda^*(T)}{2\Lambda^*(T)}& \text{ if } g=1,\\
		\end{cases} \quad \text{ for any } \chi\in\Lambda^*(T) \text{ and }\zeta\in\bbZ^n.
		$$
		Furthermore, the exact sequences \eqref{seqT-INT} and  \eqref{seqT2-INT} are contravariant with respect to homomorphisms of tori.
		
\end{teo}

Now consider the case of an arbitrary reductive group $G$. Note that any character of $G$ factors through its maximal abelian quotient $\ab:G\twoheadrightarrow G^{\ab}$, i.e. the quotient of $G$ by its derived subgroup. 
Hence, the tautological line bundles on $\bg{G}^{\delta}$ are all pull-backs of line bundles via the morphism (induced by $\ab$)
$$\ab_\#:\bg{G}^{\delta}\to \bg{G^{\ab}}^{\delta^{\ab}}$$ 
where  $\delta^{\ab}:=\pi_1(\ab)(\delta)\in \pi_1(G^{\ab})$. Moreover, Theorem \ref{T:Pic-torus} implies that the subgroup of $\RPic(\bg{G}^{\delta})$ generated by the tautological line bundles coincides with the pull-back of  $\RPic(\bg{G^{\ab}}^{\delta^{\ab}})$ via $\ab_\#$.

The next result says that, for an arbitrary reductive group $G$, the relative Picard group of $\bg{G}^{\delta}$ is generated by the image of the pull-back $\ab_\#^*$ together with the image of a functorial transgression map $\tau_G^{\delta}$ (which coincides with the transgression map $\tau_T^d$ in Theorem \ref{T:Pic-torus} if $G=T$ is a torus).

\begin{teo}\label{T:Pic-red}\cite[Thm. C]{FV1}
Assume that $g\geq 1$. 	Let $G$ be a reductive group and let $\ab:G\to G^{\ab}$ be its maximal abelian quotient. Choose a maximal torus $\iota: T_G\hookrightarrow G$ and let $\scr W_G$ be the Weyl group of $G$. Fix $\delta \in \pi_1(G)$ and denote by  $\delta^{\ab}$  its image in $\pi_1(G^{\ab})$.
	\begin{enumerate}
\item\label{T:Pic-red1} There exists a unique injective homomorphism  (called \emph{transgression map} for $G$)\footnote{This is the algebraic analogue of the topological transgression map $H^4(BG,\mathbb Z)\to H^2(\bg{G}^{\delta},\mathbb Z)$, see \cite[\S 1]{TW09}.}
\begin{equation}\label{E:trasgrINT}
\tau_G^{\delta}:\Sym^2(\Lambda^*(T_G))^{\mathscr W_G}\hookrightarrow \RPic\Big(\bg{G}^{\delta}\Big),
\end{equation}
such that, for any lift $d\in \pi_1(T_G)$ of $\delta\in \pi_1(G)$, the composition of $\tau_G^{\delta}$ with 
$$\iota_\#^*:\RPic(\bg{G}^{\delta})\to \RPic(\bg{T_G}^d)$$ 
is equal to the $\scr W_G$-invariant part of the transgression homomorphism  $\tau_{T_G}^d$ defined in Theorem \ref{T:Pic-torus}. 

\item \label{T:Pic-red2} There is a push-out diagram of injective homomorphisms of abelian groups
\begin{equation}\label{E:amalgINT}
\xymatrix{
\Sym^2(\Lambda^*(G^{\ab}))\ar@{^{(}->}[rr]^{\Sym^2 \Lambda^*_\ab}\ar@{^{(}->}[d]^{\tau_{G^\ab}^{\delta^{\ab}}}&& \Sym^2(\Lambda^*(T_G))^{\mathscr W_G}\ar@{^{(}->}[d]^{\tau_G^{\delta}}\\
\RPic\left(\bg{G^{\ab}}^{\delta^{\ab}}\right)\ar@{^{(}->}[rr]^{\ab_\#^*} &&\RPic\Big(\bg{G}^{\delta}\Big)
}
\end{equation}
where $\Sym^2 \Lambda^*_\ab$ is the homomorphism induced by the morphism of tori $T_G\xrightarrow{\iota} G\xrightarrow{\ab} G^{\ab}$. 

\end{enumerate}
	
Furthermore, the transgression homomorphism \eqref{E:trasgrINT} and the diagram \eqref{E:amalgINT} are contravariant with respect to  homomorphisms of reductive groups $\phi:H\to G$ such that $\phi(T_H)\subseteq T_G$.
\end{teo}

\begin{cor}\label{C:Pic-red}
With the notation of Theorem \ref{T:Pic-red}, there is an exact sequence of lattices
\begin{equation}\label{E:seqPic-ab}
0\to \RPic\left(\bg{G^{\ab}}^{\delta^{\ab}}\right) \xrightarrow{\ab_\#^*} \RPic\Big(\bg{G}^{\delta}\Big)  \xrightarrow{\theta_G^{\delta}} \Bil^{s,\ev}(\Lambda(T_{\scr D(G)})\vert \Lambda(T_{G^{\ss}}))^{\scr W_G}\to 0,
\end{equation}
such that  $\theta_G^{\delta}\circ \tau_G^{\delta}$ is equal to the  restriction homomorphism 
$$\res_{\scr D}: \Sym^2\Lambda^*(T_G)^{\scr W_G}\cong \Bil^{s,\ev}(\Lambda(T_G))^{\scr W_G}\twoheadrightarrow \Bil^{s,\ev}(\Lambda(T_{\scr D(G)})\vert \Lambda(T_{G^{\ss}}))^{\scr W_G}.$$
Furthermore, the above exact sequence \eqref{E:seqPic-ab}  is contravariant with respect to  homomorphisms of reductive groups $\phi:H\to G$ such that $\phi(T_H)\subseteq T_G$.
\end{cor}
\begin{proof}
This follows from the push-out diagram \eqref{E:amalgINT} together with Proposition \ref{P:forms-LTG}. 
\end{proof}

For an analogue of the above exact sequence for $g=0$, see \eqref{E:seqRPic-g0}.

\subsection{Two alternative presentations of $\RPic(\bg{G}^{\delta})$}

The aim of this subsection is to give two  alternative presentations of $\RPic(\bg{G}^{\delta})$, for $G$ a reductive group and $g\geq 1$.
We will freely use the notation \S\ref{red-grps} and \S\ref{Sec:int-forms} with respect to a fixed maximal tours $\iota: T_G\hookrightarrow G$. 

The first presentation of $\RPic(\bg{G}^{\delta})$ is based on the following homomorphism.

\begin{deflem}\label{D:gammaG}
Assume that $g\geq 1$. Let $G$ be a reductive group with maximal torus $T_G$ and Weyl group $\scr W_G$, and fix $\delta\in \pi_1(G)$.  There exists a well-defined homomorphism 
$$\gamma_G^{\delta}:\RPic(\bg{G}^{\delta})\to \Bil^{s,\scr D-ev}(\Lambda(T_G))^{\scr W_G}$$
uniquely determined by:
\begin{enumerate}[(i)]
\item \label{D:gammaG1} The composition $\gamma_G^{\delta}\circ  \tau_G^{\delta}$ is equal to 
$$
\alpha: \Sym^2\Lambda^*(T_G)^{\scr W_G}\cong  \Bil^{s,\ev}(\Lambda(T_G))^{\scr W_G}  \hookrightarrow \Bil^{s, \scr D-\ev}(\Lambda(T_G))^{\scr W_G},
$$
where the first isomorphism follows from \eqref{E:iso-b}, and the second injective  homomorphism is the obvious inclusion. 
\item \label{D:gammaG2} The composition $\gamma_G^{\delta}\circ \ab_\#^*$ is equal to the following composition 
$$
\RPic(\bg{G^{\ab}}^{\delta^{\ab}}) \xrightarrow{\gamma_{G^{\ab}}^{\delta^{\ab}}} (\Lambda^*(G^{\ab})\otimes \Lambda^*(G^{\ab}))^s \cong \Bil^{s}(\Lambda(G^{\ab})) \xrightarrow{B_{\ab}^*}   \Bil^{s, \scr D-\ev}(\Lambda(T_G))^{\scr W_G},$$
where $\gamma_{G^{\ab}}^{\delta^{\ab}}$ is the unique homorphism such that
$$
\begin{array}{ccl}
	\gamma_{G^{\ab}}^{\delta^{\ab}}(\scr L(\chi, \zeta) )&=&\chi\otimes \chi,\\
	\gamma_{G^{\ab}}^{\delta^{\ab}}(\langle (\chi, \zeta), (\chi', \zeta') \rangle)&=&\chi\otimes \chi'+\chi'\otimes \chi,
\end{array}
$$
the second isomorphism follows from \eqref{E:iso-b} and $B_{\ab}^*$ is the homomorphism defined in  Corollary  \ref{C:forms-LTG}.  
\end{enumerate}
Moreover, the homomorphism $\gamma_G^{\delta}$ is contravariant with respect to  homomorphisms of reductive groups $\phi:H\to G$ such that $\phi(T_H)\subseteq T_G$.
\end{deflem}
\begin{proof}
The fact that $\gamma_{G^{\ab}}^{\delta^{\ab}}$ is a well-defined homomorphism has been shown in \cite[Prop. 4.1.2, Eq. (4.1.3)]{FV1}.

In order to show that there exists a unique homomorphism satisfying \eqref{D:gammaG1} and \eqref{D:gammaG2},  using that  $\RPic(\bg{G}^{\delta})$ is the pushout \eqref{E:amalgINT}, it is enough to show that 
\begin{equation}\label{E:confr=}
B_{\ab}^*\circ \gamma_{G^{\ab}}^{\delta^{\ab}}\circ \tau_{G^{\ab}}^{\delta^{\ab}} =\alpha\circ \Sym^2\Lambda_{\ab}^*: \Sym^2\Lambda^*(G^{\ab})\to \Bil^{s,\scr D-\ev}(\Lambda(T_G))^{\scr W_G}
\end{equation}
Given $\chi, \chi'\in \Lambda^*(G^{\ab})$ and $x,y\in \Lambda(T_G)$, we compute using the isomorphisms \eqref{E:iso-b}:
$$
(B_{\ab}^*\circ \gamma_{G^{\ab}}^{\delta^{\ab}}\circ \tau_{G^{\ab}}^{\delta^{\ab}})(\chi\cdot \chi')(x\otimes y)= (B_{\ab}^*\circ \gamma_{G^{\ab}}^{\delta^{\ab}})(\langle(\chi,0),(\chi',0)\rangle)(x\otimes y)=
B_{\ab}^*(\chi\otimes \chi'+\chi'\otimes \chi)(x\otimes y)=
$$
\begin{equation}\label{E:confr1}
=(\Lambda_{\ab}^*(\chi)\otimes \Lambda_{\ab}^*(\chi')+\Lambda_{\ab}^*(\chi')\otimes \Lambda_{\ab}^*(\chi))(x\otimes y)= \Lambda_{\ab}^*(\chi)(x)\Lambda_{\ab}^*(\chi')(y)+\Lambda_{\ab}^*(\chi')(x) \Lambda_{\ab}^*(\chi)(y),
\end{equation}
On the other hand,
$$
(\alpha\circ \Sym^2\Lambda_{\ab}^*)(\chi\cdot \chi')(x\otimes y)=\alpha(\Lambda_{\ab}^*(\chi)\cdot \Lambda_{\ab}^*(\chi'))(x\otimes y)=
$$
\begin{equation}\label{E:confr2}
=\Lambda_{\ab}^*(\chi)(x)\Lambda_{\ab}^*(\chi')(y)+\Lambda_{\ab}^*(\chi')(x) \Lambda_{\ab}^*(\chi)(y).
\end{equation}
Hence, we conclude that the equality \eqref{E:confr=} holds and we are done. 

Finally, the (contravariant) functoriality of $\gamma_G^{\delta}$ follows from the functoriality of $B_{\ab}^*$ (see Corollary \ref{C:forms-LTG}) and of $\alpha$ and  $\gamma_{G^{\ab}}^{\delta^{\ab}}$ (which are obvious). 
\end{proof}

Using the homomorphism $\gamma_G^{\delta}$, we get the required new presentation of $\RPic(\bg{G}^{\delta})$.

\begin{teo}\label{T:gG}
Assume that $g\geq 1$. Let $G$ be a reductive group with maximal torus $T_G$ and Weyl group $\scr W_G$, and fix $\delta\in \pi_1(G)$. 
Consider the following group
$$\wh H_{g,n}:=
\begin{cases} 
\bbZ\oplus \bbZ^n  & \text{ if } g\geq 2,\\
 \bbZ^n& \text{ if } g=1.\\
\end{cases}
$$
 There is an exact sequence 
\begin{equation}\label{E:gG}
0 \to \Lambda^*(G^{\ab})\otimes \wh H_{g,n}  \xrightarrow{i_G^{\delta}} \RPic(\bg{G}^{\delta})\xrightarrow{\gamma_G^{\delta}} \Bil^{s,\scr D-ev}(\Lambda(T_G))^{\scr W_G} \to 0,
\end{equation}
where the morphism $i_G^{\delta}$ is defined as 
\begin{equation}\label{E:iG}
\begin{aligned}
& i_G^{\delta}(\chi\otimes (m,\zeta))=  \ab_\#^*\left(\langle \mt L_{\chi}, \omega_{\pi}^m(\sum_{i=1}^n \zeta_i \sigma_i) \rangle\right) &   \text{ if } g\geq 2,\\ 
& i_G^{\delta}(\chi\otimes \zeta) = \ab_\#^*\left( \langle \mt L_{\chi}, \cO(\sum_{i=1}^n \zeta_i \sigma_i) \rangle\right) & \text{ if } g=1.\\
\end{aligned}
\end{equation} 
Moreover, the exact sequence \eqref{E:gG} is contravariant with respect to  homomorphisms of reductive groups $\phi:H\to G$ such that $\phi(T_H)\subseteq T_G$.
\end{teo}
\begin{proof}
Consider the following diagram 
\begin{equation}\label{E:diag-gG}
\xymatrix{
0\ar[r]&  \RPic\left(\bg{G^{\ab}}^{\delta^{\ab}}\right) \ar[r]^{\ab_\#^*}\ar[d]^{\gamma_{G^{\ab}}^{\delta^{\ab}}}&  \RPic\Big(\bg{G}^{\delta}\Big)  \ar[r]^(0.4){\theta_G^{\delta}} \ar[d]^{\gamma_G^{\delta}}&  \Bil^{s,\ev}(\Lambda(T_{\scr D(G)})\vert \Lambda(T_{G^{\ss}}))^{\scr W_G}\ar[r] \ar@{=}[d]& 0\\
 0\ar[r] &  \Bil^{s}(\Lambda(G^{\ab})) \ar[r]^(0.45){B_{\ab}^*} &  \Bil^{s,\scr D-\ev}(\Lambda(T_G))^{\scr W_G}  \ar[r]^(0.45){\res_{\scr D}}&  \Bil^{s,\ev}(\Lambda(T_{\scr D(G)})\vert \Lambda(T_{G^{\ss}}))^{\scr W_G}\ar[r] & 0.
}
\end{equation}

\un{Claim:}  The diagram \eqref{E:diag-gG} is commutative with exact rows. 

Indeed, the first row is exact by Corollary \ref{C:Pic-red} while the second row is exact by Corollary \ref{C:forms-LTG}.  The commutativity of the left square follows from Definition/Lemma \ref{D:gammaG}\eqref{D:gammaG2}. 
In order to prove the commutativity of the right square, using that  $\RPic(\bg{G}^{\delta})$ is the pushout \eqref{E:amalgINT}, it is enough to show that 
\begin{equation}\label{E:check1}
\theta_G^{\delta}\circ \ab_\#^*= \res_{\scr D} \circ \gamma_G^{\delta} \circ \ab_\#^*,
\end{equation}
\begin{equation}\label{E:check2}
\theta_G^{\delta}\circ \tau_G^{\delta}= \res_{\scr D} \circ \gamma_G^{\delta} \circ \tau_G^{\delta},
\end{equation}
Equality \eqref{E:check1} holds since, by what observed above, we have that 
$$ \theta_G^{\delta}\circ \ab_\#^*=0 \quad \text{ and } \quad  \res_{\scr D} \circ \gamma_G^{\delta} \circ \ab_\#^*= \res_{\scr D} \circ B_{\ab}^*\circ \gamma_{G^{\ab}}^{\delta^{\ab}} =0.$$
Equality \eqref{E:check2} holds since, by Corollary \ref{C:Pic-red}, $\theta_G^{\delta}\circ \tau_G^{\delta}$ is equal to the restriction homomorphism 
$$
\Sym^2\Lambda^*(T_G)^{\scr W_G}\cong \Bil^{s,\ev}(\Lambda(T_G))^{\scr W_G}\to \Bil^{s,\ev}(\Lambda(T_{\scr D(G)})\vert \Lambda(T_{G^{\ss}}))^{\scr W_G},
$$ 
which, by Corollary \ref{C:forms-LTG} and Definition/Lemma \ref{D:gammaG}\eqref{D:gammaG1}, is equal to $\res_{\scr D} \circ \gamma_G^{\delta} \circ \tau_G^{\delta}$.

\vspace{0.1cm}
By the above claim, we can apply the snake lemma to \eqref{E:diag-gG} and we obtain the two isomorphisms 
\begin{equation}\label{E:eq-ker-coker}
\begin{aligned}
& \ab_\#^*:\ker (\gamma_{G^{\ab}}^{\delta^{\ab}})\xrightarrow{\cong} \ker(\gamma_G^{\delta}),\\
& B_{\ab}^*: \coker (\gamma_{G^{\ab}}^{\delta^{\ab}})\xrightarrow{\cong} \coker(\gamma_G^{\delta}). 
\end{aligned}
\end{equation} 
From the definition of $\gamma_{G^{\ab}}^{\delta^{\ab}}$ (see Definition/Lemma \ref{D:gammaG}\eqref{D:gammaG2}) together with \eqref{E:basis-form}, it follows that $\gamma_{G^{\ab}}^{\delta^{\ab}}$ is surjective.
Therefore, the second isomorphism in \eqref{E:eq-ker-coker}  implies that also $\gamma_G^{\delta}$ is surjective. 

It remains to prove that  the kernel morphism of $\gamma_G^{\delta}$ is equal to  $i_G^{\delta}$. 
Using the first isomorphism in \eqref{E:eq-ker-coker}  and the fact  that $i_G^{\delta}=\ab_\#^*\circ i_{G^{\ab}}^{\delta^{\ab}}$ by definition, it is enough to prove that 
\begin{equation}\label{E:ker-iab}
\text{ the kernel morphism of } \gamma_{G^{\ab}}^{\delta^{\ab}} \text{ is equal to } i_{G^{\ab}}^{\delta^{\ab}}. 
\end{equation}
With the aim of proving \eqref{E:ker-iab}, let us recall some results from \cite{FV1}. 
Fix an isomorphism $G^{\ab}\cong \bbG_m^r$ which induces an isomorphism $\Lambda^*(G^{\ab})\cong\Lambda^*(\bbG_m^r)=\bbZ^r$. Denote by $\{e_i\}_{i=1}^r$ the canonical basis of $\bbZ^r$ and by $\{f_j\}_{j=1}^n$ the canonical basis of $\bbZ^n$. By \cite[Thm. 4.0.1(2)]{FV1}, the relative Picard group of $\bg{G^{\ab}}^{\delta^{\ab}}$ is freely generated by 
		$$
		\def\arraystretch{1.5}\begin{array}{ll}
		\left\langle (e_i,0),(0,f_j)\right\rangle, &\text{for } i=1,\ldots,r, \text{ and }j=1,\ldots,n,\\
		\left\langle (e_i,0),(e_k,0) \right\rangle=\langle \mt L_{e_i}, \mt L_{e_k}\rangle, &\text{for } 
		\begin{cases} 
		1\leq i\leq k\leq r  &\text{ if } g\geq 2,\\
		1\leq i< k\leq r  &\text{ if } g=1,\\
		\end{cases}\\
		\mathscr L(e_i,0)=d_{\pi}(\mt L_{e_i}),&\text{for } i=1,\ldots,r.
		\end{array}
		$$
Take now an element $\cM\in \RPic(\bg{G^{\ab}}^{\delta^{\ab}})$ and write it as 
\begin{equation*}\label{E:elem-M}
\cM=\sum_{1\leq i \leq j \leq r} a_{ij}\langle (e_i,0),(e_j,0)\rangle +\sum_{1\leq k \leq r} \langle (e_k,0),(0,\zeta^k)\rangle + \sum_{1\leq l \leq r} b_l \mathscr L(e_l,0),
\end{equation*}
for some unique $a_{ij}, b_l \in \bbZ$,  $\zeta^k=(\zeta^k_1,\ldots, \zeta^k_n)\in \bbZ^n$, with the property that $a_{ii}=0$ if $g=1$. From the definition of $\gamma_{G^{\ab}}^{\delta^{\ab}}$ (see Definition/Lemma \ref{D:gammaG}\eqref{D:gammaG2}), we compute 
\begin{equation*}\label{E:comp-M}
\gamma_{G^{\ab}}^{\delta^{\ab}}(\cM)=\sum_{1\leq i \leq j \leq r} a_{ij}(e_i\otimes e_j+e_j\otimes e_i) + \sum_{1\leq l \leq r} b_l e_l\otimes e_l.
\end{equation*}
Hence, we have that 
$$\cM\in \ker(\gamma_{G^{\ab}}^{\delta^{\ab}}) \Leftrightarrow 
\begin{sis} 
& a_{ij}=0 \: \text{ for  } i<j,\\ 
& 2a_{ii}+b_i=0,  \\
\end{sis}
\Leftrightarrow 
\begin{sis} 
& a_{ij}=0 \: \text{ for  } i<j,\\ 
& b_i=-2a_{ii}, \\ 
& (a_{ii}, \zeta^i)\in \wh H_{g,n}.
\end{sis}
$$ 
In other words, $\cM$ belongs to the kernel of $\gamma_{G^{\ab}}^{\delta^{\ab}}$ if and only if $\cM$ has the following form
$$
\cM=\sum_{\substack{1\leq i \leq r\\ (a_{ii}, \zeta^i)\in \wh H_{g,n}}} \left[a_{ii}\langle \mt L_{e_i}, \mt L_{e_i}\rangle-2a_{ii} d_{\pi}(\mt L_{e_i}) +\left\langle \mt L_{e_i}, \cO\left(\sum_k \zeta^i_k\sigma_k\right)\right\rangle \right]= $$
$$=\sum_{\substack{1\leq i \leq r\\ (a_{ii}, \zeta^i)\in \wh H_{g,n}}} \left\langle \mt L_{e_i}, \omega_{\pi}^{a_{ii}}\left(\sum_k \zeta^i_k\sigma_k\right)\right\rangle,
$$
where the second equality follows from \cite[Rmk. 3.5.1]{FV1}. This shows that $i_{G^{\ab}}^{\delta^{\ab}}$ is an injective homomorphism whose image is equal to the kernel of $ \gamma_{G^{\ab}}^{\delta^{\ab}}$, which proves \eqref{E:ker-iab}.  
\end{proof}

We now want to get a second  presentation of $\RPic(\bg{G}^{\delta})$. With this aim, we introduce  the following  homomorphism. 

\begin{deflem}\label{D:wtG}
Assume that $g\geq 1$. Let $G$ be a reductive group with maximal torus $T_G$ and Weyl group $\scr W_G$, and fix $\delta\in \pi_1(G)$. 
 There exists a well-defined (non functorial) homomorphism 
$$\omega_G^{\delta}:\RPic(\bg{G}^{\delta})\to \frac{\Lambda^*(T_{G})}{\Lambda^*(T_{G^{\ad}})}$$
uniquely determined by:
\begin{enumerate}[(i)]
\item \label{D:wtG1} The composition $\omega_G^{\delta}\circ  \tau_G^{\delta}$ is equal to the following composition 
$$
 \Sym^2\Lambda^*(T_G)^{\scr W_G}\xrightarrow{\cong}  \Bil^{s,\ev}(\Lambda(T_G))^{\scr W_G} \subseteq  \Bil^{s,\scr D-\ev}(\Lambda(T_G))^{\scr W_G} \xrightarrow{\ev_G^{\delta}} \frac{\Lambda^*(T_{G})}{\Lambda^*(T_{G^{\ad}})}
$$
where the first isomorphism is induced by \eqref{E:iso-b} and  $\ev_G^{\delta}$ is the homomorphism in Definition/Lemma \ref{D:evG}. 
\item \label{D:wtG2} The composition $\omega_G^{\delta}\circ \ab_\#^*$ is equal to the following composition 
$$
\begin{aligned}
\RPic(\bg{G^{\ab}}^{\delta^{\ab}}) & \xrightarrow{\omega_{G^{\ab}}^{\delta^{\ab}}} \Lambda^*(G^{\ab})  \stackrel{\ov{\Lambda_{\ab}^*}}{\longrightarrow}   \frac{\Lambda^*(T_{G})}{\Lambda^*(T_{G^{\ad}})}, \\
\scr L(\chi, \zeta) & \mapsto [\chi(\delta^{\ab})+|\zeta|+1-g]\chi,  \\
\langle (\chi, \zeta), (\chi', \zeta') \rangle & \mapsto [\chi'(\delta^{\ab})+|\zeta'|]\chi+[\chi(\delta^{\ab})+|\zeta|]\chi',  \\
\end{aligned}
$$
where 
$|\zeta|=\sum_i \zeta_i\in \bbZ$ and similarly for $|\zeta'|$, 
and $\ov{\Lambda_{\ab}^*}$ is the homomorphism in \eqref{E:mult-car}.  
\end{enumerate}
\end{deflem} 
\begin{proof}
The fact that $\omega_{G^{\ab}}^{\delta^{\ab}}$ is well-defined has been proved in \cite[Prop. 4.1.2(i)]{FV1}.

In order to show  that there exists a unique homomorphism $\omega_G^{\delta}$ satisfying properties \eqref{D:wtG1} and \eqref{D:wtG2}, using that  $\RPic(\bg{G}^{\delta})$ is the pushout \eqref{E:amalgINT}, it is enough to show that 
\begin{equation}\label{E:comp=}
\ov{\Lambda_{\ab}^*} \circ \omega_{G^{\ab}}^{\delta^{\ab}}\circ \tau_{G^{\ab}}^{\delta^{\ab}} = \ev_G^{\delta} \circ \Sym^2\Lambda^*_{\ab} :\Sym^2\Lambda^*(G^{\ab})\to  \frac{\Lambda^*(T_{G})}{\Lambda^*(T_{G^{\ad}})}.
\end{equation}
Given $\chi, \chi'\in \Lambda^*(G^{\ab})$, we compute 
$$
(\ov{\Lambda_{\ab}^*} \circ \omega_{G^{\ab}}^{\delta^{\ab}}\circ \tau_{G^{\ab}}^{\delta^{\ab}}) (\chi\cdot \chi')=(\ov{\Lambda_{\ab}^*} \circ \omega_{G^{\ab}}^{\delta^{\ab}})(\langle (\chi, 0),(\chi',0)\rangle)=
\ov{\Lambda_{\ab}^*}\left(\chi(\delta^{\ab}) \chi'+\chi'(\delta^{\ab}) \chi \right)=
$$
\begin{equation}\label{E:comp1}
= \chi(\delta^{\ab}) \Lambda_{\ab}^*(\chi')+\chi'(\delta^{\ab}) \Lambda_{\ab}^*(\chi) ,
\end{equation}
\begin{equation*}
(\ev_G^{\delta} \circ \Sym^2\Lambda^*_{\ab})(\chi\cdot \chi')=\ev_G^{\delta}(\Lambda_{\ab}^*(\chi)\cdot \Lambda_{\ab}^*(\chi'))
=  \Lambda_{\ab}^*(\chi)(d) \Lambda_{\ab}^*(\chi')+\Lambda_{\ab}^*(\chi')(d) \Lambda_{\ab}^*(\chi)=
\end{equation*}
\begin{equation}\label{E:comp2}
= \chi(\Lambda_{\ab}(d)) \Lambda_{\ab}^*(\chi')+\chi'(\Lambda_{\ab}(d)) \Lambda_{\ab}^*(\chi).
\end{equation}
The expressions \eqref{E:comp1} and \eqref{E:comp2} coincide since $\Lambda_{\ab}(d)=\pi_1(\ab)(\delta)=\delta^{\ab}$ for any lift $d\in \Lambda(T_G)$ of $\delta\in \pi_1(G)$;
 hence, the equality \eqref{E:comp=} holds and we are done. 
\end{proof}

\begin{rmk} We remark that the homomorphism $\omega_G^{\delta}$ is not equal to the composition
	$$
	\RPic(\bg{G}^\delta)\xrightarrow{\gamma_G^\delta}\Bil^{s,\scr D-\ev}(\Lambda(T_G))^{\scr W_G} \xrightarrow{\ev_G^{\delta}} \frac{\Lambda^*(T_{G})}{\Lambda^*(T_{G^{\ad}})},
	$$
where $\gamma_G^\delta$ is the homomorphism in Definition/Lemma \ref{D:gammaG} and $\tau_G^{\delta}$ is the homomorphism in Definition/Lemma \ref{D:evG}. More precisely, their compositions with the pull-back $\ab_{\#}^*$ are different.
\end{rmk}

By putting together the homomorphisms of Definition/Lemmas \ref{D:gammaG}  and \ref{D:wtG}, we get the following homomorphism
$$\omega_G^{\delta}\oplus \gamma_G^{\delta}:\RPic(\bg{G}^{\delta})\to \frac{\Lambda^*(T_G)}{\Lambda^*(T_{G^{\ad}})} \oplus   \Bil^{s,\scr D-\ev}(\Lambda(T_G))^{\scr W_G}.$$
With the aim of describing its image, we give the following

\begin{defin}\label{D:NS-BunG}
Let $G$ be a reductive group with maximal torus $T_G$ and Weyl group $\scr W_G$, and fix $\delta\in \pi_1(G)$. Denote by 
$$\NS(\bg{G}^{\delta})\subseteq   \frac{\Lambda^*(T_G)}{\Lambda^*(T_{G^{\ad}})} \oplus   \Bil^{s,\scr D-\ev}(\Lambda(T_G))^{\scr W_G}$$
the subgroup consisting of all the elements $([\chi], b)$ such that 
\begin{equation}\label{E:NS-BunG}
\left[\chi_{|\Lambda(T_{\scr D(G)})}\right]=\ov{\Lambda_{\scr D}^*}([\chi])\text{ is equal to }(\ov{\Lambda_{\scr D}^*}\circ \ev_G^{\delta})(b)=b(\delta\otimes -)_{|\Lambda(T_{\scr D(G)})}:=\left[b(d\otimes -)_{|\Lambda(T_{\scr D(G)})} \right] 
\end{equation}
as elements in $\frac{\Lambda^*(T_{\scr D(G)})}{\Lambda^*(T_{G^{\ad}})}$, where $d\in \Lambda(T_G)$ is any lift of $\delta$, and $\ov{\Lambda_{\scr D}^*}$ and $\ev_G^{\delta}$ are the homomorphisms of Definition/Lemma \ref{D:evG}.
\end{defin}

The group $\NS(\bg{G}^{\delta})$ is contravariant with respect to homomorphisms of reductive groups $\phi:H\to G$ such that $\phi(T_H)\subseteq T_G$.

\begin{deflem}\label{D:funcNS}
Let $\phi: H\to G$ be a homomorphism of reductive groups, and choose maximal tori $T_G\subseteq G$ and $T_H\subseteq H$ in such a way that $\phi(T_H)\subseteq T_G$. 
Let  $\epsilon\in \pi_1(H)$ and set $\delta:=\pi_1(\phi)(\epsilon)\in \pi_1(G)$. Pick a lift $e\in \Lambda(T_H)$ of  $\epsilon\in \pi_1(H)$. 
Then there exists a well-defined homomorphism
\begin{equation}\label{E:funcNS}
\begin{aligned}
\phi^{*,\NS}:\NS(\bg{G}^{\delta})& \longrightarrow \NS(\bg{H}^{\epsilon}),\\
([\chi], b) & \mapsto ([\Lambda_{\phi}^*(\chi^{b(\Lambda_{\phi}(e)\otimes -)})],B_{\phi}^*(b)),
\end{aligned}
\end{equation}
where $\Lambda_{\phi}:\Lambda(T_H)\to \Lambda(T_G)$, $\Lambda_{\phi}^*:\Lambda^*(T_G)\to \Lambda(T_H)$ and $B_{\phi}^*:\Bil^{s}(\Lambda(T_G))\to \Bil^{s}(\Lambda(T_H))$ are the natural morphisms induced by $\phi:T_H\to T_G$, 
and  $\chi^{b(\Lambda_{\phi}(e)\otimes -)}\in \Lambda^*(T_G)$ is the unique lift of $[\chi]\in \frac{\Lambda^*(T_G)}{\Lambda^*(T_{G^{\ad}})}$ such that 
$$(\chi^{b(\Lambda_{\phi}(e)\otimes -)})_{|\Lambda(T_{\scr D}(G))}=b(\Lambda_{\phi}(e)\otimes -)_{|\Lambda(T_{\scr D(G)})}.$$

Moreover, if $\psi:L\to H$ is another homomorphism of reductive groups and we choose a maximal torus $T_L\subseteq L$ in such a way that $\psi(T_L)\subseteq T_H$, then 
$(\phi\circ \psi)^*=\psi^*\circ \phi^*$. 
\end{deflem}
\begin{proof}
Let us first consider the following two special cases:

\un{Special case I:}  $\phi: T'\to T$ is a morphism of tori. 

Choose $d'\in \Lambda(T')$ and set $d:=\Lambda_{\phi}(d')\in \Lambda(T)$. The definition \eqref{E:funcNS} reduces in this special case to  
\begin{equation}\label{E:funcNStori}
\begin{aligned}
\phi^{*,\NS}:\NS(\bg{T}^{d})& \longrightarrow \NS(\bg{T'}^{d'}),\\
(\chi, b) & \mapsto (\Lambda_{\phi}^*(\chi), B_{\phi}^*(b)),
\end{aligned}
\end{equation}
which is clearly a well-defined homomorphism. Moreover,  the association $\phi\mapsto \phi^{*,\NS}$ is  compatible with the  composition of morphisms of tori.

\un{Special case II:}  $\phi=\iota:T_G\to G$ is the inclusion of a maximal torus inside a reductive group $G$. 

Choose a lift $d\in \Lambda(T_G)$ of $\delta\in \pi_1(G)$. Pick an element $([\chi], b)\in \NS(\bg{G}^{\delta})$. Consider the following commutative diagram with surjective arrows
\begin{equation}\label{E:quot-Tad}
\xymatrix{
\Lambda^*(T_G) \ar@{->>}[r]^(0.45){\Lambda_{\scr D}^*} \ar@{->>}[d]& \Lambda^*(T_{\scr D(G)})\ar@{->>}[d] \\
\frac{\Lambda^*(T_G)}{\Lambda^*(T_{G^{\ad}})} \ar@{->>}[r]^{\ov{\Lambda_{\scr D}^*}} & \frac{\Lambda^*(T_{\scr D(G)})}{\Lambda^*(T_{G^{\ad}})}
}
\end{equation}
The diagram \eqref{E:quot-Tad} is a pull-back diagram since the kernels of the vertical surjections are both equal to $\Lambda^*(T_{G^{\ad}})$ while the kernels  of the horizontal surjections are both equal to $\Lambda^*(G^{\ab})$ and 
$\Lambda^*(T_{G^{\ad}})\cap \Lambda^*(G^{\ab})=\{0\}$. From this and condition \eqref{E:NS-BunG}, it follows that there exists a unique lift of $[\chi]\in \frac{\Lambda^*(T_G)}{\Lambda^*(T_{G^{\ad}})}$, that we denote by $\chi^{b(d\otimes -)}\in \Lambda^*(T_G)$, with the property that 
\begin{equation}\label{E:lift-chi}
(\chi^{b(d\otimes -)})_{|\Lambda(T_{\scr D(G)})}=b(d\otimes -)_{|\Lambda(T_{\scr D(G)})}. 
\end{equation}
The definition \eqref{E:funcNS} reduces in this special case to  
\begin{equation}\label{E:funcNSiota}
\begin{aligned}
\iota^{*,\NS}:\NS(\bg{G}^{\delta})& \longrightarrow \NS(\bg{T_G}^d),\\
([\chi], b) & \mapsto (\chi^{b(d\otimes -)},b),
\end{aligned}
\end{equation}
which is  a well-defined and injective homomorphism whose image is equal to 
\begin{equation}\label{E:Im-iota*}
\Im(\iota^{*,\NS})=\{(\chi, b): \: \chi_{|\Lambda(T_{\scr D(G)})}=b(d\otimes -)_{|\Lambda(T_{\scr D(G)})} \: \text{ and } \: b\in \Bil^{s,\scr D-\ev}(\Lambda(T_G))^{\scr W_G}\}.
\end{equation}

We now go back to the general case. Denote the inclusions of the maximal tori in $G$ and $H$ by, respectively, $\iota_G:T_G\hookrightarrow G$ and $\iota_H:T_H\hookrightarrow H$ and set $\phi_T:=\phi_{|T_H}:T_H\to T_G$. 
Set also $d:=\Lambda_{\phi}(e)\in \Lambda(T_G)$, which is a lift of $\delta\in \pi_1(G)$. 
Consider the composition 
\begin{equation}\label{E:compNS}
\begin{aligned}
\phi_T^{*,\NS}\circ \iota_G^{*,\NS}:\NS(\bg{G}^{\delta}) & \longrightarrow \NS(\bg{T_H}^{e})\\
([\chi], b) & \mapsto (\Lambda^*_{\phi}(\chi^{b(d\otimes -)}),B_{\phi}^*(b))
\end{aligned}
\end{equation}
which is a well-defined homomorphism by the special cases already treated. Moreover, for any $([\chi], b)\in \NS(\bg{G}^{\delta})$ and for every $x\in \Lambda(T_{\scr D(H)})$, we have that 
$$
\Lambda^*_{\phi}(\chi^{b(d\otimes -)})(x)=\chi^{b(d\otimes -)}(\Lambda_{\phi}(x))=b(d\otimes \Lambda_{\phi}(x))=b(\Lambda_{\phi}(e)\otimes  \Lambda_{\phi_T}(x))=(B_{\phi}^*(b))(e\otimes x), 
$$
where in the second equality we have used \eqref{E:lift-chi}. This computation, together with \eqref{E:Im-iota*}, implies that the image of  $\phi_T^{*,\NS}\circ \iota_G^{*,\NS}$ is contained in the image of $\iota_H^{*,\NS}: \NS(\bg{H}^{\epsilon})\to \NS(\bg{T_H}^{e})$. Hence, we get a factorization
\begin{equation}\label{E:factoNS}
\phi_T^{*,\NS}\circ \iota_G^{*,\NS}=\iota_H^{*,\NS}\circ \phi^{*,\NS} \text{ for some (unique) homomorphism }  \phi^{*,\NS}:\NS(\bg{G}^{\delta}) \rightarrow \NS(\bg{H}^{\epsilon}).
\end{equation}
From the expression \eqref{E:compNS}, we conclude that $\phi^{*,\NS}$ is given by the formula \eqref{E:funcNS}. 

Finally, the compatibility of the association $\phi\mapsto \phi^{*,\NS}$ with the composition of morphisms is due to the factorization \eqref{E:factoNS} together with the Special Case I. 
\end{proof}

The group $\NS(\bg{G}^{\delta})$ admits a functorial two-step filtration, that we describe in the following

\begin{prop}\label{P:seqNS}
Let $G$ be a reductive group with maximal torus $T_G$ and Weyl group $\scr W_G$, and fix $\delta\in \pi_1(G)$.
We have the following commutative diagram, with exact rows and columns
\begin{equation}\label{E:seqNS}
\xymatrix{
  & & \Bil^s(\Lambda(G^{\ab}))\ar@{^{(}->}[d]^{B_{\ab}^*}  \\
 \Lambda^*(G^{\ab}) \ar@{^{(}->}[r]^(0.45){\ov{\Lambda_{\ab}^*}\oplus 0} \ar@{^{(}->}[d]^{i_1} & \NS(\bg{G}^{\delta})\ar@{=}[d] \ar@{->>}[r]^(0.4){\res_G^{\NS}} & \Bil^{s,\scr D-\ev}(\Lambda(T_G))^{\scr W_G}  \ar@{->>}[d]^{\res_{\scr D}}  \\
 \Lambda^*(G^{\ab})\oplus \Bil^s(\Lambda(G^{\ab}))= \NS(\bg{G^{\ab}}^{\delta^{\ab}})\ar@{^{(}->}[r]^(0.7){{\ab}^{*,\NS}} \ar@{->>}[d]^{p_2} & \NS(\bg{G}^{\delta}) \ar@{->>}[r]^(0.35){\res_{\scr D}^{\NS}}  & \Bil^{s,\ev}(\Lambda(T_{\scr D(G)})\vert \Lambda(T_{G^{\ss}}))^{\scr W_G}  \\
 \Bil^s(\Lambda(G^{\ab}) &  & 
}
\end{equation}
where the identification $ \NS(\bg{G^{\ab}}^{\delta^{\ab}})=\Lambda^*(G^{\ab})\oplus \Bil^s(\Lambda(G^{\ab}))$ follows from Definition \ref{D:NS-BunG}, $i_1$ is the inclusion of the first factor and $p_2$ is the projection onto the second factor, 
  the right vertical column is \eqref{E:seq-LTG2}, $\res_G^{\NS}$ is the projection onto the second factor and  $\res_{\scr D}^{\NS}:= \res_{\scr D}\circ \res_G^{\NS}$. 

Moreover, the diagram \eqref{E:seqNS} is contravariant with respect to homomorphisms of reductive groups $\phi:H\to G$ such that $\phi(T_H)\subseteq T_G$.

\end{prop}
\begin{proof}
The commutativity of the right square of the diagram is clear, while the commutativity of the left square follows from the fact $\ab^{*,\NS}=\ov{\Lambda_{\ab}^*}\oplus B_{\ab}^*$, as it is easily deduced from Definition/Lemma \ref{D:funcNS}. 

The exactness of the left and central  columns is clear, while the exactness of the right  column follows from Corollary \ref{C:forms-LTG}. 

The exactness of the upper  row follows from the definition of $\NS(\bg{G}^{\delta})$ together with the exactness of the  column in \eqref{E:mult-car}.  

It remains to prove the exactness of the lower row. The injectivity of $\ab^{*,\NS}=\ov{\Lambda_{\ab}^*}\oplus B_{\ab}^*$ and the fact that $\res_{\scr D}^{\NS}\circ \ab^{*,\NS}=0$ are obvious. The surjectivity of $\res_{\scr D}^{\NS}$ follows from the surjectivity of $\res_G^{\NS}$ and of $\res_{\scr D}$.
Let us now prove that $\ker(\res_{\scr D}^{\NS})\subseteq \Im(\ov{\Lambda_{\ab}^*}\oplus B_{\ab}^*)$. Pick an element $([\chi], b)\in \NS(\bg{G}^{\delta})$ such that $0=\res_{\scr D}^{\NS}(([\chi], b))=\res_{\scr D}(b)$. 
 By Corollary \ref{C:forms-LTG}, there exists $b^{\ab}\in \Bil^s(\Lambda(G^{\ab}))$ such that $b=B_{\ab}^*(b^{\ab})$. Moreover, from  \eqref{E:NS-BunG}  it follows that $\left[\chi_{|\Lambda(T_{\scr D(G)})}\right]=0\in \frac{\Lambda^*(T_{\scr D(G)})}{\Lambda^*(T_{G^{\ad}})}$. Hence, by \eqref{E:mult-car}, there exists $\chi^{\ab}\in \Lambda^*(G^{\ab})$ such that $\ov{\Lambda_{\ab}^*}(\chi^{\ab})=[\chi]$. 
Therefore, $(\chi^{\ab},b^{\ab})\in \NS(\bg{G^{\ab}}^{\delta^{\ab}})$ and $\ov{\Lambda_{\ab}^*}\oplus B_{\ab}^*((\chi^{\ab},b^{\ab}))=([\chi],b)$ and we are done. 

Finally, the functoriality of the diagram \eqref{E:seqNS} follows straightforwardly from the definition of the pull-back morphism  \eqref{E:funcNS}. 
\end{proof}

Using the above homomorphisms $\omega_G^{\delta}$ and $\gamma_G^{\delta}$, we can now give the following new presentation of $\RPic(\bg{G}^{\delta})$. 

\begin{teo}\label{T:oG+gG}
Assume that $g\geq 1$. Let $G$ be a reductive group with maximal torus $T_G$ and Weyl group $\scr W_G$, and fix $\delta\in \pi_1(G)$. 
Consider the following group
$$H_{g,n}:=
\begin{cases} 
\{(m,\zeta)\in \bbZ\oplus \bbZ^n\: : \: (2g-2)m+|\zeta|=0\} & \text{ if } g\geq 2,\\
\{\zeta\in \bbZ^n\: : \: |\zeta|=0\} & \text{ if } g=1.\\
\end{cases}
$$
\begin{enumerate}
\item \label{T:oG+gG1} There is an exact sequence 
\begin{equation}\label{E:oG+gG1}
0 \to \Lambda^*(G^{\ab})\otimes H_{g,n}  \xrightarrow{j_G^{\delta}} \RPic(\bg{G}^{\delta})\xrightarrow{\omega_G^{\delta}\oplus \gamma_G^{\delta}} \NS(\bg{G}^{\delta}),
\end{equation}
where the morphism $j_G^{\delta}$ is defined as 
\begin{equation*}
\begin{aligned}
& j_G^{\delta}(\chi\otimes (m,\zeta))=  \ab_\#^*\left(\langle \mt L_{\chi}, \omega_{\pi}^m(\sum_{i=1}^n \zeta_i \sigma_i) \rangle\right) &   \text{ if } g\geq 2,\\ 
& j_G^{\delta}(\chi\otimes \zeta) = \ab_\#^*\left( \langle \mt L_{\chi}, \cO(\sum_{i=1}^n \zeta_i \sigma_i) \rangle\right) & \text{ if } g=1.\\
\end{aligned}
\end{equation*}
Moreover, the exact sequence \eqref{E:oG+gG1} is contravariant with respect to  homomorphisms of reductive groups $\phi:H\to G$ such that $\phi(T_H)\subseteq T_G$.
\item \label{T:oG+gG2} The image of $\omega_G^{\delta}\oplus \gamma_G^{\delta}$ is equal to 
\begin{equation}\label{E:oG+gG2}
 \Im(\omega_G^{\delta}\oplus \gamma_G^{\delta})=
 \begin{cases} 
\NS(\bg{G}^{\delta}) & \text{ if } n\geq 1,\\
  \left\{([\chi], b)\in \NS(\bg{G}^{\delta})\: : \: 
 \begin{aligned}
 &  \left[\chi(x)-b(\delta\otimes x)\right]+(g-1)b(x\otimes x) \\
 & \text{ is a multiple of } 2g-2,  \text{ for any } x\in \Lambda(T_G)
 \end{aligned}
 \right\} & \text{ if } n=0. 
\end{cases} 
 \end{equation}
 Moreover, if $n=0$ then
 $$
 \frac{\NS(\bg{G}^{\delta})}{\Im(\omega_G^{\delta}\oplus \gamma_G^{\delta})}\cong \left(\frac{\bbZ}{(2g-2)\bbZ}\right)^{\dim G^{\ab}}.
 $$
 \end{enumerate}
\end{teo}

\begin{rmk}
\noindent 
\begin{enumerate}[(i)]
\item For $n=0$, the subgroup on the right hand side  \eqref{E:oG+gG2} is well-defined. Indeed, by \eqref{E:NS-BunG}, $\left[\chi(x)-b(\delta\otimes x)\right] $ is well-defined for any $x\in \Lambda(T_G)$ and it is equal to $\left[\chi(x^{\ab})-b(\delta\otimes x^{\ab})\right]$, where $x^{\ab}$ is the image of $x$ in $\Lambda(G^{\ab})$. 
\item If either $n=0$ or $g=n=1$ then $H_{g,n}=0$, which implies that  the map $\omega_G^{\delta}\oplus \gamma_G^{\delta}$ is injective.
\end{enumerate}
\end{rmk}

\begin{proof}
The theorem has been proved for a torus  in \cite[Prop. 4.3.1]{FV1}, and we are going to apply this result for $G^{\ab}$ in order to prove the case of a general reductive group $G$. 

Let us first prove \eqref{T:oG+gG1} by dividing the proof in two steps. 

\un{Step I:}  The image of 
$$\omega_G^{\delta}\oplus \gamma_G^{\delta}:\RPic(\bg{G}^{\delta})\to  \frac{\Lambda^*(T_G)}{\Lambda^*(T_{G^{\ad}})} \oplus   \Bil^s(\Lambda(T_G))^{\scr W_G}$$
 is contained in $\NS(\bg{G}^{\delta})$. 
 
In order to prove this,  it is enough to show, using that  $\RPic(\bg{G}^{\delta})$ is the pushout \eqref{E:amalgINT},  that 
 \begin{equation}\label{E:im-oGgG}
\begin{aligned}
& \left((\omega_G^{\delta}\oplus \gamma_G^{\delta})\circ \ab_{\#}^*\right)(\RPic(\bg{G^{\ab}}^{\delta^{\ab}}))\subseteq \NS(\bg{G}^{\delta}), \\
& \left((\omega_G^{\delta}\oplus \gamma_G^{\delta})\circ \tau_G^{\delta}\right)(\Sym^2\Lambda^*(T_G)^{\scr W_G})\subseteq \NS(\bg{G}^{\delta}).
 \end{aligned}
 \end{equation}
The first containment in \eqref{E:im-oGgG} follows since Definition/Lemma \ref{D:wtG}\eqref{D:wtG2} and Definition/Lemma \ref{D:gammaG}\eqref{D:gammaG2} imply that  
$(\omega_G^{\delta}\oplus \gamma_G^{\delta})\circ \ab_{\#}^*$ factorizes as the composition 
\begin{equation}\label{E:oGgG-ab}
  \RPic(\bg{G^{\ab}}^{\delta^{\ab}}) \xrightarrow{\omega_{G^{\ab}}^{\delta^{\ab}}\oplus \gamma_{G^{\ab}}^{\delta^{\ab}}} \Lambda^*(G^{\ab})\oplus \Bil^s(\Lambda(G^{\ab}))\stackrel{\ov{\Lambda_{\ab}^*}\oplus B_{\ab}^*}{\longhookrightarrow} \frac{\Lambda^*(T_G)}{\Lambda^*(T_{G^{\ad}})} \oplus   \Bil^s(\Lambda(T_G))^{\scr W_G},
\end{equation}
and  $\Im(\ov{\Lambda_{\ab}^*}\oplus B_{\ab}^*)\subseteq \NS(\bg{G}^{\delta})$ as observed in Proposition \ref{P:seqNS}. 

Take now an element  $b\in \Sym^2\Lambda^*(T_G)^{\scr W_G} \cong \Bil^{s,\ev}(\Lambda(T_G))^{\scr W_G}$. Using  Definition/Lemma \ref{D:wtG}\eqref{D:wtG1} and Definition/Lemma \ref{D:gammaG}\eqref{D:gammaG1}, we compute 
\begin{equation}\label{E:oGgG-tras}
\left((\omega_G^{\delta}\oplus \gamma_G^{\delta})\circ \tau_G^{\delta}\right)(b)=(\ev_G^{\delta}(b), b)=(b(\delta\otimes -),b)\in \frac{\Lambda^*(T_G)}{\Lambda^*(T_{G^{\ad}})} \oplus   \Bil^s(\Lambda(T_G))^{\scr W_G}.
\end{equation}
From Definition \ref{D:NS-BunG}, it follows easily that the element $(b(\delta\otimes -),b)$ belongs to $\NS(\bg{G}^{\delta})$, and this proves  the second containment in \eqref{E:im-oGgG}.

\vspace{0.1cm}

\un{Step II:}  The kernel of  $\omega_G^{\delta}\oplus \gamma_G^{\delta}$ is equal to 
$$
j_{G}^{\delta}:\Lambda^*(G^{\ab})\otimes H_{g,n}\hookrightarrow \RPic(\bg{G}^{\delta}). 
$$

\vspace{0.1cm}

Consider the following  diagram  
\begin{equation}\label{E:diag-oG-gG}
\xymatrix{ 
\RPic\left(\bg{G^{\ab}}^{\delta^{\ab}}\right) \ar@{^{(}->}[r]^{\ab_\#^*} \ar[d]^{\omega_{G^{\ab}}^{\delta^{\ab}}\oplus \gamma_{G^{\ab}}^{\delta^{\ab}}}& \RPic\Big(\bg{G}^{\delta}\Big)  \ar@{->>}[r]^(0.4){\theta_G^{\delta}}\ar[d]^{\omega_G^{\delta}\oplus \gamma_G^{\delta}}&  \Bil^{s,\ev}(\Lambda(T_{\scr D(G)})\vert \Lambda(T_{G^{\ss}}))^{\scr W_G}\ar@{=}[d] \\
 \NS(\bg{G^{\ab}}^{\delta^{\ab}})=\Lambda^*(G^{\ab})\oplus \Bil^s(\Lambda(G^{\ab}))\ar@{^{(}->}[r]^(0.7){\ov{\Lambda_{\ab}^*}\oplus B_{\ab}^*}&   \NS(\bg{G}^{\delta}) \ar@{->>}[r]^(0.35){\res_{\scr D}^{\NS}} & 
\Bil^{s,\ev}(\Lambda(T_{\scr D(G)})\vert \Lambda(T_{G^{\ss}}))^{\scr W_G}
}
\end{equation} 
whose rows are exact by Corollary \ref{C:Pic-red} and Proposition \ref{P:seqNS}, and whose commutativity follows from \eqref{E:oGgG-ab} and \eqref{E:oGgG-tras}.  By applying the snake lemma to \eqref{E:diag-oG-gG}, we get that 
\begin{equation}\label{E:oGgG-ker}
\ker(\omega_{G^{\ab}}^{\delta^{\ab}}\oplus \gamma_{G^{\ab}}^{\delta^{\ab}})\xrightarrow[\cong]{\ab_{\#}^*} \ker(\omega_G^{\delta}\oplus \gamma_G^{\delta}), 
\end{equation}
\begin{equation}\label{E:oGgG-coker}
\coker(\omega_{G^{\ab}}^{\delta^{\ab}}\oplus \gamma_{G^{\ab}}^{\delta^{\ab}})\xrightarrow[\cong]{\ov{\Lambda_{\ab}^*}\oplus B_{\ab}^*}\coker(\omega_G^{\delta}\oplus \gamma_G^{\delta}).
\end{equation}
By applying \cite[Prop. 4.3.1]{FV1} to $G^{\ab}$, we get that the kernel of $\omega_{G^{\ab}}^{\delta^{\ab}}\oplus \gamma_{G^{\ab}}^{\delta^{\ab}}$ is equal to 
$$
j_{G^{\ab}}^{\delta^{\ab}}:\Lambda^*(G^{\ab})\otimes H_{g,n}\hookrightarrow \RPic(\bg{G^{\ab}}^{\delta^{\ab}}).
$$
By combining this with \eqref{E:oGgG-ker} and the fact that $j_G^{\delta}=\ab_{\#}^*\circ  j_{G^{\ab}}^{\delta^{\ab}}$, Step II follows. 

\vspace{0.1cm}

Finally, the functoriality of the morphism $\omega_G^{\delta}\oplus \gamma_G^{\delta}$ (and hence of the sequence \eqref{E:oG+gG1}) follows straightforwardly from \eqref{E:oGgG-ab} and \eqref{E:oGgG-tras}.

Let us now prove \eqref{T:oG+gG2}. 
If $n>0$ then $\omega_{G^{\ab}}^{\delta^{\ab}}\oplus \gamma_{G^{\ab}}^{\delta^{\ab}}$ is surjective by \cite[Prop. 4.3.1]{FV1}, which, combined with \eqref{E:oGgG-coker}, implies that  $\omega_G^{\delta}\oplus \gamma_G^{\delta}$ is also surjective. 

Assume now that $n=0$ and call $I_G^{\delta}$ the subgroup of $\NS(\bg{G}^{\delta})$ defined on the right hand side of \eqref{E:oG+gG2}. By \cite[Prop. 4.3.1]{FV1}, we know that $\Im(\omega_{G^{\ab}}^{\delta^{\ab}}\oplus \gamma_{G^{\ab}}^{\delta^{\ab}})=I_{G^{\ab}}^{\delta^{\ab}}$. Using this and \eqref{E:oGgG-coker}, in order to prove that $\Im(\omega_G^{\delta}\oplus \gamma_G^{\delta})=I_G^{\delta}$, it is enough to prove that 
\begin{equation}\label{E:IG1}
\begin{sis}
&\Im(\omega_G^{\delta}\oplus \gamma_G^{\delta})\subseteq I_G^{\delta},\\
&(\ov{\Lambda_{\ab}^*}\oplus B_{\ab}^*)^{-1}(I_G^{\delta})\subseteq I_{G^{\ab}}^{\delta^{\ab}}.
\end{sis}
\end{equation}
Furthermore, using that  $\RPic(\bg{G}^{\delta})$ is the pushout \eqref{E:amalgINT},  the inclusions \eqref{E:IG1} are equivalent to the following two conditions
 \begin{equation}\label{E:IG2}
\left((\omega_G^{\delta}\oplus \gamma_G^{\delta})\circ \tau_G^{\delta}\right)(\Sym^2\Lambda^*(T_G)^{\scr W_G})\subseteq I_G^\delta,
 \end{equation}
 \begin{equation}\label{E:IG3}
(\ov{\Lambda_{\ab}^*}\oplus B_{\ab}^*)^{-1}(I_G^{\delta})= I_{G^{\ab}}^{\delta^{\ab}}.
\end{equation}
Now, inclusions \eqref{E:IG2} follows from \eqref{E:oGgG-tras} together with the fact that $b$ is even. 
Furthermore,  if $(\chi, b)\in  \NS(\bg{G^{\ab}}^{\delta^{\ab}})=\Lambda^*(G^{\ab})\oplus \Bil^s(\Lambda(G^{\ab}))$, then we have for every $x\in \Lambda(T_G)$ 
$$
\ov{\Lambda_{\ab}^*}(\chi)(x)-B_{\ab}^*(b)(\delta, x)+(g-1)B_{\ab}^*(b)(x, x)=\chi(x^{\ab})-b(\delta^{\ab},x^{\ab})+(g-1)b(x^{\ab},x^{\ab}),
$$
where $x^{\ab}$ is the image of $x$ into $\Lambda(G^{\ab})$. Since $\Lambda(T_G)$ surjects onto $\Lambda(T_{G^{\ab}})$, we deduce that 
$$(\chi, b)\in I_{G^{\ab}}^{\delta^{\ab}}\Leftrightarrow (\ov{\Lambda_{\ab}^*},B_{\ab}^*)^{-1}(\chi, b)\in I_G^{\delta},$$
which shows \eqref{E:IG3} and concludes the first assertion of part \eqref{T:oG+gG2}.

The last assertion follows from \eqref{E:oGgG-coker} and the analogous result for $G^{\ab}$, see \cite[Rmk. 4.3.2]{FV1}.  
\end{proof}

\section{The restriction homomorphism}\label{Sec:res}

The aim of this section is to describe the restriction homomorphism 
\begin{equation}\label{E:reshom}
\res_G^{\delta}(C): \RPic(\bg{G}^{\delta})\to \Pic(\Bun_G^{\delta}(C))
\end{equation}
for any $(C,p_1,\ldots, p_n)\in \Mg(k)$.  

Before doing this, we need to recall the description of $\Pic(\Bun_G^{\delta}(C))$ obtained in \cite{BH10}.

\begin{teo}\cite[Thm. 5.3.1]{BH10}\label{T:Pic-BunGC}
Let  $C$ be a (irreducible, projective, smooth) curve of genus $g\geq 0$ over $k=\ov k$ and denote by $J_C$ its Jacobian. 
Let $G$ be a reductive group over $k$ and fix $\delta\in \pi_1(G)$. 
 Then there exists a (contravariantly) functorial  exact sequence of abelian groups
 \begin{equation}\label{E:Pic-BunGC}
 0 \to \Hom(\pi_1(G),J_C(k))\xrightarrow{j_G^{\delta}(C)} \Pic(\Bun_G^{\delta}(C))\xrightarrow{c_G^{\delta}(C)} \NS(\Bun_G^{\delta}(C))\to 0,
 \end{equation}
where the \emph{Neron-Severi} group  $\NS(\Bun_G^{\delta}(C))$ is the group of all triples $(l_{\scr R}, b_{\scr R}, b)$ consisting of 
\begin{enumerate}[(i)]
\item $l_{\scr R}\in \Lambda^*(\scr R(G))$; 
\item $b_{\scr R}\in \Hom^s(\Lambda(\scr R(G))\otimes \Lambda(\scr R(G)), \End(J_C))$,  i.e. $b_{\scr R}: \Lambda(\scr R(G))\otimes \Lambda(\scr R(G))\to \End(J_C)$ with the property that $b_{\scr R}(x_1\otimes x_2)^{\dag}=b_{\scr R}(x_2\otimes x_1)$ where  $\dag: \End(J_C)\to \End(J_C)$ is the Rosati involution associated to the canonical polarization on $J_C$; 
 \item $b\in \Bil^{s,\ev}(\Lambda(T_{G^{\sc}}))^{\scr W_G}\cong \Sym^2\Lambda^*(T_{G^{\sc}})^{\scr W_G}$;
\end{enumerate} 
subject to the following compatibility conditions
\begin{enumerate}[(a)]
\item \label{T:Pic-BunGCa} for some (equivalently any)  lift $d^{\ss}\in \Lambda(T_G)$ of the image $\delta^{\ss}$ of $\delta$ in $\pi_1(G^{\ss})$, the direct sum
$$
l_{\scr R}\oplus b(d^{\ss}\otimes -):\Lambda(\scr R(G))\oplus \Lambda(T_{G^{\sc}})\to \bbZ
$$
is integral on $\Lambda(T_G)$;
\item \label{T:Pic-BunGCb} the orthogonal direct sum 
$$
b_{\scr R}\perp (\id_{J_C}\circ b): (\Lambda(\scr R(G))\oplus \Lambda(T_{G^{\sc}})) \otimes (\Lambda(\scr R(G))\oplus \Lambda(T_{G^{\sc}}))\to \End(J_C),
$$
is integral on $\Lambda(T_G)\otimes \Lambda(T_G)$, where $\id_{J_C}:\bbZ\to \End(J_C)$ is the canonical map given by the addition on the abelian variety $J_C$. 
\end{enumerate}
\end{teo}
The (contravariant) functoriality of the exact sequence is described in \cite[Thm. 5.3.1(iv)]{BH10}. Note that if $T$ is a torus, the group in \emph{(iii)} is trivial and the conditions \emph{(a)} and \emph{(b)} are always satisfied. In particular, the projection on the first two factors give an isomorphism of groups
$$\NS(\Bun_T^{\delta}(C))\cong\Lambda^*(T)\oplus \Hom^s(\Lambda(T)\otimes \Lambda(T), \End(J_C)).
$$
and hence the elements $(l_T,b_T,0)$ of $\NS(\Bun_T^{\delta}(C))$ can been seen as pairs $(l_T, b_T)\in \Lambda^*(T)\oplus \Hom^s(\Lambda(T)\otimes \Lambda(T), \End(J_C))$. For the rest of the paper, we adopt the latter presentation for the Neron-Severi group $\NS(\Bun_T^{\delta}(C))$ for the torus case.

In general, the Neron-Severi group $\NS(\Bun_G^{\delta}(C))$ can be described as follows.

\begin{prop}\cite[Prop. 5.2.11]{BH10}\label{P:NS-BunGC}
With the above notation, there is an exact sequence
\begin{equation}\label{E:NS-BunGC}
\begin{aligned}
0\to \NS(\Bun_{G^{\ab}}^{\delta^{\ab}}(C))  \xrightarrow{\ab^{*,\NS}(C)} \NS(\Bun_{G}^{\delta}(C))  &\stackrel{p}{\longrightarrow} \Bil^{s,\ev}(\Lambda(T_{G^{\sc}}))^{\scr W_G}, \\
  (l_{\scr R},b_{\scr R}, b) & \mapsto b \\
\vspace{0.5cm} (l_{G^{\ab}}, b_{G^{\ab}})  \mapsto ((l_{G^{\ab}})_{|\scr R(G)}, (b_{G^{\ab}})_{|\Lambda(\scr R(G))\otimes \Lambda(\scr R(G)))}, 0) \\
\end{aligned}
\end{equation}
 and the image of $p$ is equal to 
 $$
 \Im(p)=
 \begin{cases}
 \Bil^{s, \sc-\ev}(\Lambda(T_{\scr D(G)})\vert \Lambda(T_{G^{\ss}}))^{\scr W_G} & \text{ if } g\geq 1, \\
 \{b\in \Bil^{s,\ev}(\Lambda(T_{G^{\sc}}))^{\scr W_G}\: : b^{\bbQ}(d^{\ss}\otimes -) \: \text{ is integral on } \Lambda(T_{\scr D(G)}) \}& \text{ if } g=0, \\
 \end{cases}
 $$
where $\Bil^{s, \sc-\ev}(\Lambda(T_{\scr D(G)})\vert \Lambda(T_{G^{\ss}}))^{\scr W_G} $ is defined in Definition/Lemma \ref{D:evGtilde} and  $d^{\ss}$ is some (equivalently any) lift to $\Lambda(T_{G^{\ss}})$ of $\delta^{\ss}\in \pi_1(G^{\ss})$. 
\end{prop}

We now describe the restriction homomorphism \eqref{E:reshom}, using the description of $\RPic(\bg{G}^{\delta})$ of Theorem \ref{T:oG+gG} and the description of $\Pic(\Bun_G^{\delta}(C))$ of Theorem \ref{T:Pic-BunGC}.

\begin{teo}\label{T:resPic}
Assume that $g\geq 1$ and let $(C,p_1,\ldots, p_n)\in \Mg(k)$ be a geometric point. Let $G$ be a reductive group 
 and fix $\delta\in \pi_1(G)$. 
Then the restriction homomorphism \eqref{E:reshom}  sits into the following functorial  commutative diagram with exact rows
\begin{equation}\label{E:resPic-seq}
\xymatrix{
0 \ar[r] &  \Lambda^*(G^{\ab})\otimes H_{g,n}  \ar[r]^{j_G^{\delta}} \ar[d]^{\res_G^{\delta}(C)^o}& \RPic(\bg{G}^{\delta})\ar[r]^{\omega_G^{\delta}\oplus \gamma_G^{\delta}} \ar[d]^{\res_G^{\delta}(C)}&  \NS(\bg{G}^{\delta}) \ar[d]^{\res_G^{\delta}(C)^{\NS}}\\
 0 \ar[r] & \Hom(\pi_1(G),J_C(k))\ar[r]^(0.55){j_G^{\delta}(C)} &  \Pic(\Bun_G^{\delta}(C))\ar[r]^{c_G^{\delta}(C)} &  \NS(\Bun_G^{\delta}(C))\ar[r] & 0,
}
\end{equation}
where 
\begin{itemize}
\item $\res_G^{\delta}(C)^o$ sends an element $[\Lambda(G^{\ab})\xrightarrow{f} H_{g.n}]\in  \Hom(\Lambda(G^{\ab}),H_{g,n})=\Lambda^*(G^{\ab})\otimes H_{g,n}$ into the element 
$$[\pi_1(G)\stackrel{\pi_1(\ab)}{\twoheadrightarrow} \pi_1(G^{\ab})=\Lambda(G^{\ab})\xrightarrow{f} H_{g,n}\xrightarrow{\iota_C} J_C(k)]\in \Hom(\pi_1(G),J_C(k)),
$$
where 
\begin{equation*}
\begin{aligned}
\iota_C=\iota_{(C,p_1,\ldots, p_n)}: H_{g,n} & \rightarrow J_C(k) \\
 (m,\zeta) & \mapsto \omega_{C}^m\left(\sum_{i=1}^n \zeta_i p_i\right) &   \text{ if } g\geq 2,\\ 
\zeta & \mapsto   \cO_C\left(\sum_{i=1}^n \zeta_i p_i\right)  & \text{ if } g=1.\\
\end{aligned}
\end{equation*}
\item $\res_G^{\delta}(C)^{\NS}([\chi],b)=(\chi_{|\Lambda(\scr R(G))}, \id_{J_C}\circ b_{|\Lambda(\scr R(G))\otimes \Lambda(\scr R(G))}, b_{|\Lambda(T_{G^{\sc}})\otimes \Lambda(T_{G^{\sc}})})$.
\end{itemize}
\end{teo}
\begin{proof}
The theorem has been proved for a torus  in \cite[Prop. 4.3.3]{FV1}; and,  in order to prove the case of a general reductive group $G$, we are going to use this result for $G^{\ab}$ and for a (fixed) maximal torus $\iota: T_G\hookrightarrow G$.

Observe that the two rows of \eqref{E:resPic-seq}  are exact by Theorems \ref{T:oG+gG} and  \ref{T:Pic-BunGC}. Moreover, the two outer vertical arrows are well-defined: for $\res_G^{\delta}(C)^o$ it is clear; for $\res_G^{\delta}(C)^{\NS}$ it follows from:
\begin{enumerate}[(i)]
\item \label{cond1} given a lift $d\in \Lambda(T_G)$ of $\delta\in \pi_1(G)$, we can choose a representative $\chi^{b(d\otimes -)}\in \Lambda^*(T_G)$ of $[\chi]$ such that $(\chi^{b(d\otimes -)})_{|\Lambda(T_{\scr D(G)})}=b(d,-)_{|\Lambda(T_{\scr D(G)})}$ by \eqref{E:lift-chi} and this implies that  
$$
\chi_{|\Lambda(\scr R(G))}\oplus b(d\otimes -)_{|\Lambda(T_{G^{\sc}})}:\Lambda(\scr R(G))\oplus \Lambda(T_{G^{\sc}})\to \bbZ
$$
is integral on $\Lambda(T_G)$, namely it is the restriction of the chosen $\chi:\Lambda(T_G)\to \bbZ$;
\item \label{cond2} the orthogonal direct sum 
$$
\id_{J_C}\circ (b_{|\Lambda(\scr R(G))\otimes \Lambda(\scr R(G))}\perp b_{|\Lambda(T_{G^{\sc}})\otimes \Lambda(T_{G^{\sc}})}): (\Lambda(\scr R(G))\oplus \Lambda(T_{G^{\sc}})) \otimes (\Lambda(\scr R(G))\oplus \Lambda(T_{G^{\sc}}))\to \End(J_C),
$$
is integral on $\Lambda(T_G)\otimes \Lambda(T_G)$, since it is the restriction of $\id_{J_C}\circ b$ by the claim in the proof of Proposition \ref{P:forms-LTG}.
\end{enumerate}
Hence it remains to prove the commutativity of the two squares in \eqref{E:resPic-seq}.

 In order to prove the commutativity of the left square in \eqref{E:resPic-seq}, consider the following diagram 
 \begin{equation}\label{E:leftdiag}
\xymatrix{
\Lambda^*(G^{\ab})\otimes H_{g,n}   \ar[r]^{j_{G^{\ab}}^{\delta^{\ab}}} \ar@/^2pc/[rr]^{j_{G}^{\delta}}  \ar[d]^{\res_{G^{\ab}}^{\delta^{\ab}}(C)^o} \ar@/_5pc/[dd]_{\res_{G}^{\delta}(C)^o}&  \RPic(\bg{G^{\ab}}^{\delta^{\ab}}) \ar[r]^{\ab_\#^*} \ar[d]^{\res_{G^{\ab}}^{\delta^{\ab}}(C)}&\RPic(\bg{G}^{\delta})\ar[d]^{\res_{G}^{\delta}(C)} \\
   \Hom(\pi_1(G^{\ab}),J_C(k)) \ar[d]^{\Hom(\pi_1(\ab),-)}\ar[r]^(0.55){j_{G^{\ab}}^{\delta^{\ab}}(C)} & \Pic(\Bun_{G^{\ab}}^{\delta^{\ab}}(C))  \ar[r]^{\ab^{*,\NS}(C)} & \Pic(\Bun_G^{\delta}(C)) \\
 \Hom(\pi_1(G),J_C(k)) \ar[urr]_{j_{G}^{\delta}(C)} & & 
} 
\end{equation}
In the above diagram, every simple subdiagram commutes: the curved triangles commutes by the definition of $j_{G}^{\delta}$ and of $\res_{G}^{\delta}(C)^o$; the left square commutes by \cite[Prop. 4.3.3]{FV1} applied to $G^{\ab}$; the right square commutes by the obvious functoriality of the restriction homomorphism; the lower triangle commutes by the functoriality of the morphism $j_{G}^{\delta}$ (see Theorem \ref{T:Pic-BunGC}).  By using all the above commutativity results, we deduce that $\res_{G}^{\delta}(C)\circ j_{G}^{\delta}=j_{G}^{\delta}(C)\circ \res_{G}^{\delta}(C)^o$, i.e. that the left square of \eqref{E:resPic-seq} commutes.  

 In order to prove the commutativity of the right square in \eqref{E:resPic-seq}, choose a lift $d\in \Lambda(T_G)$ of $\delta\in \pi_1(G)$ and consider the following diagram 
 \begin{equation}\label{E:rightdiag}
\xymatrix{
\RPic(\bg{G}^{\delta}) \ar[rr]^{\omega_G^{\delta}\oplus \gamma_G^{\delta}} \ar[dd]^{\res_G^{\delta}(C)} \ar[rd]^{\iota_\#^*} && \NS(\bg{G}^{\delta}) \ar[dd]^(0.3){\res_G^{\delta}(C)^{\NS}} \ar[rd]^{\iota^{*,\NS}} &  \\
& \RPic(\bg{T_G}^{d})\ar[rr]^(0.3){\omega_{T_G}^{d}\oplus \gamma_{T_G}^{d}} \ar[dd]^(0.3){\res_{T_G}^{d}(C)}&& \NS(\bg{T_G}^{d})\ar[dd]^{\res_{T_G}^{d}(C)^{\NS}} \\
\Pic(\Bun_{G}^{\delta}(C)) \ar[rr]^(0.3){c_G^{\delta}(C)} \ar[rd]^{\iota_\#(C)^*}  && \NS(\Bun_{G}^{\delta}(C)) \ar[rd]^{\iota^{*,\NS}(C)}&  \\
& \Pic(\Bun_{T_G}^{d}(C))\ar[rr]^{c_{T_G}^{d}(C)} && \NS(\Bun_{T_G}^{d}(C)) \\
}
\end{equation}
where  $\iota^{*,\NS}$ is the morphism of Definition/Lemma \ref{D:funcNS} and $\iota^{*,\NS}(C)$ is the morphism of \cite[Def. 5.2.5]{BH10}, and they are given by the following formulas
\begin{itemize}
\item $\iota^{*,\NS}([\chi],b)=(\chi^{b(d\otimes -)}, b)\in \NS(\bg{T_G}^d)$,  see \eqref{E:funcNSiota};
\item $\iota^{*,\NS}(C)(l_{\scr R},b_{\scr R}, b)=(l_{\scr R}\oplus b(d^{\ss}\otimes -), b_{\scr R}\perp \id_{J_C}\circ b)\in \NS(\Bun_{T_G}^{d}(C))$, where $d^{\ss}$ is the image of $d\in \Lambda(T_G)$ in $\Lambda(T_{G^{\ss}})$, which is well-defined by conditions \eqref{T:Pic-BunGCa} and \eqref{T:Pic-BunGCb} of Theorem  \ref{T:Pic-BunGC}.
\end{itemize}

We have the following commutativity properties in the above diagram \eqref{E:rightdiag}:
\begin{enumerate}[(a)]
\item $\res_{T_G}^d(C)\circ \iota_\#^*=\iota_\#(C)^*\circ \res_G^{\delta}(C)$, which follows from the functoriality of the restriction homomorphism;
\item $\res_{T_G}^d(C)^{\NS}\circ \iota^{*,\NS}=\iota^{*,\NS}(C)\circ \res_G^{\delta}(C)^{\NS}$.
In fact, for any $([\chi],b)\in \NS(\bg{G}^{\delta})$ we compute 
$$
\begin{sis}
&(\res_{T_G}^d(C)^{\NS}\circ \iota^{*,\NS})([\chi],b)=\res_{T_G}^d(C)^{\NS}(\chi^{b(d\otimes -)},b)= (\chi^{b(d\otimes -)}, \id_{J_C}\circ b), \\
& (\iota^{*,\NS}(C)\circ \res_G^{\delta}(C)^{\NS})([\chi],b) = \iota^{*,\NS}(C)(\chi_{|\Lambda(\scr R(G))}, \id_{J_C}\circ b_{|\Lambda(\scr R(G))\otimes \Lambda(\scr R(G))}, b_{|\Lambda(T_{G^{\sc}})\otimes \Lambda(T_{G^{\sc}})})=\\
 &= (\chi_{|\Lambda(\scr R(G))}\oplus b(d^{\ss}\otimes -)_{|\Lambda(T_{G^{\sc}})}, \id_{J_C}\circ (b_{|\Lambda(\scr R(G))\otimes \Lambda(\scr R(G))}\perp b_{|\Lambda(T_{G^{\sc}})\otimes \Lambda(T_{G^{\sc}})})),
  \end{sis}
  $$
  and the two results coincide by the fact that  $b(d\otimes -)_{|\Lambda(T_{G^{\sc}})}=b(d^{\ss}\otimes -)_{|\Lambda(T_{G^{\sc}})}$ by the claim in the proof of Proposition \ref{P:forms-LTG}, together with condition \eqref{cond2}.
   \item  $c_{T_G}^d(C)\circ \res_{T_G}^d(C)=  \res_{T_G}^d(C)^{\NS}\circ (\omega_{T_G}^d\oplus \gamma_{T_G}^d)$, which follows from Proposition \cite[Prop. 4.3.3]{FV1} applied to $T_G$.
 \item $\iota^{*,\NS}(C)\circ c_G^{\delta}(C)= c_{T_G}^{d}(C)\circ \iota_\#(C)^*$, which follows from the functoriality of the exact sequence \eqref{E:Pic-BunGC}  together with \cite[Def. 5.2.5, 5.2.7]{BH10}.
 \item $\iota^{*,\NS}\circ (\omega_G^{\delta}\oplus \gamma_G^{\delta})=(\omega_{T_G}^d\oplus \gamma_{T_G}^d)\circ \iota_\#^*$, which follows from the functoriality of the homomorphism $\omega_G^{\delta}\oplus \gamma_G^{\delta}$ applied to the inclusion $\iota:T_G\hookrightarrow G$ (see Theorem \ref{T:oG+gG}\eqref{T:oG+gG1}).
 


\end{enumerate}
\vspace{0.1cm}

Using the above commutativity relations together with a diagram chase in \eqref{E:rightdiag}, we conclude that 
$$
\iota^{*,\NS}(C)\circ c_G^{\delta}(C)\circ \res_G^{\delta}(C)=\iota^{*,\NS}(C)\circ \res_G^{\delta}(C)^{\NS}\circ (\omega_G^{\delta}\oplus \gamma_G^{\delta}). 
$$
Since   $\iota^{*,\NS}(C)$ is injective (using that $\id_{J_C}$ is injective since $g\geq 1$), we can simplify $\iota^{*,\NS}(C)$ from the above expression, and we get 
$$
c_G^{\delta}(C)\circ \res_G^{\delta}(C)= \res_G^{\delta}(C)^{\NS}\circ (\omega_G^{\delta}\oplus \gamma_G^{\delta}),
$$
which is exactly  the commutativity of the right square in \eqref{E:resPic-seq}, and this concludes the proof.
\end{proof}

\begin{cor}\label{C:resPic-ker}
Keep the notation of Theorem \ref{T:resPic}.
The kernel of the restriction homomorphism $\res_G^{\delta}(C)$ is equal to 
$$
\ker\left( \res_G^{\delta}(C)\right)=\Lambda^*(G^{\ab})\otimes \ker\iota_C\subseteq  \Lambda^*(G^{\ab})\otimes H_{g,n}.
$$ 
\end{cor}
\begin{proof}
This follows from the snake lemma applied to the exact sequence \eqref{E:resPic-seq}, using that $\res_G^{\delta}(C)^{\NS}$ is injective and the explicit description of $\res_G^{\delta}(C)^o$. 
\end{proof}

We now give an alternative description of the restriction homomorphism onto the Neron-Severi group 
\begin{equation}\label{E:reshomNS}
\ov{\res_G^{\delta}(C)}: \RPic(\bg{G}^{\delta})\xrightarrow{\res_G^{\delta}(C)} \Pic(\Bun_G^{\delta}(C)) \xrightarrow{c_G^{\delta}(C)} \NS(\Bun_G^{\delta}(C)),
\end{equation}
for any $(C,p_1,\ldots, p_n)\in \Mg(k)$, using the description of $\RPic(\bg{G}^{\delta})$ in Corollary \ref{C:Pic-red}  and the description of $\NS(\Bun_G^{\delta}(C))$ in Proposition \ref{P:NS-BunGC}.

\begin{teo}\label{T:resNS}
Assume that $g\geq 1$ and let $(C,p_1,\ldots, p_n)\in \Mg(k)$ be a geometric point. Let $G$ be a reductive group 
 and fix $\delta\in \pi_1(G)$. 
The homomorphism \eqref{E:reshomNS}  sits into the following functorial commutative diagram with exact rows
\begin{equation}\label{E:resNS-seq}
\xymatrix{ 
0 \ar[r]&   \RPic\left(\bg{G^{\ab}}^{\delta^{\ab}}\right) \ar[r]^{\ab_\#^*} \ar[d]^{\ov{\res_{G^{\ab}}^{\delta^{\ab}}(C)}} & \RPic\Big(\bg{G}^{\delta}\Big)  \ar@{->>}[r]^(0.4){\theta_G^{\delta}} \ar[d]^{\ov{\res_G^{\delta}(C)}}&  \Bil^{s,\ev}(\Lambda(T_{\scr D(G)})\vert \Lambda(T_{G^{\ss}}))^{\scr W_G}  \ar[d]^{r_G} \ar[r] & 0,\\
0 \ar[r] &  \NS(\Bun_{G^{\ab}}^{\delta^{\ab}}(C))\ar[r]^(0.55){\ab^{*,\NS}(C)}&  \NS(\Bun_{G}^{\delta}(C)) \ar[r]^(0.35){p} &   \Bil^{s, \sc-\ev}(\Lambda(T_{\scr D(G)})\vert \Lambda(T_{G^{\ss}}))^{\scr W_G}\ar[r] & 0
} 
\end{equation}
where $r_G$ is the injective homomorphism  in Definition/Lemma \ref{D:evGtilde}. 
\end{teo}
\begin{proof}
The rows of \eqref{E:resNS-seq}  are exact by Corollary \ref{C:Pic-red}  and Proposition \ref{P:NS-BunGC}. We need to show the commutativity of the diagram \eqref{E:resNS-seq}.

The commutativity of the left square of \eqref{E:resNS-seq} follows from the functoriality of the restriction homomorphism $\res_G^{\delta}(C)$ and the functoriality of the homomorphism $c_G^{\delta}(C)$ (see Theorem \ref{T:Pic-BunGC}). 

Consider now the following diagram
\begin{equation}\label{E:rightdia}
\xymatrix{
\RPic(\bg{G}^{\delta}) \ar[r]^(0.4){\theta_G^{\delta}} \ar[d]^{\omega_G^{\delta}\oplus \gamma_G^{\delta}} \ar@/_3pc/[dd]_{\ov{\res_G^{\delta}(C)}}& \Bil^{s,\ev}(\Lambda(T_{\scr D(G)})\vert \Lambda(T_{G^{\ss}}))^{\scr W_G}\ar@{=}[d] \\
\NS(\bg{G}^{\delta}) \ar[r]^(0.4){\res_{\scr D}^{\NS}} \ar[d]^{\res_G^{\delta}(C)^{\NS}} & \Bil^{s,\ev}(\Lambda(T_{\scr D(G)})\vert \Lambda(T_{G^{\ss}}))^{\scr W_G}\ar[d]^{r_G}\\
\NS(\Bun_G^{\delta}(C))  \ar[r]^(0.3){p} & \Bil^{s}(\Lambda(T_{\scr D(G)})\vert \Lambda(T_{G^{\ss}}))^{\scr W_G}\cap \Bil^{s,\ev}(\Lambda(T_{G^{\sc}}))^{\scr W_G}
}
\end{equation}
The commutativity of the right square of \eqref{E:resNS-seq} follows from the following commutativity results:
\begin{enumerate}[(a)]
\item The left triangle commutes because of the definition \eqref{E:reshomNS} of $\ov{\res_G^{\delta}(C)}$ and of the commutativity of the diagram \eqref{E:resPic-seq}.
\item The upper square commutes, i.e. $\theta_G^{\delta}=\res_{\scr D}^{\NS}\circ (\omega_G^{\delta}\oplus \gamma_G^{\delta})$.

In order to prove this,   it is enough to show, using that  $\RPic(\bg{G}^{\delta})$ is the pushout \eqref{E:amalgINT}, that 
 \begin{equation}\label{condb1}
\theta_G^{\delta}\circ \ab_\#^*=\res_{\scr D}^{\NS}\circ (\omega_G^{\delta}\oplus \gamma_G^{\delta})\circ \ab_\#^*,
 \end{equation}
 \begin{equation}\label{condb2}
 \theta_G^{\delta}\circ \tau_G^{\delta} =\res_{\scr D}^{\NS}\circ (\omega_G^{\delta}\oplus \gamma_G^{\delta})\circ \tau_G^{\delta}.
 \end{equation}
Condition \eqref{condb1} holds since $\theta_G^{\delta}\circ \ab_\#^*=0$ by \eqref{E:seqPic-ab} and $\res_{\scr D}^{\NS}\circ (\omega_G^{\delta}\oplus \gamma_G^{\delta})\circ \ab_\#^*=0$ which follows from 
the fact that $ (\omega_G^{\delta}\oplus \gamma_G^{\delta})\circ \ab_\#^*=\ab^{*,\NS}\circ (\omega_{G^{\ab}}^{\delta^{\ab}}\oplus \gamma_{G^{\ab}}^{\delta^{\ab}})$ (by the functoriality of the morphism  $\omega_G^{\delta}\oplus \gamma_G^{\delta}$, see Theorem \ref{T:oG+gG}\eqref{T:oG+gG1}) together with \eqref{E:seqNS}. Condition \eqref{condb2} follows from  
$$
\begin{sis} 
& (\res_{\scr D}^{\NS}\circ (\omega_G^{\delta}\oplus \gamma_G^{\delta})\circ \tau_G^{\delta})(b)=\res_{\scr D}^{\NS}(b(\delta\otimes -),b)=b_{|\Lambda(T_{\scr D(G)})\otimes \Lambda(T_{\scr D(G)})} \quad \text{ by } \eqref{E:oGgG-tras},\\
& (\theta_G^{\delta}\circ \tau_G^{\delta})(b)= b_{|\Lambda(T_{\scr D(G)})\otimes \Lambda(T_{\scr D(G)})} \quad \text{ by Corollary } \ref{C:Pic-red}.
\end{sis}
$$ 
\item The lower square commutes because, using the definitions of the maps involved, we have
$$
(r_G\circ \res_{\scr D}^{\NS})([\chi],b)=b_{|\Lambda(T_{G^{\sc}})\otimes \Lambda(T_{G^{\sc}})}=(p\circ \res_G^{\delta}(C)^{\NS})([\chi],b).
$$
\end{enumerate}

\end{proof}

\begin{cor}\label{C:resNS-coker}
Keep the notation of Theorem \ref{T:resNS} and assume furthermore that $\id_{J_C}: \bbZ\xrightarrow{\cong} \End(J_C)$ is an isomorphism. Then there exists a canonical short exact sequence
\begin{equation}\label{E:resNS-coker}
0\to \coker(\omega_{G^{\ab}}^{\delta^{\ab}}\oplus \gamma_{G^{\ab}}^{\delta^{\ab}})\to \coker(\ov{\res_G^{\delta}(C)})\to  \coker(r_G)\to 0. 
\end{equation}
In particular, if $n>0$ then we have the canonical isomorphism 
$$
\coker(\ov{\res_G^{\delta}(C)})\xrightarrow{\cong}  \coker(r_G).
$$
\end{cor}
Note that if $k$ is uncountable and $C$ is  very general in  ${\mathcal M}_g(k)$, then $\id_{J_C}: \bbZ\xrightarrow{\cong} \End(J_C)$ is an isomorphism, as it follows easily from \cite{Koi}.
\begin{proof}
By applying the snake lemma to \eqref{E:resNS-seq} and using that $r_G$ is injective, we get the exact sequence 
\begin{equation}\label{E:coker-seq}
0\to \coker(\ov{\res_{G^{\ab}}^{\delta^{\ab}}(C))})\to \coker(\ov{\res_G^{\delta}(C)})\to  \coker(r_G)\to 0. 
\end{equation}
By Theorem  \ref{T:resPic},  the morphism $\ov{\res_{G^{\ab}}^{\delta^{\ab}}(C))}$ admits the following factorization
$$
\ov{\res_{G^{\ab}}^{\delta^{\ab}}(C)}: \RPic(\bg{G^{\ab}}^{\delta^{\ab}})\xrightarrow{\omega_{G^{\ab}}^{\delta^{\ab}}\oplus \gamma_{G^{\ab}}^{\delta^{\ab}}} \NS(\bg{G^{\ab}}^{\delta^{\ab}})\xrightarrow{\res_{G^{\ab}}^{\delta^{\ab}}(C)^{\NS}} 
 \NS(\Bun_{G^{\ab}}^{\delta^{\ab}}(C)).
$$
By the assumption that $\id_{J_C}$ is an isomorphism together with the explicit description of $\res_{G^{\ab}}^{\delta^{\ab}}(C)^{\NS}$ in Theorem \ref{T:resPic}, we deduce that  $\res_{G^{\ab}}^{\delta^{\ab}}(C)^{\NS}$ is an isomorphism. 
Hence, we get a canonical isomorphism 
\begin{equation}\label{E:iso-coker}
\coker(\omega_{G^{\ab}}^{\delta^{\ab}}\oplus \gamma_{G^{\ab}}^{\delta^{\ab}}) \xrightarrow{\cong} \coker(\ov{\res_{G^{\ab}}^{\delta^{\ab}}(C))}).
\end{equation}
By combining the short exact sequence \eqref{E:coker-seq} with the isomorphism \eqref{E:iso-coker}, we conclude.  
\end{proof}

We end this section by describing the restriction homomorphism \eqref{E:reshom} in genus $0$. 

\begin{rmk}\label{R:res-g0}In genus $0$, the only smooth curve is the projective line $\mathbb{P}^1$. In this situation, it follows from the proof of \cite[Prop. 5.4.1]{FV1} that the restriction homomorphism 
\begin{equation*}
\res_G^{\delta}(\mathbb P^1): \RPic(\mathrm{Bun}_{G,0,n}^{\delta})\to \Pic(\Bun_G^{\delta}(\mathbb P^1))
\end{equation*}
is an isomorphism if $n>0$ and injective if $n=0$, for any $(\mathbb P^1,p_1,\ldots, p_n)\in \cM_{0,n}(k)$.

Using Theorem \ref{T:Pic-BunGC}, it follows that if $n>0$ then we have  
\begin{equation}\label{E:RPic-g0}
 \RPic(\mathrm{Bun}_{G,0,n}^{\delta})=\left\{(l_{\scr R}, b)\in \Lambda^*(\scr R(G))\times  \Bil^{s,\ev}(\Lambda(T_{G^{\sc}}))^{\scr W_G}\: : l_{\scr R}\oplus b(d^{\ss}\otimes -) \in \Lambda^*(T_G)\right\},
\end{equation}
for some (equivalently any)  lift $d^{\ss}\in \Lambda(T_G)$ of the image $\delta^{\ss}$ of $\delta$ in $\pi_1(G^{\ss})$.

Moreover, using Proposition \ref{P:NS-BunGC}, we have the following exact sequence for $n>0$  
\begin{equation}\label{E:seqRPic-g0}
\begin{aligned}
0\to  \RPic(\mathrm{Bun}_{G^{\ab},0,n}^{\delta^{\ab}}) \xrightarrow{\ab_\#^*} \RPic(\mathrm{Bun}_{G,0,n}^{\delta})  &\xrightarrow{\theta_G^{\delta}}  
\left\{
\begin{aligned}
&b\in \Bil^{s,\ev}(\Lambda(T_{G^{\sc}}))^{\scr W_G}\: : \\ 
&b^{\bbQ}(d^{\ss}\otimes -) \: \text{ is integral on } \Lambda(T_{\scr D(G)}) 
\end{aligned}
\right\}\to 0, \\
  (l_{\scr R}, b) & \mapsto b \\
\vspace{0.5cm} l_{G^{\ab}}  \mapsto ((l_{G^{\ab}})_{|\scr R(G)},  0) \\
\end{aligned}
\end{equation}
\end{rmk}

\section{The rigidification $\bg{G}^{\delta}\fatslash \scr Z(G)$ and its Picard group}\label{S:rig}


Since the center $\scr Z(G)$ of a reductive group $G$ acts functorially on any $G$-bundle, we have that $\scr Z(G)$  sits functorially inside the automorphism group of any $S$-point $(\cC\to S, \un \sigma, E)$ of $\bg{G}^{\delta}$ for any $\delta\in \pi_1(G)$. Hence we can form the rigidification 
\begin{equation}\label{E:rigid}
\nu_G^{\delta}:\bg{G}^{\delta} \to \bg{G}^{\delta}\fatslash \scr Z(G):=\bgr{G}^{\delta},
\end{equation}
which turns out to be a $\scr Z(G)$-gerbe, i.e. a gerbe banded by $\scr Z(G)$.
The aim of this section is to study the Picard group of $\bgr{G}^{\delta}$ and the class of the $\scr Z(G)$-gerbe $\nu_G^{\delta}$. 

From the Leray spectral sequence 
$$
E_{p,q}^2=H^p(\bgr{G}^{\delta}, R^q(\nu_{G}^{\delta})_*(\Gm))\Rightarrow H^{p+q}(\bg{G}^{\delta}, \Gm), 
$$
and using that $(\nu_G^{\delta})_*(\Gm)=\Gm$ and that $R^1(\nu_{G}^{\delta})_*(\Gm)$ is the constant sheaf $\Lambda^*(\scr Z(G))=\Hom(\scr Z(G),\Gm)$, we get the exact sequence 
\begin{equation}\label{E:seqLeray}
 \Pic(\bgr{G}^{\delta})\stackrel{(\nu_G^{\delta})^*}{\hookrightarrow} \Pic(\bg{G}^{\delta}) \xrightarrow{\w_G^{\delta}} \Lambda^*(\scr Z(G))  \xrightarrow{\obs_G^{\delta}} H^2(\bgr{G}^{\delta}, \Gm)\xrightarrow{(\nu_G^{\delta})^*}  H^2(\bg{G}^{\delta}, \Gm).
\end{equation}
The homomorphism $\w_G^{\delta}$ (called the \emph{weight homomorphism}) has the following geometric interpretation: given a line bundle $\cL$ on $\bg{G}^{\delta}$, the character $\w_G^{\delta}(\cL)\in \Hom(\scr Z(G),\Gm)$ is such that, for any  $\cE:=(\cC\to S, \un \sigma, E)\in \bg{G}^{\delta}(S)$, we have the factorization
$$
\w_G^{\delta}(S):\scr Z(G)(S) \hookrightarrow \Aut_{\bg{G}^{\delta}(S)}(\cE)\xrightarrow{\Aut(\cL_{S})} \Aut_{\cO_S}(\cL_S(\cE))=\Gm(S),
$$
where the first homomorphism is given by the canonical action of $\scr Z(G)$ on every $G$-gerbe, and the second homomorphism is induced by the functor of groupoids 
$$\cL_S:\bg{G}^{\delta}(S)\to \{\text{Line bundles on } S\}$$ 
determined by $\cL$. 

The homomorphism $\obs_G^{\delta}$ (called the \emph{obstruction homomorphism}) has the following geometric interpretation: given any character $\lambda\in \Hom(\scr Z(G),\Gm)$, the element $\obs_G^{\delta}(\lambda)$ is the class in $H^2(\bgr{G}^{\delta}, \Gm)$ of the $\Gm$-gerbe $\lambda_*(\nu_G^{\delta})$ obtained by pushing forward the $\scr Z(G)$-gerbe $\nu_G^{\delta}$ along $\lambda$. 

If we take the fiber of \eqref{E:rigid} over a geometric point $(C,p_1,\ldots, p_n)\in \Mg(k)$, we get the $\scr Z(G)$-gerbe 
\begin{equation}\label{E:rigidC}
\nu_G^{\delta}(C): \Bun_G^{\delta}(C)\to \Bunr_G^{\delta}(C):=\Bun_G^{\delta}(C)\fatslash \scr Z(G).
\end{equation}
The Leray spectral sequence for $\Gm$ relative to the $\scr Z(G)$-gerbe gives the exact sequence (analogously to \eqref{E:seqLeray})
\begin{equation}\label{E:seqLerayC}
 \Pic(\Bunr_G^{\delta}(C))\stackrel{\nu_G^{\delta}(C)^*}{\hookrightarrow} \Pic(\Bun_G^{\delta}(C)) \xrightarrow{\w_G^{\delta}(C)} \Lambda^*(\scr Z(G))  \xrightarrow{\obs_G^{\delta}(C)} H^2(\Bunr_G^{\delta}(C), \Gm).
\end{equation}
The weight homomorphism $\w_G^{\delta}(C)$ and its cokernel, which coincides with the image of $\obs_G^{\delta}(C)$ by \eqref{E:rigidC}, have been determined by Biswas-Hoffmann \cite[Prop. 7.2]{BH12}, as we now recall 
(in a form which is slightly different from \cite{BH12}).

\begin{teo}\label{T:weightC}(\cite{BH12})
Let $C$ be a (irreducible, projective, smooth) curve of genus  $g\geq 1$ over $\ov k=k$.  Let $G$ be a reductive group $G$ and fix $\delta\in \pi_1(G)$. 
\begin{enumerate}
\item \label{T:weightC1} The weight homomorphism $\w_G^{\delta}(C)$ factors as 
\begin{equation}\label{E:w-factT}
\begin{aligned}
\w_G^{\delta}(C):\Pic(\Bun_G^{\delta}(C))\xrightarrow{c_G^{\delta}(C)} \NS(\Bun_G^{\delta}(C)) & \xrightarrow{\ov{\w}_G^{\delta}(C)} \frac{\Lambda^*(T_G)}{\Lambda^*(T_{G^{\ad}})}\\
(l_{\scr R}, b_{\scr R}, b) & \mapsto [l_{\scr R}\oplus b(d^{\ss}\otimes -)],
\end{aligned}
\end{equation}
where $c_G^{\delta}(C)$ is the homomorphism of Theorem \ref{T:Pic-BunGC} and $d^{\ss} \in \Lambda(T_{G^{\ss}})$ is any lift of the image $\delta^{\ss}$ of $\delta$ in $\pi_1(G^{\ss})$.
\item \label{T:weightC2} The homomorphism $\ov{\Lambda_{\scr D}^*}$ of \eqref{E:mult-car} induces an isomorphism 
 \begin{equation}\label{E:coker-omC}
 \wt{\Lambda_{\scr D}^*}:\coker(\w_G^{\delta}(C))\xrightarrow{\cong} \coker(\wt\ev_{\scr D(G)}^{\delta}),
 \end{equation}
where $\wt \ev_{\scr D(G)}^{\delta}$ is the homomorphism of Definition/Lemma \ref{D:evGtilde}.  
\end{enumerate}
\end{teo}
\begin{proof}
Part \eqref{T:weightC1}: note that $\ov{\w}_G^{\delta}(C)$ is well-defined since $l_{\scr R}\oplus b(d^{\ss}\otimes -)$ is integral on $\Lambda(T_G)$ by condition \eqref{T:Pic-BunGCa} of Theorem \ref{T:Pic-BunGC} and its class in $\frac{\Lambda^*(T_G)}{\Lambda^*(T_{G^{\ad}})}$ is independent of the choice of the lift $d^{\ss}$ by Lemma \ref{L:intGad}. 
The equality $\w_G^{\delta}(C)=\ov\w_G^{\delta}(C)\circ c_G^{\delta}(C)$ is a consequence of the following factorization (see \cite[Eq. (5)]{BH12})
\begin{equation}\label{E:w-fact2}
\w_G^{\delta}(C):\Pic(\Bun_G^{\delta}(C))\xrightarrow{c_G^{\delta}(C)} \NS(\Bun_G^{\delta}(C)) \xrightarrow{\iota^{*,\NS}(C)} \NS(\Bun_{T_G}^{d}(C)) \xrightarrow{p_1}\Lambda^*(T_G) \twoheadrightarrow \frac{\Lambda^*(T_G)}{\Lambda^*(T_{G^{\ad}})}
\end{equation}
where $c_G^{\delta}(C)$ is the homomorphism of Theorem \ref{T:Pic-BunGC}, $d\in \Lambda(T_G)$ is any lift of $\delta\in \pi_1(G)$,  $\iota^{*,\NS}(C)$ is the homomorphism defined in the second bullet below the diagram \eqref{E:rightdiag}, 
$p_1:\NS(\Bun_{T_G}^d(C))\to \Lambda^*(T_G)$ is the projection onto the first factor.

Part \eqref{T:weightC2}: from part \eqref{T:weightC1} and using that $c_G^{\delta}(C)$ is surjective by Theorem \ref{T:Pic-BunGC}, we get that $\coker(\w_G^{\delta}(C))=\coker(\ov\w_G^{\delta}(C))$. 
From the definition of $\ov\w_G^{\delta}(C)$, using Proposition \ref{P:NS-BunGC}  and the vertical exact sequence of \eqref{E:mult-car}, we get the commutative diagram with exact rows
\begin{equation}\label{E:w-NSC}
\xymatrix{
0\ar[r] &  \NS(\bg{G^{\ab}}^{\delta^{\ab}}(C))  \ar[r]^{\ab^{*,\NS}(C)} \ar[d]^{\ov{\w}_{G^{\ab}}^{\delta^{\ab}}(C)} & \NS(\bg{G}^{\delta}(C))  \ar[r]^(0,4){p} \ar[d]^{\ov{\w}_G^{\delta}(C)}&   \Bil^{s, \sc-\ev}(\Lambda(T_{\scr D(G)})\vert \Lambda(T_{G^{\ss}}))^{\scr W_G} \ar[r] \ar[d]^{\wt \ev_{\scr D(G)}^{\delta}}& 0 \\
0\ar[r] & \Lambda^*(G^{\ab}) \ar[r]^{\ov{\Lambda_{\ab}^*}}& \frac{\Lambda^*(T_G)}{\Lambda^*(T_{G^{\ad}})} \ar[r]^{\ov{\Lambda_{\scr D}^*}}&  \frac{\Lambda^*(T_{\scr D(G)})}{\Lambda^*(T_{G^{\ad}})} \ar[r] & 0
}
\end{equation}
From the definition, it follows that $\ov{\w}_{G^{\ab}}^{\delta^{\ab}}(C)$ is the projection onto the first factor of $ \NS(\bg{G^{\ab}}^{\delta^{\ab}}(C))$, and hence it is surjective. 
Therefore the snake lemma applied to \eqref{E:w-NSC} gives the conclusion. 
\end{proof}

Using the above result, we can now give an explicit expression for the weight homomorphism $\w_G^{\delta}$. 

\begin{prop}\label{P:weight}
Assume that $g\geq 1$. Fix a maximal torus  $\iota: T_G\hookrightarrow G$ of the reductive group $G$ and let $\delta\in \pi_1(G)$. 
The weight homomorphism $\w_G^{\delta}$ is equal to the following composition
$$
\w_G^{\delta}:\Pic(\bg{G}^{\delta})\twoheadrightarrow \RPic(\bg{G}^{\delta})\xrightarrow{\omega_G^{\delta}} \Lambda^*(\scr Z(G))=\frac{\Lambda^*(T_G)}{\Lambda^*(T_{G^{\ad}})}, 
$$
 $\omega_G^{\delta}$ is the homomorphism of Definition/Lemma \ref{D:wtG}. 
\end{prop}
\begin{proof}
From the geometric description of the weight homomorphism, it follows that $\w_G^{\delta}$ can be computed on any fiber of $\Phi_G^{\delta}$ over a geometric point $(C,p_1,\ldots, p_n)\in \Mg(k)$, i.e. $\w_G^{\delta}$ factors as the composition  
\begin{equation}\label{E:w-fact1}
\w_G^{\delta}:\Pic(\bg{G}^{\delta})\twoheadrightarrow \RPic(\bg{G}^{\delta})\xrightarrow{\res_G^{\delta}(C)} \Pic(\Bun_G^{\delta}(C))\xrightarrow{\w_G^{\delta}(C)} \Lambda^*(\scr Z(G)),
\end{equation}
where $\res_G^{\delta}(C)$ is the restriction homomorphism \eqref{E:reshom} and $\w_G^{\delta}(C)$ is the weight homomorphism of \eqref{E:seqLerayC}. 



From Theorems \ref{T:resPic} and \ref{T:weightC},  it follows that the composition $\w_G^{\delta}(C)\circ \res_G^{\delta}(C)$ is equal to the morphism $\omega_G^{\delta}\oplus \gamma_G^{\delta}:\RPic(\bg{G}^{\delta})\to \NS(\bg{G}^{\delta})$ followed the composition
\begin{equation}\label{E:w-fact3}
\NS(\bg{G}^{\delta})\xrightarrow{\res_G^{\delta}(C)^{\NS}} \NS(\Bun_G^{\delta}(C)) \xrightarrow{\ov\w_G^{\delta}(C)} \frac{\Lambda^*(T_G)}{\Lambda^*(T_{G^{\ad}})},
\end{equation}
 which sends an element $([\chi],b)\in \NS(\bg{G}^{\delta})$ into 
 $$
 \left[\chi_{|\scr R(G)}\oplus b(d\otimes-)_{|\Lambda(T_{G^{\sc}})} \right]=[\chi]\in \frac{\Lambda^*(T_G)}{\Lambda^*(T_{G^{\ad}})}. 
 $$
 Therefore, the composition $\w_G^{\delta}(C)\circ \res_G^{\delta}(C)$ is equal to the morphism $\omega_G^{\delta}$, and we conclude by \eqref{E:w-fact1}.   
\end{proof}

Note that the forgetful morphism $\Phi_G^{\delta}:\bg{G}^{\delta}\to \Mg$ factors as 
$$
\Phi_G^{\delta}: \bg{G}^{\delta} \xrightarrow{\nu_G^{\delta}} \bgr{G}^{\delta}\xrightarrow{\Psi_G^{\delta}} \Mg, 
$$
where the morphism $\Psi_G^{\delta}$ is fpqc (i.e. faithfully flat and locally quasi-compact) and cohomologically flat in degree zero (i.e. the natural morphism $(\Psi_G^{\delta})^{\sharp}:\cO_{\Mg} \to  (\Psi_G^{\delta})_*(\cO_{\bgr{G}^{\delta}})$ is a universal isomorphism) since $\Phi_G^{\delta}$ satisfies these properties (see Theorem \ref{T:propBunG}\eqref{T:propBunG2})
and $\nu_G^{\delta}$ is smooth, surjective, and cohomologically flat in degree zero.

 Hence, the pull-back morphism $\Psi_G^{\delta}$ in injective on Picard groups and we have an injective  pull-back morphism on the relative Picard groups 
\begin{equation}\label{E:nubar}
\ov{(\nu_G^{\delta})^*}:\RPic(\bgr{G}^{\delta}):=\frac{\Pic(\bgr{G}^{\delta})}{(\Psi_G^{\delta})^*(\Pic(\Mg))}\hookrightarrow \RPic(\bg{G}^{\delta}):=\frac{\Pic(\bg{G}^{\delta})}{(\Phi_G^{\delta})^*(\Pic(\Mg))}.
\end{equation}
Therefore, combining \eqref{E:seqLeray} with Proposition \ref{P:weight},  
we get the new exact sequence  
\begin{equation}\label{E:seqLeray2}
\RPic(\bgr{G}^{\delta})\stackrel{\ov{(\nu_G^{\delta})^*}}{\hookrightarrow} \RPic(\bg{G}^{\delta}) \xrightarrow{\omega_G^{\delta}} \Lambda^*(\scr Z(G))  \xrightarrow{\obs_G^{\delta}} H^2(\bgr{G}^{\delta}, \Gm)\xrightarrow{(\nu_G^{\delta})^*}  H^2(\bg{G}^{\delta}, \Gm).
\end{equation}

We now want to compute the kernel of $\omega_G^{\delta}$, which can be identified with the Picard group of $\bgr{G}^{\delta}$, and the cokernel of $\omega_G^{\delta}$, which is an obstruction to the vanishing of the obstruction morphism. 
With this aim, we introduce the following group.

\begin{defin}\label{D:NS-rig}
Let $G$ be a reductive group with fixed maximal torus $T_G$ and Weyl group $\scr W_G$, and fix $\delta\in \pi_1(G)$. 
Denote by 
$$\NS(\bgr{G}^{\delta})\subset \Bil^{s,\scr D-\ev}(\Lambda(T_G))^{\scr W_G}$$
the subgroup consisting of those $b\in \Bil^{s,\scr D-\ev}(\Lambda(T_G))^{\scr W_G}$ such that 
\begin{equation}\label{E:NS-rig}
0=(\ev_{\scr D(G)}^{\delta} \circ \res_{\scr D})(b)=b(\delta\otimes -)_{|\Lambda(T_{\scr D(G)})}:=\left[b(d\otimes -)_{|\Lambda(T_{\scr D(G)})} \right]  \in  \frac{\Lambda^*(T_{\scr D(G)})}{\Lambda^*(T_{G^{\ad}})},
\end{equation}
where $d\in \Lambda(T_G)$ is any lift of $\delta$, and $\ev_{\scr D(G)}^{\delta}$ and $\res_{\scr D}$ are the homomorphisms of Definition/Lemma \ref{D:evG}.
\end{defin}

The relation between the groups $\NS(\bgr{G}^{\delta})$ and $\NS(\bg{G}^{\delta})$ is explained in the following

\begin{prop}\label{P:seqNSrig}
Let $G$ be a reductive group with maximal torus $T_G$ and Weyl group $\scr W_G$, and fix $\delta\in \pi_1(G)$.
We have the following commutative diagram, with exact rows and columns
\begin{equation}\label{E:seqNSrig}
\xymatrix{
 \Bil^s(\Lambda(G^{\ab}))= \NS(\bgr{G^{\ab}}^{\delta^{\ab}})\ar@{^{(}->}[r]^(0.6){\wt{{\ab}^{*,\NS}}}  \ar@{^{(}->}[d]^{\nu_{G^{\ab}}^{\delta^{\ab},\NS}}& \NS(\bgr{G}^{\delta})\ar@{^{(}->}[d]^{\nu_G^{\delta, \NS}} \ar@{->>}[r]^(0.5){\wt{\res_{\scr D}^{\NS}}} & \ker((\ev_{\scr D(G)}^{\delta}) ) \ar@{^{(}->}[d]  \\
 \Lambda^*(G^{\ab})\oplus \Bil^s(\Lambda(G^{\ab}))= \NS(\bg{G^{\ab}}^{\delta^{\ab}})\ar@{^{(}->}[r]^(0.75){{\ab}^{*,\NS}} \ar@{->>}[d]^{\omega_{G^{\ab}}^{\delta^{\ab},\NS}}& \NS(\bg{G}^{\delta}) \ar@{->>}[r]^(0.35){\res_{\scr D}^{\NS}}  \ar[d]^{\omega_G^{\delta,\NS}} & \Bil^{s,\ev}(\Lambda(T_{\scr D(G)})\vert \Lambda(T_{G^{\ss}}))^{\scr W_G}  \ar[d]^{\ev_{\scr D(G)}^{\delta}}\\
 \Lambda^*(G^{\ab}) \ar@{^{(}->}[r]^{\ov{\Lambda_{\ab}^*}}& \frac{\Lambda^*(T_G)}{\Lambda^*(T_{G^{\ad}})} \ar@{->>}[r]^{\ov{\Lambda_{\scr D}^*}}\ar@{->>}[d] &  \frac{\Lambda^*(T_{\scr D(G)})}{\Lambda^*(T_{G^{\ad}})} \ar@{->>}[d]\\
 & \coker(\omega_G^{\delta,\NS}) \ar[r]^{\cong} & \coker(\ev_{\scr D(G)}^{\delta}) 
}
\end{equation}
where the identification $\NS(\bgr{G^{\ab}}^{\delta^{\ab}})=\Bil^s(\Lambda(G^{\ab}))$ follows from Definition \ref{D:NS-rig},  the morphisms $\wt{{\ab}^{*,\NS}}$ and  $\wt{\res_{\scr D}^{\NS}}$  are the  restrictions of, respectively,  the morphisms $\ab^{*,\NS}$ and $\res_{\scr D}^{\NS}$ (which are defined in Proposition \ref{P:seqNS}),   the morphisms $\nu_{G^{\ab}}^{\delta^{\ab},\NS}$ and $\nu_G^{\delta, \NS}$ are induced by the inclusions onto the second  factors, the morphisms $\omega_{G^{\ab}}^{\delta^{\ab},\NS}$ and $\omega_G^{\delta,\NS}$ are induced by the projections onto the first factors.  
\end{prop}
\begin{proof}
The second row is exact by Proposition \ref{P:seqNS}, and the third row is exact by  \eqref{E:mult-car}.  Moreover, the second and  third rows commute: the equality $\omega_G^{{\delta},\NS}\circ \ab^{*,\NS}=\ov{\Lambda_{\ab}^*}\circ \omega_{G^{\ab}}^{\delta^{\ab},\NS}$ follows from the fact that $\ab^{*,\NS}=\ov{\Lambda_{\ab}^*}\oplus B_{\ab}^*$ (as it can be deduced from Definition/Lemma \ref{D:funcNS}), while the equality  $\ev_{\scr D(G)}^{\delta}\circ \res_{\scr D}^{\NS}= \ov{\Lambda_{\scr D}^*}\circ \omega_G^{\delta,\NS}$ is exactly the condition \eqref{E:NS-BunG}.  
We now conclude by applying the snake lemma, and using that the kernel of $\omega_G^{\delta,\NS}$ (resp. $\omega_{G^{\ab}}^{\delta^{\ab},\NS}$) is equal to the image of $\nu_G^{\delta,\NS}$ (resp. $\nu_{G^{\ab}}^{\delta^{\ab},\NS}$) 
and that $\omega_{G^{\ab}}^{\delta^{\ab},\NS}$ is surjective. 
\end{proof}

We can now give an explicit description of $\RPic(\bgr{G}^{\delta})$, which, via the morphism $\ov{(\nu_G^{\delta})^*}$, can be identified with  $\ker(\omega_G^{\delta})\subseteq \RPic(\bg{G}^{\delta})$, see \eqref{E:seqLeray2}.

\begin{teo}\label{T:Pic-rig}
Assume that $g\geq 1$. Let $G$ be a reductive group with maximal torus $T_G$ and Weyl group $\scr W_G$, and fix $\delta\in \pi_1(G)$. 
\begin{enumerate}
\item \label{T:Pic-rig1} There is an exact sequence 
\begin{equation}\label{E:Pic-rig1}
0 \to \Lambda^*(G^{\ab})\otimes H_{g,n}  \xrightarrow{\ov{j_G^{\delta}}} \RPic(\bgr{G}^{\delta})\xrightarrow{\ov{\gamma_G^{\delta}}} \NS(\bgr{G}^{\delta}),
\end{equation}
where $\ov{j_G^{\delta}}$ and $\ov{\gamma_G^{\delta}}$ are uniquely determined by 
$$\ov{(\nu_G^{\delta})^*}\circ \ov{j_G^{\delta}}=j_G^{\delta}\quad \text{ and } \quad \nu_G^{\delta,\NS} \circ \ov{\gamma_G^{\delta}}= (\omega_G^{\delta}\oplus \gamma_G^{\delta})\circ \ov{(\nu_G^{\delta})^*},$$
see Theorem \ref{T:oG+gG}\eqref{T:oG+gG1} and Proposition \ref{P:seqNSrig}. 


\item \label{T:Pic-rig2} The image of $\ov{\gamma_G^{\delta}}$ is equal to 
\begin{equation}\label{E:Pic-rig2}
 \Im(\ov{\gamma_G^{\delta}})=
 \begin{cases} 
\NS(\bgr{G}^{\delta}) & \text{ if } n\geq 1,\\
  \left\{b\in \NS(\bgr{G}^{\delta})\: : \: 
 \begin{aligned}
 & b(\delta\otimes x)+(g-1)b(x\otimes x)   \text{ is a multiple of } 2g-2, \\
&  \text{ for any } x\in \Lambda(T_G)
 \end{aligned}
 \right\} & \text{ if } n=0. 
\end{cases} 
 \end{equation}
 \end{enumerate}
\end{teo}

\begin{rmk}\label{R:divis}
For $n=0$, the divisibility condition in the right hand side  \eqref{E:Pic-rig2} depends only on the image $x^{\ab}$ of $x\in \Lambda(T_G)$ in $ \Lambda(G^{\ab})$. Indeed, $b(\delta\otimes x)=b^{\bbQ}(\delta\otimes x^{\ab})$ by \eqref{E:NS-rig}, while the parity of $b(x\otimes x)$ depends solely on $x^{\ab}$ since $b_{|\Lambda(T_{\scr D(G)})\otimes \Lambda(T_{\scr D(G)})}$ is even by definition and $b^{\bbQ}_{|\Lambda(T_{\scr D(G)})\otimes \Lambda(G^{\ab})}\equiv 0$ by the claim in the proof of Proposition \ref{P:forms-LTG}.   
\end{rmk}

\begin{proof}Consider the following commutative diagram (of solid arrows)
\begin{equation}\label{E:diag-Picrig}
\xymatrix{
& \RPic(\bgr{G}^{\delta})\ar@{^{(}->}[d]^{\ov{(\nu_G^{\delta})^*}}\ar@{-->}[r]^{\ov{\gamma_G^{\delta}}} &  \NS(\bgr{G}^{\delta}) \ar@{^{(}->}[d]^{\nu_G^{\delta,\NS}} \\
 \Lambda^*(G^{\ab})\otimes H_{g,n}  \ar@{^{(}->}[r]^{j_G^{\delta}}\ar@{-->}[ru]^{\ov{j_G^{\delta}}}&  \RPic(\bg{G}^{\delta})\ar[r]^{\omega_G^{\delta}\oplus \gamma_G^{\delta}} \ar@{->>}[d]^{\omega_G^{\delta}} & \NS(\bg{G}^{\delta}) \ar@{->>}[d]^{\omega_G^{\delta,\NS}}\\
 & \Lambda^*(\scr Z(G))  \ar@{=}[r] &   \frac{\Lambda^*(T_G)}{\Lambda^*(T_{G^{\ad}})}
 }
\end{equation}
where the horizontal row is exact by Theorem \ref{T:oG+gG}\eqref{T:oG+gG1}, the  left column is exact by \eqref{E:seqLeray2} and the right column is exact by  Proposition \ref{P:seqNSrig}.

The above diagram implies the existence of the dotted arrows $\ov{j_G^{\delta}}$ and $\ov{\gamma_G^{\delta}}$. Then, the exact sequence \eqref{E:Pic-rig1} follows immediately. Moreover, we have
$$
\nu_G^{\delta,\NS}(\Im(\ov{\gamma_G^{\delta}}))=\Im(\omega_G^{\delta}\oplus \gamma_G^{\delta}).
$$
We now conclude using the description of $\Im(\omega_G^{\delta}\oplus \gamma_G^{\delta})$ in Theorem \ref{T:oG+gG}\eqref{T:oG+gG2}.
\end{proof}

We now describe the cokernel of the  homomorphism $\omega_G^{\delta}$, which, via the obstruction morphism $\obs_G^{\delta}$, can be identified with the kernel of the homomorphism 
$$
(\nu_G^{\delta})^*:H^2(\bgr{G}^{\delta}, \Gm)\to H^2(\bg{G}^{\delta}, \Gm),
$$
see \eqref{E:seqLeray}.  

\begin{teo}\label{T:coker-om}
Assume that $g\geq 1$. Let $G$ be a reductive group with maximal torus $T_G$ and Weyl group $\scr W_G$, and fix $\delta\in \pi_1(G)$. 
\begin{enumerate}
\item \label{T:coker-om1} If $n>0$ then the homomorphism $\ov{\Lambda_{\scr D}^*}$ of \eqref{E:mult-car} induces an isomorphism 
 \begin{equation}\label{E:coker-om1}
 \wt{\Lambda_{\scr D}^*}:\coker(\omega_G^{\delta})\xrightarrow{\cong} \coker(\ev_{\scr D(G)}^{\delta}),
 \end{equation}
where $\ev_{\scr D(G)}^{\delta}$ is the homomorphism of Definition/Lemma \ref{D:evG}.  
\item \label{T:coker-om2} If $n=0$ then there exists an exact sequence 
 \begin{equation}\label{E:coker-om2}
 0\to \coker(\ov{\gamma_G^{\delta}})\xrightarrow{\partial_G^{\delta}} \Hom\left(\Lambda(G^{\ab}), \frac{\bbZ}{(2g-2)\bbZ}\right)\xrightarrow{\wt{\Lambda_{\ab}^*}} \coker(\omega_G^{\delta})\xrightarrow{ \wt{\Lambda_{\scr D}^*}} \coker(\ev_{\scr D(G)}^{\delta})\to 0,
  \end{equation}
where $\displaystyle \coker(\ov{\gamma_G^{\delta}})=\frac{\NS(\bgr{G}^{\delta})}{\Im(\ov{\gamma_G^{\delta}})}$ is the co-kernel of the homomorphism $\ov{\gamma_G^{\delta}}$ defined in Theorem \ref{T:Pic-rig}(\ref{T:Pic-rig1}), $\wt{\Lambda_{\ab}^*}$ and $\wt{\Lambda_{\scr D}^*}$ are the homomorphisms induced by, respectively, $\ov{\Lambda_{\ab}^*}$ and $\ov{\Lambda_{\scr D}^*}$ of \eqref{E:mult-car}, and $\partial_G^{\delta}$ is defined as it follows
\begin{equation}\label{E:partial}
\begin{aligned}
\partial_G^{\delta}: \coker(\ov{\gamma_G^{\delta}}) & \longrightarrow \Hom\left(\Lambda(G^{\ab}), \frac{\bbZ}{(2g-2)\bbZ}\right)\\
[b] & \mapsto \left\{x\mapsto \left[b(\delta\otimes \wt x)+(1-g) b(\wt x\otimes \wt x)\right]\right\},  
\end{aligned}
\end{equation}
where $b\in \NS(\bgr{G}^{\delta})\subset \Bil^{s,\scr D-\ev}(\Lambda(T_G))^{\scr W_G}$, and $\wt x\in \Lambda(T_G)$ is any lift of $x\in \Lambda(G^{\ab})$. 
\end{enumerate}
\end{teo}
\begin{proof}
Consider the  diagram 
\begin{equation}\label{E:diag-coker}
\xymatrix{
 \Lambda^*(G^{\ab})\otimes \wh H_{g,n}  \ar@{^{(}->}[r]^{i_G^{\delta}} \ar[d]_{\wh{\omega}:=\omega_{G^{\ab}}^{\delta^{\ab}}\circ i_{G^{\ab}}^{\delta^{\ab}}}&\RPic(\bg{G}^{\delta})\ar@{->>}[r]^(0.45){\gamma_G^{\delta}}\ar[d]^{\omega_G^{\delta}}&  \Bil^{s,\scr D-ev}(\Lambda(T_G))^{\scr W_G}\ar[d]^{\ev_{\scr D(G)}^{\delta} \circ \res_{\scr D}}\\
 \Lambda^*(G^{\ab}) \ar@{^{(}->}[r]^{\ov{\Lambda_{\ab}^*}}& \frac{\Lambda^*(T_G)}{\Lambda^*(T_{G^{\ad}})} \ar@{->>}[r]^{\ov{\Lambda_{\scr D}^*}} &  \frac{\Lambda^*(T_{\scr D(G)})}{\Lambda^*(T_{G^{\ad}})} \\
} 
\end{equation}
which has exact rows by \eqref{E:mult-car} and \eqref{E:gG}. Furthermore, the diagram is commutative. Indeed, the left-hand square is commutative by the functoriality of $\omega_G^{\delta}$ with respect to the inclusion $G^{\ab}\hookrightarrow G$. On the other hand, the morphism $\omega_G^{\delta}\oplus\gamma_G^{\delta}$ factors through $\NS(\bg{G}^{\delta})$. Then, the commutativity of the right-hand square follows by \eqref{E:seqNSrig}. 

Using \eqref{E:iG},  \cite[Rmk. 3.5.1]{FV1} and Definition/Lemma \ref{D:wtG}\eqref{D:wtG2}, we compute 
$$
\wh{\omega}(\chi\otimes (m,\zeta))=\omega_{G^{\ab}}^{\delta^{\ab}}\left(\langle \mt L_{\chi}, \omega_{\pi}^m(\sum \zeta_i \sigma_i) \right)=\omega_{G^{\ab}}^{\delta^{\ab}}\left(\langle(\chi, 0),(0,\zeta)\rangle+m\langle(\chi,0),(\chi,0)\rangle-2m\scr L(\chi, 0) \right)=
$$
\begin{equation}\label{E:whomega1}
=|\zeta|\cdot \chi+2m\chi(\delta^{\ab})\chi-2m(\chi(\delta^{\ab})+1-g)\chi=(|\zeta|+m(2g-2))\chi \quad  \text{ if } g\geq 2. 
\end{equation}
And similarly we get 
\begin{equation}\label{E:whomega2}
\wh{\omega}(\chi\otimes \zeta)=|\zeta|\cdot \chi \quad  \text{ if } g=1. 
\end{equation}
Therefore, combining \eqref{E:whomega1} and \eqref{E:whomega2}, we deduce that 
\begin{equation}\label{E:coker-wh}
\coker(\wh{\omega})=
\begin{cases}
0 & \text{ if } n\geq 1, \\
\frac{\Lambda^*(G^{\ab})}{(2g-2)\Lambda^*(G^{\ab})}= \Hom\left(\Lambda(G^{\ab}), \frac{\bbZ}{(2g-2)\bbZ}\right)& \text{ if } n=0. 
\end{cases}
\end{equation} 
Using this, and applying the snake lemma to \eqref{E:diag-coker}, we get that:
\begin{itemize}
\item if $n\geq 1$ then  
$$\wt{\Lambda_{\scr D}^*}:\coker(\omega_G^{\delta})\xrightarrow{\cong} \coker((\ev_{\scr D(G)}^{\delta}) \: \text{ is an isomorphism},$$ 
which proves \eqref{T:coker-om1}; 
 \item if $n=0$ then we have an exact sequence 
 \begin{equation}\label{E:snakecoker}
\RPic(\bgr{G}^{\delta}) \xrightarrow{\ov{\gamma_G^{\delta}}}  \NS(\bgr{G}^{\delta})\xrightarrow{\wt{\partial_G^{\delta}}} \Hom\left(\Lambda(G^{\ab}), \frac{\bbZ}{(2g-2)\bbZ}\right)\xrightarrow{\wt{\Lambda_{\ab}^*}} \coker(\omega_G^{\delta})\stackrel{\wt{\Lambda_{\scr D}^*}}{\twoheadrightarrow} \coker(\ev_{\scr D(G)}^{\delta}),
 \end{equation}
 where $\wt{\partial_G^{\delta}}$ is the boundary homomorphism. 
\end{itemize}
Part  \eqref{T:coker-om2} follows from the  exact sequence  \eqref{E:snakecoker} together with the following 

\un{Claim:} \: The boundary homomorphism  $\wt{\partial_G^{\delta}}$ is given by 
\begin{equation}\label{E:for-part}
\begin{aligned}
  \NS(\bgr{G}^{\delta}) & \longrightarrow \Hom\left(\Lambda(G^{\ab}), \frac{\bbZ}{(2g-2)\bbZ}\right)\\
b & \mapsto  \left\{x\mapsto \left[b(\delta\otimes \wt x)+(1-g) b(\wt x\otimes \wt x)\right]\right\},  
\end{aligned}
\end{equation}
where $\wt x\in \Lambda(T_G)$ is any lift of $x\in \Lambda(G^{\ab})$. 

Let us first check that formula \eqref{E:for-part} is a well-defined homomorphism. We recall $\NS(\bgr{G}^{\delta})$ is a subgroup of $\Bil^{s,\scr D-\ev}(\Lambda(T_G))^{\scr W_G}$ (see Definition \ref{D:NS-rig}). Then, the class  
$$[b(\delta\otimes \wt x)+(1-g) b(\wt x\otimes \wt x)] \in \frac{\bbZ}{(2g-2)\bbZ}$$ 
 is independent from the chosen lift $\wt x$ of $x$ since $b(\delta\otimes -)_{|\Lambda(T_{\scr D(G)})}\equiv 0$  by \eqref{E:NS-rig}, while the parity of $b(\wt x\otimes \wt x)$ depends solely on $x$ since $b_{|\Lambda(T_{\scr D(G)})\otimes \Lambda(T_{\scr D(G)})}$ is even by Definition \ref{D:NS-rig} and $b^{\bbQ}_{|\Lambda(T_{\scr D(G)})\otimes \Lambda(G^{\ab})}\equiv 0$ by the claim in the proof of Proposition \ref{P:forms-LTG}.  Moreover, the map 
 $$x\mapsto \left[b(\delta\otimes \wt x)+(1-g) b(\wt x\otimes \wt x)\right]$$
  is linear since $b(\delta\otimes -)$ is linear and $b((\wt x+\wt y)\otimes (\wt x+\wt y))\equiv b(\wt x\otimes \wt x)+ b(\wt y\otimes \wt y) \mod 2$.

It remains to check that the boundary homomorphism $\wt{\partial_G^{\delta}}$ is given by the above formula \eqref{E:for-part}. 
By the definition of the boundary homomorphism, for any $b\in  \NS(\bgr{G}^{\delta})=\ker (\ev_{\scr D(G)}^{\delta}\circ \res_{\scr D}) $ the image $\wt{\partial_G^{\delta}}(b)$ is equal to the class 
$[\mu]\in \coker(\wh{\omega})$ of any element $\mu\in \Lambda^*(G^{\ab})$ such that $\ov{\Lambda^*_{\ab}}(\mu)=\omega_G^{\delta}(\wt b)$, where $\wt b\in \RPic(\bg{G}^{\delta})$ is any lift of $b$ via the surjective homomorphism $\gamma_G^{\delta}$. 
Using Theorem \ref{T:Pic-red}\eqref{T:Pic-red2} and Definition/Lemma \ref{D:gammaG}, there are the following three possibilities:
\begin{enumerate}[(a)]
\item If $b\in \Bil^{s,\ev}(\Lambda(T_G))^{\scr W_G}\cap \NS(\bgr{G}^{\delta})\subseteq \Bil^{s,\scr D-ev}(\Lambda(T_G))^{\scr W_G}  $, then we can choose $\wt b=\tau_G^{\delta}(b)$. 
Using Definition/Lemma \ref{D:wtG}\eqref{D:wtG1} and the hypothesis that $b(\delta\otimes -)_{|\Lambda(T_{\scr D(G)})}\equiv 0$, we compute 
$$\omega_G^{\delta}(\tau_G^{\delta}(b))=b(\delta\otimes -)=\ov{\Lambda_{\ab}^*}(b^{\bbQ}(\delta\otimes -)_{|\Lambda(G^{\ab})})\Rightarrow \wt{\partial_G^{\delta}}(b)=[b^{\bbQ}(\delta\otimes -)_{|\Lambda(G^{\ab})}].$$
This shows that $\wt{\partial_G^{\delta}}(b)$ is given by formula \eqref{E:for-part} applied to $b$, noticing that the second term in  \eqref{E:for-part} goes away because $b$ is even.

\item If $b=B_{\ab}^*(\chi\otimes \chi'+\chi'\otimes \chi)$ for some $\chi, \chi'\in \Lambda^*(G^{\ab})$, then we can choose $\wt b=\ab_\#^*(\langle \chi,\chi'\rangle)$. 
Using Definition/Lemma  \ref{D:wtG}\eqref{D:wtG2}, we compute 
$$
\omega_G^{\delta}(\ab_\#^*(\langle \chi,\chi'\rangle))=\ov{\Lambda^*_{\ab}}(\chi(\delta^{\ab})\chi'+\chi'(\delta^{\ab})\chi)\Rightarrow \wt{\partial_G^{\delta}}(b)=[\chi(\delta^{\ab})\chi'+\chi'(\delta^{\ab})\chi].
$$
This shows that $\wt{\partial_G^{\delta}}(b)$ is given by formula \eqref{E:for-part} applied to $b=B_{\ab}^*(\chi\otimes \chi'+\chi'\otimes \chi)$ since, for any $x\in \Lambda(G^{\ab})$, we have
$$
B_{\ab}^*(\chi\otimes \chi'+\chi'\otimes \chi)(\delta\otimes \wt x)+(1-g) B_{\ab}^*(\chi\otimes \chi'+\chi'\otimes \chi)(\wt x\otimes \wt x)=$$
$$= (\chi\otimes \chi'+\chi'\otimes \chi)(\delta^{\ab}\otimes x)+(1-g)(\chi\otimes \chi'+\chi'\otimes \chi)(x\otimes x)=$$
$$=\chi(\delta^{\ab})\chi'(x)+\chi'(\delta^{\ab})\chi(x)+2(1-g)\chi(x)\chi'(x)\equiv \chi(\delta^{\ab})\chi'(x)+\chi'(\delta^{\ab})\chi(x)  \mod (2g-2).$$

\item  If $b=B_{\ab}^*(\chi\otimes \chi)$ for some $\chi \in \Lambda^*(G^{\ab})$, then we can choose $\wt b=\ab_\#^*(\scr L(\chi))$.
Using Definition/Lemma  \ref{D:wtG}\eqref{D:wtG2}, we compute 
$$
\omega_G^{\delta}(\ab_\#^*(\scr L(\chi)))=(\chi(\delta^{\ab})+1-g)\ov{\Lambda^*_{\ab}}(\chi)\Rightarrow \wt{\partial_G^{\delta}}(b)=(\chi(\delta^{\ab})+1-g) [\chi].
$$
This shows that $\wt{\partial_G^{\delta}}(b)$ is given by formula \eqref{E:for-part} applied to $b=B_{\ab}^*(\chi\otimes \chi)$ since, for any $x\in \Lambda(G^{\ab})$, we have
$$
B_{\ab}^*(\chi\otimes \chi)(\delta\otimes \wt x)+(1-g) B_{\ab}^*(\chi\otimes \chi+\chi)(\wt x\otimes \wt x)= (\chi\otimes \chi)(\delta^{\ab}\otimes x)+(1-g)(\chi\otimes \chi)(x\otimes x)=$$
$$=\chi(\delta^{\ab})\chi(x)+(1-g)\chi(x)^2\equiv \chi(\delta^{\ab})\chi(x)+(1-g)\chi(x)  \mod (2g-2).$$
\end{enumerate}

 \end{proof}
\begin{rmk}
The exact sequence \eqref{E:coker-om2} (if $n=0$) is canonically isomorphic to the exact sequence 
 \begin{equation}\label{E:coker-seqBIS}
 0\to \coker(\ov{\gamma_G^{\delta}})\xrightarrow{\wt{\nu_G^{\delta,\NS}}} \coker(\omega_G^{\delta}\oplus \gamma_G^{\delta}) \xrightarrow{\wt{\omega_G^{\delta,\NS}}} \coker(\omega_G^{\delta})\xrightarrow{} \coker(\omega_G^{\delta,\NS})\to 0,
\end{equation}
which is obtained by quotienting out the exact sequence (see Proposition \ref{P:seqNSrig})
\begin{equation*}
 0\to \NS(\bgr{G}^{\delta})\xrightarrow{\nu_G^{\delta,\NS}} \NS(\bg{G}^{\delta}) \xrightarrow{\omega_G^{\delta,\NS}} \frac{\Lambda^*(T_G)}{\Lambda^*(T_{G^{\ad}})} \xrightarrow{} \coker(\omega_G^{\delta,\NS})\to 0,
\end{equation*}
by the exact sub-sequence
$$
 0\to \Im(\ov{\gamma_G^{\delta}})\xrightarrow{} \Im(\omega_G^{\delta}\oplus \gamma_G^{\delta}) \xrightarrow{} \Im(\omega_G^{\delta})\xrightarrow{}  0\to 0.
$$

Indeed, Proposition  \ref{P:seqNSrig} provides a canonical isomorphism $\coker(\omega_G^{\delta,\NS})\xrightarrow{\cong} \coker(\ev_{\scr D(G)}^{\delta})$ which commutes with the two surjections from $\coker(\omega_G^{\delta})$ in \eqref{E:coker-om2} and \eqref{E:coker-seqBIS}.

Moreover, consider the following diagram 
\begin{equation}\label{E:comp-coker}
\xymatrix{
 \Lambda^*(G^{\ab})\otimes \wh H_{g,n} \ar@{^{(}->}[r]^{i_{G^{\ab}}^{\delta^{\ab}}} \ar[d]^{\wh{\omega}:=\omega_{G^{\ab}}^{\delta^{\ab}}\circ i_{G^{\ab}}^{\delta^{\ab}}}&\RPic(\bg{G^{\ab}}^{\delta^{\ab}}) \ar[r]^{\ab_\#^*} \ar[d]^{\omega_{G^{\ab}}^{\delta^{\ab}}\oplus \gamma_{G^{\ab}}^{\delta^{\ab}}}&  \RPic(\bg{G}^{\delta}) \ar@{=}[r] \ar[d]^{\omega_G^{\delta}\oplus \gamma_G^{\delta}}& \RPic(\bg{G}^{\delta}) \ar[d]^{\omega_G^{\delta}} \\
 \Lambda^*(G^{\ab}) \ar@{^{(}->}[r]^{i_1} & \NS(\bg{G^{\ab}}^{\delta^{\ab}}) \ar[r]^{\ab^{*,\NS}} & \NS(\bg{G}^{\delta}) \ar[r]^(0.6){\omega_G^{\delta, \NS}} & \frac{\Lambda^*(T_G)}{\Lambda^*(G^{\ab})} 
}
\end{equation} 
which is commutative: the left square commutes since $\Im(i_{G^{\ab}}^{\delta^{\ab}})=\ker(\gamma_{G^{\ab}}^{\delta^{\ab}})$ by Theorem \ref{T:gG};  the middle square is commutative by the functoriality of $\omega_G^{\delta}\oplus \gamma_G^{\delta}$ (see Theorem \ref{T:oG+gG}\eqref{T:oG+gG1}) and the right square commutes by  the definition of $\omega_G^{\delta, \NS}$. 

From the commutativity of \eqref{E:comp-coker}, we get the following homomorphisms
\begin{equation}\label{E:hom-coker}
\Hom\left(\Lambda(G^{\ab}), \frac{\bbZ}{(2g-2)\bbZ}\right)=\coker(\wh{\omega}) \xrightarrow[\cong]{\wt{i_1}} \coker(\omega_{G^{\ab}}^{\delta^{\ab}}\oplus \gamma_{G^{\ab}}^{\delta^{\ab}}) \xrightarrow[\cong]{\wt{\ab^{*,\NS}}} \coker(\omega_G^{\delta}\oplus \gamma_G^{\delta})\xrightarrow{\wt{\omega_G^{\delta, \NS}}} \coker(\omega_G^{\delta})
\end{equation}
 where $\wt{\ab^{*,\NS}}$ is an isomorphism by \eqref{E:oGgG-coker} and $\wt{i_1}$ is an isomorphism since $\omega_{G^{\ab}}^{\delta^{\ab}}\oplus \gamma_{G^{\ab}}^{\delta^{\ab}}$ induces an isomorphism between the cokernerls of $i_{G^{\ab}}^{\delta^{\ab}}$ and of  $i_1$ which are both canonically isomorphic to $\Bil^s(\Lambda(G^{\ab}))$ (see Theorem \ref{T:gG}).  
 Moreover, the composition of all the homomorphisms of \eqref{E:hom-coker} coincides with the homomorphism $\wt{\Lambda_{\ab}^*}$ since the composition of all the homomorphisms in the bottom row of \eqref{E:comp-coker}  is equal to  $\ov{\Lambda_{\ab}^*}$.
 
 Therefore, we deduce that there exists an isomorphism $\phi=\wt{\ab^{*,\NS}}\circ \wt{i_1}$ from the second term of   \eqref{E:coker-om2} into the second term of \eqref{E:coker-seqBIS}, which commutes with their homomorphisms $\wt{\Lambda_{\ab}^*}$ and $\wt{\omega_G^{\delta,\NS}}$ onto $\coker(\omega_G^{\delta})$. This implies that $\phi$ also commutes with the kernel homomorphisms of $\wt{\Lambda_{\ab}^*}$ and $\wt{\omega_G^{\delta,\NS}}$, which are, respectively, $\partial_G^{\delta}$ and $\wt{\nu_G^{\delta,\NS}}$, and this completes the proof. 
\end{rmk} 

\begin{rmk}\label{R:tori-coker}
Assume that $g\geq 1$. Let  $G=T$ be a torus and let $d\in \Lambda(T)$.  Then clearly $\coker(\ev_{\scr D(T)}^{d})=0$, which implies that $\coker(\omega_T^{d})=0$ if $n>0$. On the other hand, if $n=0$, then, using the explicit basis of $\RPic(\bg{T}^d)$ given in \cite[Thm. 4.0.1(2)]{FV1}, it is possible to check that 
$$
\begin{aligned}
& \coker(\ov{\gamma_T^d})\cong \frac{\bbZ}{\frac{2g-2}{\gcd(2g-2,\div(d)+1-g)}\bbZ}\oplus \left[ \frac{\bbZ}{\frac{2g-2}{\gcd(g-1,\div(d))}\bbZ}\right]^{\oplus(\dim(T)-1)}  \\
&\coker(\omega_T^d)\cong \frac{\bbZ}{\gcd(2g-2,\div(d)+1-g)\bbZ}\oplus \left[\frac{\bbZ}{\gcd(g-1,\div(d))\bbZ}\right]^{\oplus(\dim(T)-1)}, \\
\end{aligned}
$$
where $\div(d)$ is the divisibility of $d$ in the lattice $\Lambda(T)$, with the convention that $\coker(\ov{\gamma_T^d})=\{0\}$ if $g=1$ and $d=0$ (when the above expression for $\coker(\ov{\gamma_T^d})$ is not well-defined).
\end{rmk} 

We end this section by describing the relative Picard group of the rigidifcation $\bgr{G}^{\delta}$ and the cokernel of the weight homomorphism $\w_G^{\delta}$, in genus $g=0$. 

\begin{rmk}\label{R:weight-g0}
Assume that $g=0$ and $n\geq 1$. 
Using \eqref{E:RPic-g0}, it can be proved that the weight homomorphism $\w_G^{\delta}$ is equal to the following composition
\begin{equation}\label{E:wG-g0}
\begin{aligned}
\w_G^{\delta}:\Pic(\mathrm{Bun}_{G,0,n}^{\delta})\twoheadrightarrow \RPic(\mathrm{Bun}_{G,0,n}^{\delta}) & \xrightarrow{\omega_G^{\delta}} \Lambda^*(\scr Z(G))=\frac{\Lambda^*(T_G)}{\Lambda^*(T_{G^{\ad}})}, \\
(l_{\scr R}, b) & \mapsto [l_{\scr R}\oplus b(d^{ss}\otimes -)]
\end{aligned}
\end{equation}
Therefore, using the exact sequences \eqref{E:seqRPic-g0} and \eqref{E:mult-car} and the fact that $\omega_{G^{\ab}}^{\delta}$ is an isomorphism, it follows that
(for some, or equivalently any,  lift $d^{\ss}\in \Lambda(T_G)$ of the image $\delta^{\ss}$ of $\delta$ in $\pi_1(G^{\ss})$):
\begin{enumerate}[(i)]
\item  the homomorphism $\theta_G^{\delta}$ induces an isomorphism 
$$
 \RPic(\mathfrak{Bun}_{G,0,n}^{\delta})\xrightarrow{\cong} \left\{b\in   \Bil^{s,\ev}(\Lambda(T_{G^{\sc}}))^{\scr W_G}\: : b(d^{\ss}\otimes -) \in \Lambda^*(T_{G^{\ad}})\right\};
$$
\item the cokernel of $\omega_G^{\delta}$ (and hence of the weight homomorphism $\w_G^{\delta}$) is equal to the cokernel of the homomorphism 
\begin{equation*} 
\begin{aligned}
\left\{
\begin{aligned}
&b\in \Bil^{s,\ev}(\Lambda(T_{G^{\sc}}))^{\scr W_G}\: : \\ 
&b^{\bbQ}(d^{\ss}\otimes -) \: \text{ is integral on } \Lambda(T_{\scr D(G)}) 
\end{aligned}
\right\} & \rightarrow  \frac{\Lambda^*(T_{\scr D(G)})}{\Lambda^*(T_{G^{\ad}})}, \\
b& \mapsto [b(d^{\ss}\otimes -)].
\end{aligned}
\end{equation*}
\end{enumerate}
\end{rmk}

\section{The universal moduli space $M_{G,g,n}^{ss}$ and its divisor class group}\label{Sec:Mss}

In this section, we will describe the divisor class group of the universal moduli space $M_{G,g,n}^{\delta,ss}$ of semistable $G$-bundles (over $n$-marked smooth curves of genus $g$)  in terms of the Picard group of $\bgr{G}$. Before presenting the results, we need some preparation.

\begin{defin}\label{D:sstableG}
\noindent 
\begin{enumerate}[(i)]
\item Let $P\to C$ be a $G$-bundle over a $k$-curve $C$. We say that $P$ is \emph{(semi)stable} if for any reduction $F$ to any parabolic subgroup $P\subseteq G$, we have 
$$\text{deg(\text{ad}(F))}\underset{(\leq)}{<} 0,$$ 
where $\text{ad}(F):=(F\times\mathfrak p)/P$ is the adjoint bundle of $F$, i.e.  the vector bundle on $C$ induced by $F$ via the adjoint representation $P\to GL(\mathfrak p)$. We say that $P$ is \emph{regularly stable}, if either $G$ is a torus or $P$ is stable and $\operatorname{Aut}(P)=\scr Z(G)$.
\item Denote by $\bg{G}^{\delta,ss}$, resp. $\bg{G}^{\delta,rs}$, the locus in $\bg{G}^{\delta}$ consisting of $G$-bundles over families of $n$-marked curves of genus $g$ whose geometric fibers are semistable, resp. regularly stable, and by $\bgr{G}^{\delta,ss}$, resp. $\bgr{G}^{\delta,rs}$, its image in $\bgr{G}^{\delta}$. 
\end{enumerate}
\end{defin}

We collect in the following Proposition the properties of the loci  $\bg{G}^{\delta,ss}$ and $\bg{G}^{\delta,rs}$ (resp. $\bgr{G}^{\delta,ss}$ and $\bgr{G}^{\delta,rs}$).

\begin{prop}\label{P:ss-loci}
\noindent
\begin{enumerate}[(i)]
\item \label{P:ss-loci1} The loci $\bg{G}^{\delta,ss}\subseteq \bg{G}^{\delta}$ and $\bgr{G}^{\delta,ss}\subseteq \bgr{G}^{\delta}$ are open (smooth) substacks of finite type over $\Mg$. 
\item \label{P:ss-loci3} If $G$ is a torus then  $\bg{G}^{\delta,rs}=\bg{G}^{\delta,ss}= \bg{G}^{\delta}$ and $\bgr{G}^{\delta,rs}=\bgr{G}^{\delta,ss}=\bgr{G}^{\delta}$.
\item \label{P:ss-loci4} If $G$ is not a torus then the complements $\bg{G}^{\delta}\setminus \bg{G}^{\delta, ss}$ and $\bgr{G}^{\delta}\setminus \bgr{G}^{\delta, ss}$ have codimension at least $g$. 
\item\label{P:ss-loci5} If $G$ is not a torus and one of the following holds
\begin{enumerate}[(a)]
	\item $\operatorname{char}(k)>0$ and $g\geq 4$,
	\item $\operatorname{char}(k)=0$ and $g\geq 2$, with the exception of the case $g=2$ and $G$ having a non-trivial homomorphism into $PGL_2$,
\end{enumerate}
then the complements $\bg{G}^{\delta,ss}\setminus \bg{G}^{\delta, rs}$ and $\bgr{G}^{\delta,ss}\setminus \bgr{G}^{\delta, rs}$ have codimension at least two.
\end{enumerate}
\end{prop}
\begin{proof}
The properties for $\bg{G}^{\delta,ss}$ have been proved in \cite[Prop. 3.2.3, 3.2.5]{FV1}. The properties for $\bg{G}^{\delta,rs}$ have been proved in \cite[Theorem 2.5]{BH12} for $\operatorname{char}(k)>0$ and in \cite[Theorem II.6]{FaSt} for $\operatorname{char}(k)=0$. The properties for $\bgr{G}^{\delta,ss}$ and $\bgr{G}^{\delta,rs}$ follow since $\nu_G^{\delta}:\bg{G}^{\delta}\to \bgr{G}^{\delta}$ is a $\scr Z(G)$-gerbe. 
\end{proof}

\begin{cor}\label{C:pic-ss=all} 
	\noindent
	\begin{enumerate}[(i)]
		\item If either $G$ is a torus or $g\geq 2$, the restriction homomorphisms
		$$
		\Pic(\bg{G}^{\delta})\to \Pic(\bg{G}^{\delta,ss}) \text{ and }\Pic(\bgr{G}^{\delta})\to \Pic(\bgr{G}^{\delta,ss})
		$$
		are bijective.
		\item If one of the following holds
		\begin{enumerate}[(a)]
			\item $G$ is a torus,
			\item $G$ is not a torus, $\operatorname{char}(k)>0$ and $g\geq 4$,
			\item $G$ is not a torus, $\operatorname{char}(k)=0$ and $g\geq 2$, with the exception of the case $g=2$ and $G$ having a non-trivial homomorphism into $PGL_2$,
		\end{enumerate}
		then the restriction homomorphisms
		$$
		\Pic(\bg{G}^{\delta,ss})\to \Pic(\bg{G}^{\delta,rs}) \text{ and }\Pic(\bgr{G}^{\delta,ss})\to \Pic(\bgr{G}^{\delta,rs})
		$$
		are bijective.
	\end{enumerate}
\end{cor}
\begin{proof}
Under the assumptions in (i), Proposition \ref{P:ss-loci} implies that the complements  $ \bg{G}^{\delta}\setminus \bg{G}^{\delta, ss}$ and $\bgr{G}^{\delta}\setminus \bgr{G}^{\delta, ss}$ are either empty or they have codimension at least two. Then the conclusion follows since $\bg{G}^{\delta}$ (and hence $\bgr{G}^{\delta}$) is smooth by Theorem \ref{T:propBunG}\eqref{T:propBunG3}, see e.g. \cite[Lemma 2.3.1]{FV1}. The same argument applies to point (ii).
\end{proof}

We now make the following 

\begin{assumption}\label{Assump}
There exists an adequate moduli space \begin{equation}\label{E:adeqmod}
\pi:\bgr{G}^{\delta,ss}\to M_{G,g,n}^{\delta,ss}
\end{equation}
in the sense of Alper \cite{AlAd} (which is the same as a good moduli space if $\rm{char}(k)=0$, see \cite[Prop. 5.1.4]{AlAd}).
\end{assumption}


Although we expect that Assumption \ref{Assump} should always hold true, we do not know of a reference in the literature where this is proved in full generality. As far as we know, the cases covered in the literature are the following: 
\begin{itemize}
	\item $G=T$ torus, in which case $\pi$ is a coarse moduli space in the sense of Keel-Mori;
	\item $G=GL_r$, by \cite{P96};
	\item $\rm{char}(k)=0$, $g\geq 2$ and $n=0$, by \cite{Cast};
	\item $\rm{char}(k)=0$ and $\Mg$ is a variety (which happens if and only if $n>2g+2$), by \cite{Langer}.
\end{itemize}
Furthermore, in these cases, the adequate moduli space is a quasi-projective variety.

\begin{rmk}In characteristic zero and $g\geq 2$, the remaining cases should follow by a slight modification of the argument in \cite{Cast}. In positive characteristic over a fixed curve, the problem has been solved in \cite{GLSS1}, \cite{GLSS2}. We are not aware, if the same results hold in the universal setting.
\end{rmk}

We are now ready to compare the divisor class group $\Cl(M_{G,g,n}^{\delta,ss})$ of the algebraic space $M_{G,g,n}^{\delta,ss}$ (see \cite[\href{https://stacks.math.columbia.edu/tag/0EDQ}{Tag 0EDQ}]{stacks-project}) with the Picard group $\Pic(\bgr{G}^{\delta})$.

We denote by $M_{G,g,n}^{\delta,rs}$ the (open) subset $\pi(\bg{G}^{\delta,rs})\subset M_{G,g,n}^{\delta,ss}$. Since $\bgr{G}^{\delta}$ is smooth, the algebraic space $M_{G,g,n}^{\delta,rs}$ is normal, see \cite[Prop. 5.4.1]{AlAd}. Then, the complement of the smooth locus $(M_{G,g,n}^{\delta,rs})_{\operatorname{sm}}$ of $M_{G,g,n}^{\delta,rs}$ has codimension at least two. In particular, the restriction homomorphism gives an isomorphism
\begin{equation*}
\Cl(M_{G,g,n}^{\delta,rs})\xrightarrow{\cong}\Cl((M_{G,g,n}^{\delta,rs})_{\operatorname{sm}})=\Pic((M_{G,g,n}^{\delta,rs})_{\operatorname{sm}}).
\end{equation*}
By definition, the pull-back along the adequate moduli space $\pi$ gives an injective homomorphism
\begin{equation*}
\Pic((M_{G,g,n}^{\delta,rs})_{\operatorname{sm}})\hookrightarrow \Pic(\pi^{-1}((M_{G,g,n}^{\delta,rs})_{\operatorname{sm}}))\xrightarrow[\res^{-1}]{\cong}\Pic(\bgr{G}^{\delta, rs}).
\end{equation*}
The last isomorphism is the inverse of the restriction morphism $\res$, which is an isomorphism since $\bgr{G}^{\delta,rs}$ is smooth and that the complement of $\pi^{-1}((M_{G,g,n}^{\delta,rs})_{\operatorname{sm}})$ has codimension at least two (since $\bgr{G}^{\delta,rs}$ has finite inertia and, so, the inverse image along $\pi$ preserves the codimension). 
Putting all together, we get an injective homomorphism
\begin{equation}\label{E:pi-pull}
\widetilde{\pi}^*:\Cl(M_{G,g,n}^{\delta,rs})\hookrightarrow \Pic(\bgr{G}^{rs}).
\end{equation}

\begin{teo}\label{T:Cl-Pic}Let $g+n\geq 3$ (i.e. $\mathcal M_{g,n}$ is generically a variety). Suppose that Assumption \ref{Assump} holds true and that one of the followings hold
\begin{enumerate}[(i)]
	\item $G$ is a torus,
	\item $G$ is not a torus, $\operatorname{char}(k)>0$, $g\geq 4$,
	\item $G$ is not a torus, $\operatorname{char}(k)=0$, $g\geq 2$, with the exception of the case $g=2$ and $G$ having a non-trivial homomorphism into $PGL_2$,
\end{enumerate}
then we have the following isomorphisms
$$
\Cl(M_{G,g,n}^{\delta,ss})\xrightarrow[\res]{\cong}\Cl(M_{G,g,n}^{\delta,rs})\xrightarrow[\widetilde\pi^*]{\cong} \Pic(\bgr{G}^{\delta,rs})\xrightarrow[\res^{-1}]{\cong}\Pic(\bgr{G}^{\delta, ss})
\xrightarrow[\res^{-1}]{\cong}\Pic(\bgr{G}^{\delta})
$$
where $\res$ are the obvious restriction homomorphisms and $\widetilde{\pi}^*$ is the homomorphism \eqref{E:pi-pull}.
\end{teo}

\begin{proof}
The last two isomorphisms follow by Corollary \ref{C:pic-ss=all}. Observe that
	$$
	\operatorname{cod}_{M_{G,g,n}^{\delta,ss}}\left(M_{G,g,n}^{\delta,ss}\setminus M_{G,g,n}^{\delta,rs}\right)\geq \operatorname{cod}_{\bg{G}^{\delta,ss}}\left(\bg{G}^{\delta,ss}\setminus \bg{G}^{\delta,rs}\right)\geq 2
	$$
where the last inequality follows by Proposition \ref{P:ss-loci}. In particular, the first homomorphism $\res$ is bijective. 

In order to prove that $\widetilde\pi^*$ is an isomorphism, let $\bg{G}^{\delta,o-rs}\subset\bg{G}^{\delta,rs}$ be the open substack whose $k$-points are triples $(C,\underline{\sigma},P)\in \bg{G}^{\delta}(k)$ such that 
\begin{itemize}
	\item $\operatorname{Aut}(C,\underline{\sigma})=\{1\}$,
	\item $\operatorname{Aut}(P)=\scr Z(G)$.
\end{itemize}
We denote by $M_{G,g,n}^{\delta,o-rs}$ the open subset $\pi(\bg{G}^{\delta,o-rs})\subset M_{G,g,n}^{\delta,rs}$ and by $\bgr{G}^{\delta,o-rs}$ the open substack $\nu_G^{\delta}(\bg{G}^{\delta,o-rs})\subset \bgr{G}^{\delta,rs}$.

We now assume $g+n>3$. In this range, the locus of $n$-marked curves without non-trivial automorphism has codimension at least two in $\Mg$ (see \cite[Chap. XII, Prop. (2.5)]{GAC2}), 
we deduce that
\begin{equation}\label{E:cod-rs}
\operatorname{cod}_{\bgr{G}^{\delta,rs}}\left(\bgr{G}^{\delta,rs}\setminus \bgr{G}^{\delta,o-rs}\right)\geq 2\text{ and }
\operatorname{cod}_{M_{G,g,n}^{\delta,rs}}\left(M_{G,g,n}^{\delta,rs}\setminus M_{G,g,n}^{\delta,o-rs}\right)\geq 2.
\end{equation}
In particular, the restriction homomorphisms
\begin{equation}\label{E:pic-rs}
	\Pic(\bgr{G}^{\delta,rs})\to \Pic(\bgr{G}^{\delta,o-rs})\text{ and }\Cl(M_{G,g,n}^{\delta,rs})\to\Cl(M_{G,g,n}^{\delta,o-rs})
\end{equation}
are bijective. Hence, it is enough to show that the restriction 
$$
\widetilde\pi^*_{|\Cl(M_{G,g,n}^{\delta,o-rs})}: \Cl(M_{G,g,n}^{\delta,o-rs})\hookrightarrow \Pic(\bgr{G}^{\delta,o-rs})
$$
is an isomorphism. This follows from the fact that  the adequate moduli space \eqref{E:adeqmod} restricted to those open substacks
\begin{equation}\label{E:ad-rs}
\pi|_{\bgr{G}^{\delta,o-rs}}:\bgr{G}^{\delta,o-rs}\to M_{G,g,n}^{\delta,o-rs}
\end{equation}
is an isomorphism. Indeed, by definition $\bg{G}^{\delta,o-rs}$ is an algebraic space. In particular, the adequate moduli space \eqref{E:ad-rs} is a coarse moduli space in the sense of Keel-Mori, see \cite[Theorem 8.3.2]{AlAd}. Since the coarse moduli space is universal for maps to algebraic spaces, we must have that $\pi|_{\bgr{G}^{\delta,o-rs}}$ is an isomorphism. 

It remains the case $g+n=3$, i.e. $(g,n)=(3,0),(2,1)$. Under these assumptions, the locus of $n$-marked curves with non-trivial automorphisms has one (irreducible) divisor component $D$ (see \cite[Chap. XII, Prop. (2.5)]{GAC2}). Since the morphism $\Phi^{\delta}_G:\bgr{G}^{\delta,rs}\to\Mg$ is smooth with irreducible fibers, we deduce that the restriction homomorphisms \eqref{E:pic-rs} are surjective and their kernels are freely generated by the irreducible divisors $(\Phi^{\delta}_G)^{-1}(D)$ and $\pi((\Phi^{\delta}_G)^{-1}(D))$, respectively. With this in mind, the theorem follows by repeating the argument of the previous case.
\end{proof}

	

\section{Examples}\label{Sec:Examples}

The aim of this section is to make explicit the results of this paper for the reductive groups $G$ such that the semisimple factor $\g^{\ss}$ of the Lie algebra $\g$ of $G$ (see \eqref{E:Lie-split}) is simple. 
We will therefore distinguish several cases according to the type of the simple Lie algebra $\g^{\ss}$.
For each of these cases, we will first compute   the  lattices of $W_G$-symmetric bilinear forms appearing in Proposition \ref{P:forms-LTG} and Definition/Lemma \ref{D:evGtilde}
\begin{equation}\label{E:latt-forms}
\Bil^{s,\ev}(\Lambda(T_{\scr D(G)})\vert \Lambda(T_{G^{\ss}}))^{\scr W_G} \stackrel{r_G}{\hookrightarrow}  \Bil^{s,\sc-\ev}(\Lambda(T_{\scr D(G)})\vert \Lambda(T_{G^{\ss}}))^{\scr W_G} \subseteq  \Bil^{s, \ev}(\Lambda(T_{G^{\sc}}))^{\scr W_G}, \end{equation}
which have rank one by Corollary \ref{C:rank-Bil}.
Then we will compute, for any $\delta\in \pi_1(G)$, the cokernels of the morphisms 
$$
\begin{aligned}
\ev_{\scr D(G)}^{\delta}:\Bil^{s,\ev}(\Lambda(T_{\scr D(G)})\vert \Lambda(T_{G^{\ss}}))^{\scr W_G} & \longrightarrow \Lambda^*(\scr Z(\scr D(G)))=\frac{\Lambda^*(T_{\scr D(G)})}{\Lambda^*(T_{G^{\ad}})}, \\
\wt \ev_{\scr D(G)}^{\delta}: \Bil^{s,\sc-\ev}(\Lambda(T_{\scr D(G)})\vert \Lambda(T_{G^{\ss}}))^{\scr W_G} & \longrightarrow \Lambda^*(\scr Z(\scr D(G)))=\frac{\Lambda^*(T_{\scr D(G)})}{\Lambda^*(T_{G^{\ad}})}.
\end{aligned}
$$
appearing in Definition/Lemmas \ref{D:evG} and \ref{D:evGtilde}. Note that, since $\ev_{\scr D(G)}^{\delta}$ is the restriction of $\wt \ev_{\scr D(G)}^{\delta}$,  there is a surjection 
$$\coker(\ev_{\scr D(G)}^{\delta})\twoheadrightarrow \coker(\wt \ev_{\scr D(G)}^{\delta}),$$
whose kernel is either trivial or isomorphic to $\bbZ/2\bbZ$ by Definition/Lemma \ref{D:evGtilde}\eqref{D:evGtilde1}  and the fact that the lattices \eqref{E:latt-forms} have rank one.

\subsection{Type $A_{n-1}$ ($n\geq 2$)}

Let us first recall some properties of the root system $A_{n-1}$.

Consider the vector space $\bbR^n$ endowed with the standard scalar product $(-,-)$ and with the canonical bases $\{\epsilon_1,\ldots, \epsilon_n\}$.  
Consider the subvector space 
\begin{equation*}
V(A_{n-1}):=\{\xi=(\xi_1,\ldots, \xi_n)\in \bbR^n\: : \sum_i \xi_i=0\}\subset \bbR^n.
\end{equation*}
We will freely identify $V(A_{n-1})$ with its dual vector space by mean of the (restriction of the) standard scalar product $(-,-)$. 
The root (resp. coroot) lattice $Q(A_{n-1})$ (resp. $Q(A_{n-1}^\vee)$) and the weight (resp. coweight) lattice $P(A_{n-1})$ (resp. $P(A_{n-1}^{\vee})$) of $A_{n-1}$ are given by 
\begin{equation*}
Q(A_{n-1})=Q(A_{n-1}^\vee)=V(A_{n-1})\cap \bbZ^n \subset P(A_{n-1})=P(A_{n-1}^\vee)=V(A_{n-1})\cap \bbZ^n +\left\langle \omega_1:=\epsilon_1-\frac{\sum_i \epsilon_i}{n}\right\rangle.
\end{equation*}
It follows that group $P(A_{n-1})/Q(A_{n-1})$ is cyclic of order $n$ and it is generated by $\omega_1$. The Weyl group $\scr W(A_{n-1})$ of $A_{n-1}$ is equal to $S_n$ and it acts on the above lattices by permuting the coordinates of $V(A_{n-1})\subset \bbR^n$.

A semisimple group $H$ which is almost-simple of type $A_{n-1}$ is isomorphic to $\SL_n/\mu_r$, for some (unique) $r\in \bbN$ such that $ r \vert n$. In particular, $H^{\sc}=\SL_n$ and $H^{\ad}=\PSL_n$. 
By choosing the standard maximal tours $T_H$ of $H$  consisting of diagonal matrices, we get the canonical identifications 
\begin{equation}\label{E:lattAn}
\begin{aligned}
\Lambda(T_{\SL_n})=Q(A_{n-1}^{\vee})\subseteq  \Lambda(T_{\SL_n/\mu_r})=Q(A_{n-1}^{\vee})+\left\langle \frac{n}{r} \omega_1\right\rangle\subseteq  \Lambda(T_{\PSL_n})=P(A_{n-1}^{\vee}), \\
\Lambda^*(T_{\SL_n})=P(A_{n-1})\supseteq  \Lambda^*(T_{\SL_n/\mu_r})=Q(A_{n-1})+\left\langle r \omega_1\right\rangle\supseteq  \Lambda^*(T_{\PSL_n})=Q(A_{n-1}). \\
\end{aligned}
\end{equation}
It follows that the fundamental group of $\SL_n/\mu_r$ is equal to 
\begin{equation}\label{E:pi1A}
\pi_1(H)=\frac{\Lambda(T_{\SL_n/\mu_r})}{\Lambda(T_{\SL_n})}=\left\langle  \frac{n}{r} \omega_1 \right\rangle \cong \bbZ/r\bbZ,
\end{equation}
while the character group of the center $\scr Z(\SL_n/\mu_r)= \mu_{n}/\mu_r\cong \mu_{n/r}$ is equal to 
\begin{equation}\label{E:ZA}
\Lambda^*(\scr Z(\SL_n/\mu_r))=\frac{\Lambda^*(T_{\SL_n/\mu_r})}{\Lambda^*(T_{\PSL_n})}=\left\langle  r \omega_1 \right\rangle \cong \bbZ/\frac{n}{r} \bbZ,
\end{equation}

From now on, we will consider the following 

\un{Set-up}: Let $G$ be a reductive group such that $\scr D(G)= \SL_n/\mu_r$ and $G^{\ss}=\SL_n/\mu_s$, with $1\leq r\vert s\vert n$.  
Equivalently, $G$ is the product of a torus and one of the following reductive groups (see \cite[(2.1)]{BLS98})
$$
\begin{sis}
&\SL_n/\mu_r & \text{ if } r=s, \\
& C_{\mu_s/\mu_r}(\SL_n/\mu_r):=\frac{\SL_n/\mu_r\times \Gm}{\mu_s/\mu_r} & \text{ if } r\neq s,
 \end{sis}
 $$ 
where $\mu_s/\mu_r\cong \mu_{s/r}$ is embedded diagonally in $\SL_n/\mu_r\times \Gm$.


\begin{lem}\label{L:bilA}
Let $G$ be a reductive group as in the above set-up. Then we have that 
\begin{enumerate}[(i)]
\item \label{L:bilA1}  $\Bil^{s, \ev}(\Lambda(T_{G^{\sc}}))^{\scr W_G}=\langle (-,-)\rangle$;
\item  \label{L:bilA2} $\Bil^{s, \sc-\ev}(\Lambda(T_{\scr D(G)})\vert \Lambda(T_{G^{\ss}}))^{\scr W_G}=\left\langle \frac{\lcm(rs,n)}{n}(-,-)\right\rangle$;
\item  \label{L:bilA3} $\Bil^{s,\ev}(\Lambda(T_{\scr D(G)})\vert \Lambda(T_{G^{\ss}}))^{\scr W_G}=
\begin{cases}
\left\langle 2 \frac{\lcm(rs,n)}{n}(-,-)\right\rangle & \text{ if } v_2(r)=v_2(s)\geq \frac{v_2(n)}{2} >0, \\
\left\langle \frac{\lcm(rs,n)}{n}(-,-)\right\rangle & \text{ otherwise. } 
\end{cases}
$
\end{enumerate}
where $v_2(N)$ is the $2$-adic valuation of a number $N\in \bbN$, i.e. $2^{v_2(N)}\vert N$ but $2^{v_2(N)+1}\not\vert N$. 

In particular, 
$$
\coker(r_G)=
\begin{cases}
\bbZ/2\bbZ & \text{ if } v_2(r)=v_2(s)\geq \frac{v_2(n)}{2} >0, \\
0 & \text{ otherwise.}
\end{cases}
$$
\end{lem}
Part \eqref{L:bilA2}, combined with \cite{LS97},  recovers \cite[\S 3]{BLS98} if $G=\SL_n$ or $\PGL_n$ and \cite[Thm. 5.7]{Laszlo} if $G=\SL_n/\mu_r$ for some $1\leq r \mid n$.
\begin{proof}
First of all, the three lattices have dimension one by Corollary \ref{C:rank-Bil}. 

Part \eqref{L:bilA1}: the symmetric bilinear form $(-,-)\in \Bil^s(\Lambda(T_{G^{\sc}}))$ is $\scr W(A_{n-1})=S_n$-invariant and it is even since $\Lambda(T_{G^{\sc}})=Q(A_{n-1}^\vee)$ is generated by the elements $\epsilon_i-\epsilon_j$ (for $1\leq i\neq j \leq n$) and we have that 
$$(\epsilon_i-\epsilon_j,\epsilon_i-\epsilon_j)=2.$$ 
This also shows that $(-,-)$ is a generator of  $\Bil^{s, \ev}(\Lambda(T_{G^{\sc}}))^{\scr W_G}$ since $2$ is the smallest non-zero even integer.

Part \eqref{L:bilA2}: consider an element of  $\Bil^{s, \ev}(\Lambda(T_{G^{\sc}}))^{\scr W_G}$, which, by part \eqref{L:bilA1}, is of the form $\alpha(-,-)$ for $\alpha \in \bbZ$. 
Using \eqref{E:lattAn}, we get that the element $\alpha(-,-)$ belongs to  $\Bil^{s, \sc-\ev}(\Lambda(T_{\scr D(G)})\vert \Lambda(T_{G^{\ss}}))^{\scr W_G}$ if and only if 
$$
\bbZ\ni \alpha\left(\frac{n}{r}\omega_1,\frac{n}{s}\omega_1\right)=\alpha\frac{n^2}{rs}\left(\omega_1,\omega_1\right)=\alpha\frac{n^2}{rs}\frac{n-1}{n}\Longleftrightarrow \alpha\frac{n}{rs}\in \bbZ \Longleftrightarrow \frac{\lcm(rs, n)}{n}\vert \alpha,
$$
where in the second equivalence we have used that $r$ and $s$ are coprime with $n-1$ since they divide $n$ by assumption. 

Part \eqref{L:bilA3}: since $\Bil^{s,\sc-\ev}(\Lambda(T_{\scr D(G)})\vert \Lambda(T_{G^{\ss}}))^{\scr W_G}$ has rank one, by Definition/Lemma \ref{D:evGtilde}\eqref{D:evGtilde1} we have that the inclusion $r_G$ has index at most two. 
Using \eqref{E:lattAn} and part \eqref{L:bilA2}, we get that $r_G$ has index two if and only if 
$$
\frac{\lcm(rs,n)}{n}\left(\frac{n}{r}\omega_1,\frac{n}{r}\omega_1\right)=\frac{\lcm(rs,n)}{r^2}(n-1)\quad  \text{ is odd } \Longleftrightarrow \begin{sis} & n \: \text{ is even}, \\ & v_2(\lcm(rs, n))\leq 2v_2(r). \end{sis}
$$
The second condition is equivalent to (using that $r\vert s$)
$$ 
\max(v_2(rs), v_2(n))=v_2(\lcm(rs, n)) \leq 2v_2(r) \Longleftrightarrow \begin{sis}  & v_2(r)+v_2(s) \leq 2 v_2(r), \\ & v_2(n)\leq 2v_2(r), \end{sis}
\Longleftrightarrow \begin{sis}  & v_2(s)=  v_2(r), \\ & v_2(n)\leq 2v_2(r), \end{sis}
$$
and this concludes the proof.
\end{proof}


\begin{lem}\label{L:evA}
Let $G$ be a reductive group as in the above set-up. Consider an element $\delta\in \pi_1(G)$ and denote by $\Delta^{\ss}\in \bbZ/s\bbZ$ the element corresponding to $\delta^{\ss}\in \pi_1(G^{\ss})=\pi_1(\SL_n/\mu_s)$ under the  isomorphism 
$\pi_1(\SL_n/\mu_s) \cong \bbZ/s\bbZ$ of \eqref{E:pi1A}. Then we have that 
\begin{enumerate}[(i)]
\item \label{L:evA1} $\coker(\wt \ev_{\scr D(G)}^{\delta})\cong \frac{\bbZ}{\gcd\left(\frac{\Delta^{\ss}\lcm(rs,n)}{rs},\frac{n}{r} \right)\bbZ};$
\item \label{L:evA2} $\coker(\ev_{\scr D(G)}^{\delta})\cong 
\begin{cases}
\frac{\bbZ}{\gcd\left(\frac{2\Delta^{\ss}\lcm(rs,n)}{rs},\frac{n}{r} \right)\bbZ} & \text{ if }  v_2(r)=v_2(s)\geq \frac{v_2(n)}{2} >0, \\
\frac{\bbZ}{\gcd\left(\frac{\Delta^{\ss}\lcm(rs,n)}{rs},\frac{n}{r} \right)\bbZ} & \text{ otherwise}.
\end{cases} 
$
\end{enumerate}
In particular, 
$$
|\coker(\ev_{\scr D(G)}^{\delta})|=
\begin{cases} 
|\coker(\wt \ev_{\scr D(G)}^{\delta})|+1 & \text{ if }  
\begin{sis} 
& v_2(r)=v_2(s)\geq \frac{v_2(n)}{2} >0, \\
& v_2(\Delta^{\ss})<v_2(n)-v_2(r),
\end{sis}\\
|\coker(\wt \ev_{\scr D(G)}^{\delta})| & \text{ otherwise}.
\end{cases}
$$
\end{lem}
\begin{proof}
Let us first prove part \eqref{L:evA1}. Lemma \ref{L:bilA}\eqref{L:bilA2} implies that $\Bil^{s, \sc-\ev}(\Lambda(T_{\scr D(G)})\vert \Lambda(T_{G^{\ss}}))^{\scr W_G}$ is freely generated by $ \frac{\lcm(rs,n)}{n}(-,-)$. 
Therefore part \eqref{L:evA1} follows from the following 

\un{Claim}: We have that 
$$
\wt \ev_{\scr D(G)}^{\delta}\left(\frac{\lcm(rs,n)}{n}(-,-) \right)=\frac{\Delta^{\ss}\lcm(rs,n)}{rs}\in \bbZ/\frac{n}{r}\bbZ\cong \Lambda^*(\scr Z(\scr D(G))),
$$
where the last isomorphism follows from \eqref{E:ZA} together with our assumption that $\scr D(G)=\SL_n/\mu_r$. 

Indeed, using that $\pi_1(G^{\ss})=\pi_1(\SL_n/\mu_s)$ is generated by $\frac{n}{s}\omega_1$ by \eqref{E:pi1A} and $ \Lambda^*(\scr Z(\scr D(G)))= \Lambda^*(\scr Z(\SL_n/\mu_r))$  is generated by $r\omega_1$ by \eqref{E:ZA}, the claim follows from the obvious identity 
\begin{equation*}\label{E:ClaimA}
\frac{\lcm(rs,n)}{n}\left(\Delta^{\ss}\frac{n}{s}\omega_1, -\right)=\Delta^{\ss} \frac{\lcm(rs,n)}{rs}(r\omega_1,-)\in \Lambda^*(T_{\SL_n/\mu_r}).
\end{equation*}

Part \eqref{L:evA2} is proved in the same way using the description of $\Bil^{s,\ev}(\Lambda(T_{\scr D(G)})\vert \Lambda(T_{G^{\ss}}))^{\scr W_G}$ contained in Lemma \ref{L:bilA}\eqref{L:bilA3}. 

The last assertion is straightforward.
\end{proof}

\subsection{Types $B_l$ and $C_l$ (with $l\geq 2$)}

Let us first recall some properties of the dual root systems $B_{l}$ and $C_l$.

Consider the vector space $V(B_l)=V(C_l)=\bbR^l$ endowed with the standard scalar product $(-,-)$ and with the canonical bases $\{\epsilon_1,\ldots, \epsilon_l\}$.  
We will freely identify $\bbR^l$ with its dual vector space by mean of the (restriction of the) standard scalar product $(-,-)$. 
The (co)root and (co)weight lattices of $B_l$ and $C_l$ are given by 
\begin{equation*}
\begin{aligned}
& Q(B_l)=Q(C_l^\vee)=\bbZ^l \subset P(B_l)=P(C_l^\vee)= \bbZ^l +\left\langle \omega_l:=\frac{\sum_i \epsilon_i}{2}\right\rangle,\\
& Q(C_l)=Q(B_l^\vee)=\{\xi\in \bbZ^l \: : (\xi,\xi) \: \text{ is even} \} \subset P(C_l)=P(B_l^\vee)= \bbZ^l.
\end{aligned}
\end{equation*}
It follows that group $P(B_l)/Q(B_l)$ is cyclic of order $2$ generated by $\omega_l$, while  $P(C_l)/Q(C_l)$ is cyclic of order two generated by, say, $\epsilon_1$. The Weyl group  of $B_l$ and $C_l$ is  equal to $\scr W(B_l)=\scr W(C_l)=(\bbZ/2\bbZ)^l\rtimes S_l$ and it acts on the above lattices in such a way that $S_l$  permutes the coordinates of $ \bbR^l$ while $(\bbZ/2\bbZ)^l$ changes the signs of all the coordinates. 

A semisimple group $H$ which is almost-simple of type $B_{l}$ (resp. $C_l$) is isomorphic to either the (simply-connected) spin group $\Spin_{2l+1}$  (resp. the symplectic group $\Sp_{2l}$) or the (adjoint) orthogonal group $\SO_{2l+1}$ (resp. the projective symplectic group $\PSp_{2l}$). By choosing the standard maximal tours $T_H$ of $H$  consisting of diagonal matrices, we get the canonical identifications 
\begin{equation}\label{E:lattBC}
\begin{aligned}
\Lambda(T_{\Spin_{2l+1}})=\Lambda^*(T_{\PSp_{2l}})=\{\xi\in \bbZ^l \: : (\xi,\xi) \: \text{ is even} \}\subset    \Lambda(T_{\SO_{2l+1}})=\Lambda^*(T_{\Sp_{2l}})=\bbZ^l, \\
\Lambda^*(T_{\Spin_{2l+1}})=\Lambda(T_{\PSp_{2l}})=\bbZ^l +\left\langle \omega_l:=\frac{\sum_i \epsilon_i}{2}\right\rangle\supset   \Lambda^*(T_{\SO_{2l+1}})=\Lambda(T_{\Sp_{2l}})=\bbZ^l. \\
\end{aligned}
\end{equation}
It follows that the fundamental group of $H$ is equal to 
\begin{equation}\label{E:pi1BC}
\pi_1(H)=
\begin{cases}
 \{0\} & \text{ if } H=\Spin_{2l+1}, \Sp_{2l}, \\
 \frac{\Lambda(T_{\SO_{2l+1}})}{\Lambda(T_{\Spin_{2l+1}})}=\left\langle \epsilon_1 \right\rangle \cong \bbZ/2\bbZ & \text{ if } H=\SO_{2l+1},\\
  \frac{\Lambda(T_{\PSp_{2l}})}{\Lambda(T_{\Sp_{2l}})}=\left\langle \omega_l \right\rangle \cong \bbZ/2\bbZ & \text{ if } H=\PSp_{2l},\\
\end{cases}
\end{equation}
while the character group of the center $\scr Z(H)$ is equal to 
\begin{equation}\label{E:ZBC}
\Lambda^*(\scr Z(H))=
\begin{cases}
 \{0\} & \text{ if } H=\SO_{2l+1}, \PSp_{2l}, \\
 \frac{\Lambda^*(T_{\Spin_{2l+1}})}{\Lambda^*(T_{\SO_{2l+1}})}=\left\langle \omega_l \right\rangle \cong \bbZ/2\bbZ & \text{ if } H=\Spin_{2l+1},\\
  \frac{\Lambda^*(T_{\Sp_{2l}})}{\Lambda^*(T_{\PSp_{2l}})}=\left\langle \epsilon_1 \right\rangle \cong \bbZ/2\bbZ & \text{ if } H=\Sp_{2l},\\
\end{cases}
\end{equation}

From now on, we will consider the following 

\un{Set-up}: Let $G$ be a reductive group such that $\scr D(G)$ (or equivalently $G^{\ss}$) is almost simple of type $B_l$ or $C_l$. 
Equivalently, $G$ is the product of a torus and one of the following reductive groups  (see \cite[Example 2.3]{BLS98})
$$
\Spin_{2l+1}, \SO_{2l+1}, C_{\mu_2}(\Spin_{2l+1})=\CSpin_{2l+1}, \Sp_{2l}, \PSp_{2l}, C_{\mu_2}(\Sp_{2l})=\CSp_{2l},
$$
where $\CSpin_{2l+1}$ is the Clifford group of order $2l+1$ and $\CSp_{2l}$ is the group of symplectic similitudes of order $2l$. 


\begin{lem}\label{L:bilBC}
Let $G$ be a reductive group as in the above set-up. Then we have that 
\begin{enumerate}[(i)]
\item \label{L:bilBC1}  $\Bil^{s, \ev}(\Lambda(T_{G^{\sc}}))^{\scr W_G}=
\begin{cases}
\langle (-,-)\rangle & \text{ if } G^{\sc}=\Spin_{2l+1}, \\
\langle 2(-,-)\rangle & \text{ if } G^{\sc}=\Sp_{2l}.\\
\end{cases}$
\item  \label{L:bilBC2} $\Bil^{s, \sc-\ev}(\Lambda(T_{\scr D(G)})\vert \Lambda(T_{G^{\ss}}))^{\scr W_G}=
\begin{cases}
\langle (-,-)\rangle & \text{ if } \scr D(G)=\Spin_{2l+1} \: \text{ or } \SO_{2l+1},\\
\langle 2(-,-)\rangle & \text{ if either } \scr D(G)=\Sp_{2l} \\
&  \text{ or } \scr D(G)=\PSp_{2l} \: \text{ and } l \:\text{ is even},\\ 
\langle 4(-,-)\rangle & \text{ if } \scr D(G)=\PSp_{2l} \: \text{ and } l \: \text{ is odd}.\\
\end{cases}$
\item  \label{L:bilBC3} $\Bil^{s,\ev}(\Lambda(T_{\scr D(G)})\vert \Lambda(T_{G^{\ss}}))^{\scr W_G}=
\begin{cases}
\langle (-,-)\rangle & \text{ if } \scr D(G)=\Spin_{2l+1}, \\
\langle 2(-,-)\rangle & \text{ if either } \scr D(G)=\SO_{2l+1} \: \text{ or } \scr D(G)=\Sp_{2l} \\
& \text{ or } \scr D(G)=\PSp_{2l}\: \text{ and } 4\mid l,\\ 
\langle 4(-,-)\rangle & \text{ if } \scr D(G)=\PSp_{2l} \: \text{ and }  l\equiv 2 \mod 4, \\
\langle 8(-,-)\rangle & \text{ if } \scr D(G)=\PSp_{2l} \: \text{ and }  l\: \text{ is odd}. \\
\end{cases}$
\end{enumerate}
In particular, 
$$
\coker(r_G)=
\begin{cases}
\bbZ/2\bbZ & \text{ if either } \scr D(G)=\SO_{2l+1} \text{ or } \scr D(G)=\PSp_{2l}\: \text{ and } 4\nmid l,\\\
0 & \text{ otherwise.}
\end{cases}
$$
\end{lem}
Part \eqref{L:bilBC2}, combined with \cite{LS97}, recovers \cite[\S 4]{BLS98}  if $G=\Sp_{2l}$ or $\PSp_{2l}$ and \cite[(5.1)]{BLS98} if $G=\Spin_{2l+1}$ or $\SO_{2l+1}$.
\begin{proof}
First of all, the three lattices have dimension one by Corollary \ref{C:rank-Bil}. 

Part \eqref{L:bilBC1}: the symmetric bilinear form $(-,-)\in \Bil^s(\Lambda(T_{G^{\sc}}))$ is $\scr W(B_{l})=\scr W(C_l)$-invariant. 
From \eqref{E:lattBC}, it follows that if $G^{\sc}=\Spin_{2l+1}$ then $(-,-)$ is even and it generates $\Bil^{s, \ev}(\Lambda(T_{\Spin_{2l+1}}))^{\scr W(B_l)}$; while if 
$G^{\sc}=\Sp_{2l}$ then $2(-,-)$ is even and it generates $\Bil^{s, \ev}(\Lambda(T_{\Sp_{2l}}))^{\scr W(C_l)}$.

 Part \eqref{L:bilBC2}: if $\scr D(G)$ is almost simple of type $B_l$, then from \eqref{E:lattBC} it follows that $(-,-)$ is integral on $\Lambda(T_{\SO_{2l+1}})\otimes \Lambda(T_{\SO_{2l+1}})=\bbZ^l\otimes \bbZ^l$, which implies that 
 $\Bil^{s, \sc-\ev}(\Lambda(T_{\scr D(G)})\vert \Lambda(T_{G^{\ss}}))^{\scr W_G}=\Bil^{s, \ev}(\Lambda(T_{G^{\sc}}))^{\scr W_G}=\langle (-,-)\rangle$ by part \eqref{L:bilBC1}. 
  
If, instead, $\scr D(G)$ is almost simple of type $C_l$, then from \eqref{E:lattBC} it follows that $2(-,-)$ is always integral on $\Lambda(T_{\Sp_{2l}})\otimes \Lambda(T_{\PSp_{2l}})$ while
$2\alpha(-,-)$ (for $\alpha \in \bbZ$) it is integral on $\Lambda(T_{\PSp_{2l}})\otimes \Lambda(T_{\PSp_{2l}})$ if and only if 
$$
\bbZ\ni 2\alpha\left(\frac{\sum_i \epsilon_i}{2}, \frac{\sum_i \epsilon_i}{2}\right)=2\alpha\frac{l}{4} \Longleftrightarrow 2 \mid \alpha \cdot l,
$$
 from which the conclusion follows. 

Part \eqref{L:bilBC3}:  from \eqref{E:lattBC}, it follows that $(-,-)$ is even on $\Lambda(T_{\Spin_{2l+1}})$ while $\alpha (-,-)$ is even on $\Lambda(T_{\SO_{2l+1}})$ if and only if $\alpha$ is even, 
which gives the conclusion if if $\scr D(G)$ is almost simple of type $B_l$.

Again from \eqref{E:lattBC}  it follows that $2(-,-)$ is even on $\Lambda(T_{\Sp_{2l}})$, while $\alpha2(-,-)$ is even on  $\Lambda(T_{\PSp_{2l}})$ if and only if 
$$
2\bbZ\ni 2\alpha\left(\frac{\sum_i \epsilon_i}{2}, \frac{\sum_i \epsilon_i}{2}\right)=2\alpha\frac{l}{4} \Longleftrightarrow 4\mid \alpha\cdot l  \Longleftrightarrow
\begin{sis}
1 \mid \alpha & \quad \text{ if } 4\mid l, \\
2\mid \alpha & \quad \text{ if } l\equiv 2 \mod 4,\\
4\mid \alpha & \quad \text{ if } 2\nmid l,
\end{sis}
$$
from which the conclusion follows in the case when  $\scr D(G)$ is almost simple of type $C_l$.
\end{proof}


\begin{lem}\label{L:evBC}
Let $G$ be a reductive group as in the above set-up. Consider an element $\delta\in \pi_1(G)$ with image $\delta^{\ss}\in \pi_1(G^{\ss})$.  Then we have that 
$$\coker(\wt \ev_{\scr D(G)}^{\delta})=\coker(\ev_{\scr D(G)}^{\delta})=
\begin{cases}
0 & \text{ if either } \scr D(G)=\SO_{2l+1} \text{ or } \scr D(G)=\PSp_{2l},\\
& \text{ or } \scr D(G)=\Sp_{2l} \text{ and }  \delta^{\ss}\neq 0 \text{ and } 2\nmid l,\\
\bbZ/2\bbZ & \text{otherwise.}
\end{cases}
$$
\end{lem}
The computation of $\coker(\wt \ev_{\scr D(G)}^{\delta})$ for $G=\Sp_{2l}$ or $\PSp_{2l}$ can  be found in \cite[\S 8.2]{BH13} and for $G=\Spin_{2l+1}$ or $\SO_{2l+1}$ in \cite[\S 8.3]{BH13}. 
\begin{proof}
If $\scr D(G)$ is of type $B_l$, then $\wt \ev_{\scr D(G)}^{\delta}\equiv 0$ since $\Bil^{s, \sc-\ev}(\Lambda(T_{\scr D(G)})\vert \Lambda(T_{G^{\ss}}))^{\scr W_G}=\langle (-,-)\rangle$ by Lemma \ref{L:bilBC}\eqref{L:bilBC2} and 
$(-,-)$ is integral on $\Lambda(T_{\SO_{2l+1}})\otimes \Lambda(T_{\SO_{2l+1}})$ by \eqref{E:lattBC}. Therefore, the conclusion follows from \eqref{E:ZBC}.

Assume now that $\scr D(G)$ is of type $C_l$. The map $\wt \ev_{\scr D(G)}^{\delta}$ (and hence also $\ev_{\scr D(G)}^{\delta}$) is obviously zero if either $\delta^{\ss}=0$ (which is always the case if $G^{\ss}=\Sp_{2l}$  by \eqref{E:pi1BC}) or $\scr D(G)=\PSp_{2l}$.  In the remaining cases,  $\Bil^{s, \sc-\ev}(\Lambda(T_{\scr D(G)})\vert \Lambda(T_{G^{\ss}}))^{\scr W_G}=\Bil^{s,\ev}(\Lambda(T_{\scr D(G)})=\langle 2(-,-)\rangle$ by Lemma \ref{L:bilBC} and hence 
$$\wt \ev_{\scr D(G)}^{\delta}= \ev_{\scr D(G)}^{\delta}\equiv 0\Longleftrightarrow \bbZ\ni 2\left(\frac{\sum_i \epsilon_i}{2}, \frac{\sum_i \epsilon_i}{2}\right)=2\frac{l}{4}\Longleftrightarrow l \: \text{ is even. }$$
We conclude using \eqref{E:ZBC}.
\end{proof}

\subsection{Type $D_l$ (with $l\geq 3$)}

Let us first recall some properties of the root system $D_{l}$.

Consider the vector space $\bbR^l$ endowed with the standard scalar product $(-,-)$ and with the canonical bases $\{\epsilon_1,\ldots, \epsilon_l\}$.  
We will freely identify $\bbR^l$ with its dual vector space by mean of the (restriction of the) standard scalar product $(-,-)$. 
The root (resp. coroot) lattices and  the weight (resp. coweight) lattices of $D_{l}$ are given by 
\begin{equation*}
Q(D_{l})=Q(D_{l}^\vee)=\{\xi\in \bbZ^l \: : (\xi,\xi)\: \text{ is even}\}  \subset P(D_{l})=P(D_{l}^\vee)= \bbZ^l +\left\langle \frac{\sum_i \epsilon_i}{2}\right\rangle.
\end{equation*}
We set 
$$
\omega_l:=\frac{\epsilon_1+\ldots +\epsilon_l}{2} \quad \text{ and } \quad \omega_{l-1}:=\frac{\epsilon_1+\ldots+\epsilon_{l-1}-\epsilon_l}{2}.
$$
The group $P(D_l)/Q(D_l)$ is equal to 
\begin{equation*}
P(D_l)/Q(D_l)=\{2\epsilon_1=0,\epsilon_1,\omega_l,\omega_{l-1}\} \cong 
\begin{cases}
\bbZ/4\bbZ=\langle \omega_{l}\rangle=\langle \omega_{l-1}\rangle & \text{ if } l \: \text{ is odd.} \\
\bbZ/2\bbZ\times \bbZ/2\bbZ & \text{ if } l \: \text{ is even.} 
\end{cases} 
\end{equation*}
The Weyl group  of $D_l$  is  equal to 
$$\scr W(D_l)=(\bbZ/2\bbZ)^{l-1}\rtimes S_l=\{(\xi=(\xi_1,\ldots, \xi_l),\sigma)\in  (\bbZ/2\bbZ)^l\rtimes S_l \: : \prod_i \xi_i=1\},$$
 and it acts on the above lattices in such a way that $S_l$  permutes the coordinates of $ \bbR^l$ while $(\bbZ/2\bbZ)^{l-1}\leq (\bbZ/2\bbZ)^l$ changes the signs of all the coordinates. 

A semisimple group $H$ which is almost-simple of type $D_{l}$  is isomorphic to either the (simply-connected) spin group $\Spin_{2l}$, or the orthogonal group $\SO_{2l}$,   or the (adjoint) projective orthogonal group $\PSO_{2l}$ 
or, if $l$ is even, to the one of the two semisimple groups  
$$\Omega_{2l}^{\pm 1}:=\Spin_{2l}/\langle \omega_{l-\frac{1}{2}\pm \frac{1}{2}}\rangle.$$
Note that the two groups $\Omega_{2l}^{\pm 1}$ are (abstractly)  isomorphic; the isomorphism is induced by the automorphism of the Dynkin diagram $D_l$ that exchanges the last two nodes.

 By choosing the standard maximal tours $T_H$ of $H$  consisting of diagonal matrices, we get the canonical identifications 
\begin{equation}\label{E:lattD}
\begin{aligned}
&\Lambda(T_{\Spin_{2l}})=\{\xi\in \bbZ^l \: : (\xi,\xi) \: \text{ is even} \}\subset    \Lambda(T_{\SO_{2l}})=\bbZ^l\subset \Lambda(T_{\PSO_{2l}})=\bbZ^l+\left\langle \frac{\sum_i \epsilon_i}{2}\right\rangle, \\
&\Lambda^*(T_{\Spin_{2l}})=\bbZ^l +\left\langle \frac{\sum_i \epsilon_i}{2}\right\rangle\supset   \Lambda^*(T_{\SO_{2l}})=\bbZ^l \supset   \Lambda^*(T_{\SO_{2l}})=\{\xi\in \bbZ^l \: : (\xi,\xi) \: \text{ is even} \}, \\
&\Lambda(T_{\Omega_{2l}^{\pm}})=\Lambda^*(T_{\Omega_{2l}^{\pm}})=\{\xi\in \bbZ^l \: : (\xi,\xi) \: \text{ is even} \}+\langle \omega_{l-\frac{1}{2}\pm \frac{1}{2}} \rangle.
\end{aligned}
\end{equation}
It follows that the fundamental group of $H$ is equal to 
\begin{equation}\label{E:pi1D}
\pi_1(H)=
\begin{cases}
 \{0\} & \text{ if } H=\Spin_{2l}, \\
 \frac{\Lambda(T_{\SO_{2l}})}{\Lambda(T_{\Spin_{2l}})}=\left\langle \epsilon_1 \right\rangle \cong \bbZ/2\bbZ & \text{ if } H=\SO_{2l},\\
\frac{\Lambda(T_{\Omega_{2l}^{\pm}})}{\Lambda(T_{\Spin_{2l}})}=\langle \omega_{l-\frac{1}{2}\pm \frac{1}{2}} \rangle \cong \bbZ/2\bbZ & \text{ if } H=\Omega_{2l}^{\pm}, \\
  \frac{\Lambda(T_{\PSO_{2l}})}{\Lambda(T_{\Spin_{2l}})}=\{0,\epsilon_1,\omega_l,\omega_{l-1}\} & \text{ if } H=\PSO_{2l}.\\
\end{cases}
\end{equation}
while the character group of the center $\scr Z(H)$ is equal to 
\begin{equation}\label{E:ZD}
\Lambda^*(\scr Z(H))=
\begin{cases}
 \{0\} & \text{ if } H=\PSO_{2l},  \\
 \frac{\Lambda^*(T_{\SO_{2l}})}{\Lambda^*(T_{\PSO_{2l}})}=\left\langle \epsilon_1 \right\rangle \cong \bbZ/2\bbZ & \text{ if } H=\SO_{2l},\\
 \frac{\Lambda^*(T_{\Omega_{2l}^{\pm}})}{\Lambda^*(T_{\PSO_{2l}})}=\langle \omega_{l-\frac{1}{2}\pm \frac{1}{2}} \rangle \cong \bbZ/2\bbZ & \text{ if } H=\Omega_{2l}^{\pm}, \\
  \frac{\Lambda^*(T_{\Spin_{2l}})}{\Lambda^*(T_{\PSO_{2l}})}=\{0,\epsilon_1,\omega_l,\omega_{l-1}\} & \text{ if } H=\Spin_{2l},\\
 \end{cases}
\end{equation}

From now on, we will consider the following

\un{Set-up}: Let $G$ be a reductive group such that $\scr D(G)$ (or equivalently $G^{\ss}$) is almost simple of type $D_l$. 
Equivalently, $G$ is the product of a torus and one of the following reductive groups  (see \cite[Example 2.3]{BLS98})
$$
\begin{sis}
&\Spin_{2l}, \SO_{2l}, \PSO_{2l}, C_{\langle \epsilon_1\rangle}(\Spin_{2l})=\CSpin_{2l},  C_{\mu_2}(\SO_{2l})=\CSO_{2l}, C_{\scr Z(\Spin_{2l})}(\Spin_{2l}), \\
&\Omega_{2l}^{\pm}, C_{\langle \omega_{l-\frac{1}{2}\pm \frac{1}{2}}\rangle}(\Spin_{2l}), C_{\mu_2}(\Omega_{2l}^{\pm}) \quad  \text{ if } l \: \text{ is even,} 
\end{sis} 
$$
where $\CSpin_{2l}$ is the Clifford group of order $2l$ and $\CSO_{2l}$ is the group of special ortogonal similitudes of order $2l$. 



\begin{lem}\label{L:bilD}
Let $G$ be a reductive group as in the above set-up. Then we have that 
\begin{enumerate}[(i)]
\item \label{L:bilD1}  $\Bil^{s, \ev}(\Lambda(T_{G^{\sc}}))^{\scr W_G}=\langle (-,-)\rangle$.
\item  \label{L:bilD2} $\Bil^{s, \sc-\ev}(\Lambda(T_{\scr D(G)})\vert \Lambda(T_{G^{\ss}}))^{\scr W_G}=
\begin{cases}
\langle (-,-)\rangle & \text{ if either } \scr D(G)=\Spin_{2l}, \\
&  \text{ or } \scr D(G)=G^{\ss}=\SO_{2l},\\
& \text{ or } \scr D(G)=G^{\ss}=\Omega_{2l}^{\pm} \: \text{ and } 4 \mid l, \\
\langle 2(-,-)\rangle & \text{ if either } \scr D(G)=\SO_{2l}   \text{ and } G^{\ss}=\PSO_{2l},  \\
& \text{ or } \scr D(G)=\PSO_{2l} \: \text{ and } 2\mid l,\\ 
& \text{ or } \scr D(G)=G^{\ss}=\Omega_{2l}^{\pm} \: \text{ and } l\equiv 2 \mod 4, \\
& \text{ or } \scr D(G)=\Omega_{2l}^{\pm} \: \text{ and } G^{\ss}=\PSO_{2l}, \\
\langle 4(-,-)\rangle & \text{ if } \scr D(G)=\PSO_{2l} \: \text{ and } 2\nmid l.\\
\end{cases}$
\item  \label{L:bilD3} $\Bil^{s,\ev}(\Lambda(T_{\scr D(G)})\vert \Lambda(T_{G^{\ss}}))^{\scr W_G}=
\begin{cases}
\langle (-,-)\rangle & \text{ if } \scr D(G)=\Spin_{2l}, \\
\langle 2(-,-)\rangle & \text{ if either } \scr D(G)=\SO_{2l}  \\
& \text{ or } \scr D(G)=\PSO_{2l}\: \text{ and } 4\mid l,\\ 
& \text{ or } \scr D(G)=\Omega^{\pm}_{2l}\: \text{ and } 4\mid l,\\ 
\langle 4(-,-)\rangle & \text{ if } \scr D(G)=\PSO_{2l} \: \text{ and }  l\equiv 2 \mod 4, \\
& \text{ or } \scr D(G)=\Omega^{\pm}_{2l}\: \text{ and } l \equiv 2 \mod 4,\\ 
\langle 8(-,-)\rangle & \text{ if } \text{ if } \scr D(G)=\PSO_{2l} \: \text{ and }  2\nmid l. \\
\end{cases}$
\end{enumerate}
In particular, 
$$
\coker(r_G)=
\begin{cases}
\bbZ/2\bbZ & \text{ if either } \scr D(G)=G^{\ss}=\SO_{2l}, \\
&  \text{ or } \scr D(G)=G^{\ss}=\Omega_{2l}^{\pm} \: \text{ and } 4 \mid l,\\
& \text{ or } \scr D(G)=\Omega^{\pm}_{2l}\: \text{ and } l \equiv 2 \mod 4,\\ 
& \text{ or }  \scr D(G)=\PSO_{2l} \: \text{ and }  4\nmid l,\\
0 & \text{ otherwise.}
\end{cases}
$$
\end{lem}
Part \eqref{L:bilD2}, combined with \cite{LS97}, recovers \cite[\S 5]{BLS98} if $G=\Spin_{2l}$ or $\SO_{2l}$ or $\PSO_{2l}$.
\begin{proof}
First of all, the three lattices have dimension one by Corollary \ref{C:rank-Bil}. 

Part \eqref{L:bilD1}: the symmetric bilinear form $(-,-)\in \Bil^s(\Lambda(T_{\Spin_{2l}}))$ is $\scr W(D_{l})$-invariant and, from  \eqref{E:lattBC}, it follows that $(-,-)$ is even and it generates $\Bil^{s, \ev}(\Lambda(T_{\Spin_{2l}}))^{\scr W(D_l)}$.

 Part \eqref{L:bilD2}: from  \eqref{E:lattBC}, we get that (for any $\alpha$):
 \begin{itemize}
 \item $(-,-)$ is integral on $\Lambda(T_{\Spin_{2l}})\otimes \Lambda(T_{\PSO_{2l}})$ and on $\Lambda(T_{\SO_{2l}})\otimes \Lambda(T_{\SO_{2l}})$;
 \item $\alpha(-,-)$ is integral on $\Lambda(T_{\SO_{2l}})\otimes \Lambda(T_{\PSO_{2l}})$ if and only $2\mid \alpha$;
 \item $\alpha(-,-)$ is integral on $\Lambda(T_{\PSO_{2l}})\otimes \Lambda(T_{\PSO_{2l}})$ if and only if 
 $$ 
 \bbZ \ni \alpha\left(\epsilon_1, \frac{\sum_i \epsilon_i}{2}\right)=\frac{\alpha}{2}  \text{ and } \:  \bbZ \ni \alpha\left(\frac{\sum_i \epsilon_i}{2}, \frac{\sum_i \epsilon_i}{2}\right)=\frac{\alpha\cdot l}{4} 
  \Longleftrightarrow
   \begin{sis} 
   2\mid \alpha & \quad \: \text{ if } 2\mid l, \\
   4\mid \alpha & \quad \text{ if } 2\nmid l; 
  \end{sis} 
 $$
  \item $\alpha(-,-)$ is integral on $\Lambda(T_{\Omega_{2l}^{\pm}})\otimes \Lambda(T_{\Omega_{2l}^{\pm}})$ if and only if
  $$
  \bbZ\ni \alpha\left(\omega_{l-\frac{1}{2}\pm \frac{1}{2}}, \omega_{l-\frac{1}{2}\pm \frac{1}{2}}  \right)=\frac{\alpha\cdot  l}{4} 
   \Longleftrightarrow
   \begin{sis} 
   1\mid \alpha & \quad \: \text{ if } 4\mid l, \\
   2\mid \alpha & \quad \text{ if } l\equiv 2 \mod 4; 
  \end{sis} 
  $$
   \item $\alpha(-,-)$ is integral on $\Lambda(T_{\Omega_{2l}^{\pm}})\otimes \Lambda(T_{\PSO_{2l}})$ if and only if 
 $$ 
 \bbZ \ni \alpha\left(\epsilon_1, \omega_{l-\frac{1}{2}\pm \frac{1}{2}}\right)=\frac{\alpha}{2}  \text{ and } \:  \bbZ\ni \alpha\left(\omega_{l-\frac{1}{2}\pm \frac{1}{2}}, \omega_{l-\frac{1}{2}\pm \frac{1}{2}}  \right)=\frac{\alpha\cdot  l}{4}
  \Longleftrightarrow 2\mid \alpha. 
  $$
 \end{itemize}
 Combining the above equivalences, part  \eqref{L:bilD2} follows.

  Part \eqref{L:bilD3}: from  \eqref{E:lattBC}, we get that (for any $\alpha$):
 \begin{itemize}
 \item $(-,-)$ is even on $\Lambda(T_{\Spin_{2l}})\otimes \Lambda(T_{\PSO_{2l}})$;
 \item $\alpha(-,-)$ is even on $\Lambda(T_{\SO_{2l}})\otimes \Lambda(T_{\SO_{2l}})$ if and only if $2\mid \alpha$;
 \item $\alpha(-,-)$ is even on $\Lambda(T_{\PSO_{2l}})\otimes \Lambda(T_{\PSO_{2l}})$ if and only 
 $$
 2\bbZ\ni \alpha(\epsilon_1,\epsilon_1)=\alpha \quad \text{ and } \quad   2\bbZ\ni \alpha\left(\frac{\sum_i \epsilon_i}{2}, \frac{\sum_i \epsilon_i}{2}\right)=\frac{\alpha\cdot l}{4} 
 \Longleftrightarrow
   \begin{sis} 
   2\mid \alpha & \quad \: \text{ if } 4\mid l, \\
   4\mid \alpha & \quad \text{ if }  l\equiv 2 \mod 4,\\
   8 \mid \alpha & \quad \text{ if } 2\nmid l;
  \end{sis} 
 $$
  \item $\alpha(-,-)$ is even on on $\Lambda(T_{\Omega_{2l}^{\pm}})\otimes \Lambda(T_{\Omega_{2l}^{\pm}})$ if and only if
  $$
 2\bbZ\ni \alpha\left(\frac{\sum_i \epsilon_i}{2}, \frac{\sum_i \epsilon_i}{2}\right)=\frac{\alpha\cdot l}{4} 
    \Longleftrightarrow
   \begin{sis} 
   2\mid \alpha & \quad \: \text{ if } 4\mid l, \\
   4\mid \alpha & \quad \text{ if } l\equiv 2 \mod 4.
  \end{sis} 
  $$
  \end{itemize}
 Combining the above equivalences with part \eqref{L:bilD2}, we get \eqref{L:bilD3}.
\end{proof}

\begin{lem}\label{L:evD}
Let $G$ be a reductive group as in the above set-up. Consider an element $\delta\in \pi_1(G)$ and denote by  $\delta^{\ss}$ its image in $\pi_1(G^{\ss})$.
Then we have that 
\begin{enumerate}[(i)]
\item \label{L:evD1} $\coker(\wt \ev_{\scr D(G)}^{\delta})\cong
\begin{cases}
\bbZ/4\bbZ & \text{ if } \scr D(G)=\Spin_{2l}, 2\nmid l, \delta^{\ss}=0,\\
\bbZ/2\bbZ\times \bbZ/2\bbZ & \text{ if } \scr D(G)=\Spin_{2l}, 2\mid l, \delta^{\ss}=0,\\
\bbZ/2\bbZ & \text{ if } \scr D(G)=\Spin_{2l},  \ord(\delta^{\ss})=2,\\
& \text{ or }  \scr D(G)=G^{\ss}=\SO_{2l},  \delta^{\ss}=0, \\
& \text{ or } \scr D(G)=\SO_{2l}, G^{\ss}=\PSO_{2l}, \ord(\delta^{\ss})\neq 4, \\
& \text{ or } \scr D(G)=G^{\ss}=\Omega^{\pm}_{2l}, \delta^{\ss}=0, \\
& \text{ or } \scr D(G)=\Omega^{\pm}_{2l}, G^{\ss}=\PSO_{2l}, \\
0 & \text{ otherwise, } 
\end{cases}
$\\
where $\ord(\delta^{\ss})$ is the order of $\delta^{\ss}$ inside the group $\pi_1(G^{\ss})$. 
\item \label{L:evD2} $\coker(\ev_{\scr D(G)}^{\delta})\cong 
\begin{cases}
\bbZ/4\bbZ & \text{ if } \scr D(G)=\Spin_{2l}, 2\nmid l, \delta^{\ss}=0,\\
\bbZ/2\bbZ\times \bbZ/2\bbZ & \text{ if } \scr D(G)=\Spin_{2l}, 2\mid l, \delta^{\ss}=0,\\
\bbZ/2\bbZ & \text{ if } \scr D(G)=\Spin_{2l},  \ord(\delta^{\ss})=2,\\
& \text{ or }  \scr D(G)=G^{\ss}=\SO_{2l}, \\
& \text{ or } \scr D(G)=\SO_{2l}, G^{\ss}=\PSO_{2l}, \ord(\delta^{\ss})\neq 4, \\
& \text{ or } \scr D(G)=G^{\ss}=\Omega^{\pm}_{2l}, \\
& \text{ or } \scr D(G)=\Omega^{\pm}_{2l}, G^{\ss}=\PSO_{2l}, \\
0 & \text{ otherwise. } 
\end{cases}
$
\end{enumerate}
In particular, 
$$
|\coker(\ev_{\scr D(G)}^{\delta})|=
\begin{cases} 
|\coker(\wt \ev_{\scr D(G)}^{\delta})|+1 &  \text{ if }  \scr D(G)=G^{\ss}=\SO_{2l} \text{ and } \delta^{\ss}\neq 0, \\
& \text{ or } \scr D(G)=G^{\ss}=\Omega^{\pm}_{2l}  \: \text{ and }  \delta^{\ss}\neq 0,\\
|\coker(\wt \ev_{\scr D(G)}^{\delta})| & \text{ otherwise}.
\end{cases}
$$
\end{lem}
The computation of $\coker(\wt \ev_{\scr D(G)}^{\delta})$ for $G=\Spin_{2l}$ or $\SO_{2l}$ or $\PSO_{2l}$ or $\Omega_{2l}^{\pm}$ can be found in \cite[\S 8.3]{BH13}. 
\begin{proof}
We will distinguish several cases according to $\scr D(G)$ and $G^{\ss}$. We will freely use \eqref{E:pi1D} and \eqref{E:ZD}.
\begin{enumerate}[(a)]
\item $\scr D(G)=\PSO_{2l}$. In this case, the codomains of $\wt \ev_{\scr D(G)}^{\delta}$ and of $\ev_{\scr D(G)}^{\delta}$ are zero, hence their cokernels are zero.

\item $\scr D(G)=G^{\ss}=\SO_{2l}$.  In this case, we have that 
$$
 \Lambda^*(\scr Z(\scr D(G)))=\{0,\epsilon_1\}\cong \bbZ/2\bbZ \quad \text{ and } \quad  \pi_1(G^{\ss})=\{0,\epsilon_1\}\cong \bbZ/2\bbZ. 
$$ 
Since $\Bil^{s, \sc-\ev}(\Lambda(T_{\scr D(G)})\vert \Lambda(T_{G^{\ss}}))^{\scr W(G)}$ is generated by $(-,-)$ by Lemma \ref{L:bilD}\eqref{L:bilD2}, we have that 
$$
\wt \ev_{\scr D(G)}^{\delta}((-,-))=
\begin{cases}
\epsilon_1 & \text{ if } \delta^{\ss}=\epsilon_1,\\
0 & \text{ if } \delta^{\ss}=0,
\end{cases}
\Longrightarrow \coker(\wt \ev_{\scr D(G)}^{\delta})=
\begin{cases}
0 & \text{ if } \delta^{\ss}\neq 0, \\
\bbZ/2\bbZ & \text{ if } \delta^{\ss}=0. 
\end{cases}
$$
On the other hand, since $\Bil^{s, \ev}(\Lambda(T_{\scr D(G)})\vert \Lambda(T_{G^{\ss}}))^{\scr W(G)}$ is generated by $2(-,-)$ by Lemma \ref{L:bilD}\eqref{L:bilD3}, the map $\ev_{\scr D(G)}^{\delta}$ is zero for any $\delta$, which implies 
that $\coker(\ev_{\scr D(G)}^{\delta})=\bbZ/2\bbZ$. 

\item $\scr D(G)=\SO_{2l}$ and $G^{\ss}=\PSO_{2l}$. In this case, we have that 
$$
 \Lambda^*(\scr Z(\scr D(G)))=\{0,\epsilon_1\}\cong \bbZ/2\bbZ \quad \text{ and } \quad  \pi_1(G^{\ss})=\{0,\epsilon_1, \omega_l, \omega_{l-1}\}\cong 
 \begin{cases}
 \bbZ/4\bbZ & \text{ if } l \: \text{ is odd,}\\
  \bbZ/2\bbZ\times \bbZ/2\bbZ & \text{ if } l \: \text{ is even.}\\
 \end{cases} 
$$ 
Since $\Bil^{s, \ev}(\Lambda(T_{\scr D(G)})\vert \Lambda(T_{G^{\ss}}))^{\scr W(G)}=\Bil^{s, \sc-\ev}(\Lambda(T_{\scr D(G)})\vert \Lambda(T_{G^{\ss}}))^{\scr W(G)}$ is generated by $2(-,-)$ by Lemma \ref{L:bilD}, we have that 
$$
\begin{aligned}
\wt \ev_{\scr D(G)}^{\delta}(2(-,-))=\ev_{\scr D(G)}^{\delta}(2(-,-))=2\delta^{\ss}=
\begin{cases}
0 & \text{ if } \ord(\delta^{\ss})\neq 4, \\
\epsilon_1 & \text{ if } \ord(\delta^{\ss})=4. 
\end{cases} 
\end{aligned}
$$
Hence we get that 
$$\coker(\wt \ev_{\scr D(G)}^{\delta})=\coker(\ev_{\scr D(G)}^{\delta})=
\begin{cases}
0 & \text{ if }\ord(\delta^{\ss})=4,  \\
\bbZ/2\bbZ & \text{ if } \ord(\delta^{\ss})\neq 4. 
\end{cases}
$$

\item $\scr D(G)=G^{\ss}=\Omega_{2l}^{\pm}$ (and $l$ is even). In this case, we have that 
$$
 \Lambda^*(\scr Z(\scr D(G)))=\{0,\omega_{l-\frac{1}{2}\pm \frac{1}{2}} \}\cong \bbZ/2\bbZ \quad \text{ and } \quad  \pi_1(G^{\ss})=\{0,\omega_{l-\frac{1}{2}\pm \frac{1}{2}} \}\cong \bbZ/2\bbZ. 
$$ 
Since $\Bil^{s, \sc-\ev}(\Lambda(T_{\scr D(G)})\vert \Lambda(T_{G^{\ss}}))^{\scr W(G)}$ is generated by $(-,-)$ by Lemma \ref{L:bilD}\eqref{L:bilD2}, we have that 
$$
\wt \ev_{\scr D(G)}^{\delta}((-,-))=
\begin{cases}
\omega_{l-\frac{1}{2}\pm \frac{1}{2}}  & \text{ if } \delta^{\ss}=\omega_{l-\frac{1}{2}\pm \frac{1}{2}} ,\\
0 & \text{ if } \delta^{\ss}=0,
\end{cases}
\Longrightarrow \coker(\wt \ev_{\scr D(G)}^{\delta})=
\begin{cases}
0 & \text{ if } \delta^{\ss}\neq 0, \\
\bbZ/2\bbZ & \text{ if } \delta^{\ss}=0. 
\end{cases}
$$
On the other hand, since $\Bil^{s, \ev}(\Lambda(T_{\scr D(G)})\vert \Lambda(T_{G^{\ss}}))^{\scr W(G)}$ is generated by $2(-,-)$ by Lemma \ref{L:bilD}\eqref{L:bilD3}, the map $\ev_{\scr D(G)}^{\delta}$ is zero for any $\delta$, which implies 
that $\coker(\ev_{\scr D(G)}^{\delta})=\bbZ/2\bbZ$.

\item $\scr D(G)=\Omega_{2l}^{\pm}$ and $G^{\ss}=\PSO_{2l}$ (and $l$ is even). In this case, we have that 
$$
 \Lambda^*(\scr Z(\scr D(G)))=\{0,\omega_{l-\frac{1}{2}\pm \frac{1}{2}} \}\cong \bbZ/2\bbZ \quad \text{ and } \quad   \pi_1(G^{\ss})=\{0,\epsilon_1, \omega_l, \omega_{l-1}\}.
$$ 
Since $\Bil^{s, \sc-\ev}(\Lambda(T_{\scr D(G)})\vert \Lambda(T_{G^{\ss}}))^{\scr W(G)}$ is generated by $2(-,-)$ by Lemma \ref{L:bilD}\eqref{L:bilD2}, we have that $\wt \ev_{\scr D(G)}^{\delta}\equiv 0$, which implies that 
$$
\coker(\wt \ev_{\scr D(G)}^{\delta})=\coker(\ev_{\scr D(G)}^{\delta})\cong \bbZ/2\bbZ.
$$

\item $\scr D(G)=\Spin_{2l}$. In this case we have that
$$
 \Lambda^*(\scr Z(\scr D(G)))= \Lambda^*(\scr Z(\Spin_{2l}))=\{0,\epsilon_1, \omega_l, \omega_{l-1}\}\cong 
 \begin{cases}
 \bbZ/4\bbZ & \text{ if } l \: \text{ is odd,}\\
  \bbZ/2\bbZ\times \bbZ/2\bbZ & \text{ if } l \: \text{ is even,}\\
 \end{cases} 
$$
where $\pi_1(G^{\ss})$ is a subgroup of $\pi_1(G^{\ad})=\pi_1(\PSO_{2l})$, which is canonically isomorphic to $ \Lambda^*(\scr Z(\Spin_{2l}))$ via the standard scalar product $(-,-)$. 
Since $\Bil^{s, \ev}(\Lambda(T_{\scr D(G)})\vert \Lambda(T_{G^{\ss}}))^{\scr W(G)}=\Bil^{s, \sc-\ev}(\Lambda(T_{\scr D(G)})\vert \Lambda(T_{G^{\ss}}))^{\scr W(G)}$ is generated by $(-,-)$ by Lemma \ref{L:bilD}, we have that 
$$
\wt \ev_{\scr D(G)}^{\delta}((-,-))=\ev_{\scr D(G)}^{\delta}((-,-))=\delta^{\ss}\in \pi_1(G^{\ss})\subseteq \pi_1(\PSO_{2l})\cong  \Lambda^*(\scr Z(\Spin_{2l})).
$$
Hence we deduce that 
$$\coker(\wt \ev_{\scr D(G)}^{\delta})=\coker(\ev_{\scr D(G)}^{\delta})=
\begin{cases}
\bbZ/4\bbZ & \text{ if } 2\nmid l, \delta^{\ss}=0,\\
\bbZ/2\bbZ\times \bbZ/2\bbZ & \text{ if }  2\mid l, \delta^{\ss}=0,\\
\bbZ/2\bbZ & \text{ if }  \ord(\delta^{\ss})=2,\\
0 &  \text{ if } \ord(\delta^{\ss})=4 \: (\text{which can occur only if } 2\nmid l). 
\end{cases}
$$
\end{enumerate}
\end{proof}

\subsection{Types $E_6$, $E_7$ and $E_8$}

Let us first recall some properties of the root systems $E_6$, $E_7$ and $E_8$.

Consider the vector space $\bbR^8$ endowed with the standard scalar product $(-,-)$ and with the canonical bases $\{\epsilon_1,\ldots, \epsilon_l\}$.  
Inside $\bbR^8$, we will consider the following subvector spaces
\begin{equation*}
\begin{aligned}
V(E_7):=\{\xi=(\xi_1,\ldots, \xi_8)\in \bbR^8\: : \xi_8=-\xi_7\}, \\
V(E_6):=\{\xi=(\xi_1,\ldots, \xi_8)\in \bbR^8\: : \xi_8=-\xi_6, \xi_7=\xi_6\}. \\
\end{aligned}
\end{equation*}
We will freely identify $\bbR^8$, $V(E_7)$ and $V(E_6)$ with its dual vector spaces by mean of the (restriction of the) standard scalar product $(-,-)$. 
The root (resp. coroot) lattices and  the weight (resp. coweight) lattices of $E_8, E_7$ and $E_6$ are given by 
\begin{equation}\label{E:QP-E}
\begin{aligned}
Q(E_8)=Q(E_8^\vee)=\{\xi\in \bbZ^l \: : (\xi,\xi)\: \text{ is even}\} +\left\langle \frac{\sum_i \epsilon_i}{2}\right\rangle= P(E_8)=P(E_8^\vee),\\
Q(E_7)=Q(E_7^{\vee})=Q(E_8)\cap V(E_7)\subset P(E_7)=P(E_7^\vee)=Q(E_7)+\left\langle \omega_7:=\frac{2\epsilon_6+\epsilon_7-\epsilon_8}{2}\right\rangle, \\
Q(E_6)=Q(E_6^{\vee})=Q(E_8)\cap V(E_6)\subset P(E_6)=P(E_6^\vee)=Q(E_6)+\left\langle \omega_1:=\frac{2}{3}(\epsilon_8-\epsilon_7-\epsilon_6)\right\rangle, \\
\end{aligned}
\end{equation}
In particular,  the group $P(E_*)/Q(E_*)$ is equal to 
\begin{equation*}
\begin{aligned}
P(E_8)/Q(E_8)=\{0\}, \\
P(E_7)/Q(E_7)=\left \langle \omega_7 \right\rangle \cong \bbZ/2\bbZ, \\
P(E_6)/Q(E_6)=\left \langle \omega_1 \right\rangle \cong \bbZ/3\bbZ.
\end{aligned}
\end{equation*}
An explicit bases of the lattices $Q(E_8)$, $Q(E_7)$ and $Q(E_6)$ is given by 
\begin{equation}\label{E:basesE}
\begin{aligned}
&Q(E_8)=\langle \alpha_1,\alpha_2,\alpha_3,\alpha_4, \alpha_5, \alpha_6, \alpha_7, \alpha_8\rangle, \\
&Q(E_7)=\langle \alpha_1,\alpha_2,\alpha_3,\alpha_4, \alpha_5, \alpha_6, \alpha_7\rangle, \\
&Q(E_6)=\langle \alpha_1,\alpha_2,\alpha_3,\alpha_4, \alpha_5, \alpha_6\rangle,  \\
\end{aligned}
\end{equation}
where the elements $\alpha_i$ are given by 
\begin{equation*}
\begin{aligned}
& \alpha_1=\frac{\epsilon_1-\epsilon_2-\epsilon_3-\epsilon_4-\epsilon_5-\epsilon_6-\epsilon_7+\epsilon_8}{2}, \alpha_2=\epsilon_1+\epsilon_2, \alpha_3=\epsilon_2-\epsilon_1, \alpha_4=\epsilon_3-\epsilon_2, \\
& \alpha_5=\epsilon_4-\epsilon_3, \alpha_6=\epsilon_5-\epsilon_4, \alpha_7=\epsilon_6-\epsilon_5, \alpha_8=\epsilon_7-\epsilon_6.
\end{aligned}
\end{equation*}

From \eqref{E:QP-E}, it follows that there the following semisimple algebraic groups which are almost-simple of type $E_*$: a  simply-connected and adjoint  group $\bbE_8=\bbE_8^{\sc}=\bbE_8^{\ad}$ of type $E_8$, a simply-connected group $\bbE_7^{\sc}$ (resp. 
$\bbE_6^{\sc}$) of type $E_7$ (resp. $E_6$), an adjoint group $\bbE_7^{\ad}$ (resp. $\bbE_6^{\ad}$) of type $E_7$ (resp. $E_6$).
The (co)character lattices of the maximal tori of the above semisimple groups are therefore equal to 
\begin{equation}\label{E:lattE}
\begin{aligned}
&\Lambda(T_{\bbE_8})=\Lambda^*(T_{\bbE_8})=Q(E_8),\\
&\Lambda(T_{\bbE_7^{\sc}})=\Lambda^*(T_{\bbE_7^{\ad}})=Q(E_7)\subset \Lambda(T_{\bbE_7^{\ad}})=\Lambda^*(T_{\bbE_7^{\sc}})=P(E_7)=Q(E_7)+\left\langle \omega_7:=\frac{2\epsilon_6+\epsilon_7-\epsilon_8}{2}\right\rangle,\\
&\Lambda(T_{\bbE_6^{\sc}})=\Lambda^*(T_{\bbE_6^{\ad}})=Q(E_6)\subset \Lambda(T_{\bbE_6^{\ad}})=\Lambda^*(T_{\bbE_6^{\sc}})=P(E_6)+\left\langle \omega_1:=\frac{2}{3}(\epsilon_8-\epsilon_7-\epsilon_6)\right\rangle.\\
\end{aligned}
\end{equation}
It follows that the fundamental group of a semisimple group $H$ as above is equal to 
\begin{equation}\label{E:pi1E}
\pi_1(H)=
\begin{cases}
 \{0\} & \text{ if } H=\bbE_8, \bbE_7^{\sc}, \bbE_6^{\sc},  \\
 \frac{\Lambda(T_{\bbE_7^{\ad}})}{\Lambda(T_{\bbE_7^{\sc}})}=\left\langle \omega_7:=\frac{2\epsilon_6+\epsilon_7-\epsilon_8}{2} \right\rangle \cong \bbZ/2\bbZ & \text{ if } H=\bbE_7^{\ad},\\
  \frac{\Lambda(T_{\bbE_6^{\ad}})}{\Lambda(T_{\bbE_6^{\sc}})}=\left\langle\omega_1:=\frac{2}{3}(\epsilon_8-\epsilon_7-\epsilon_6) \right\rangle \cong \bbZ/3\bbZ & \text{ if } H=\bbE_6^{\ad},\\
\end{cases}
\end{equation}
while the character group of the center $\scr Z(H)$ is equal to 
\begin{equation}\label{E:ZE}
\Lambda^*(\scr Z(H))=
\begin{cases}
 \{0\} & \text{ if } H=\bbE_8, \bbE_7^{\ad}, \bbE_6^{\ad},  \\
 \frac{\Lambda^*(T_{\bbE_7^{\sc}})}{\Lambda^*(T_{\bbE_7^{\ad}})}=\left\langle \omega_7:=\frac{2\epsilon_6+\epsilon_7-\epsilon_8}{2} \right\rangle \cong \bbZ/2\bbZ & \text{ if } H=\bbE_7^{\sc},\\
  \frac{\Lambda^*(T_{\bbE_6^{\sc}})}{\Lambda^*(T_{\bbE_6^{\ad}})}=\left\langle\omega_1:=\frac{2}{3}(\epsilon_8-\epsilon_7-\epsilon_6) \right\rangle \cong \bbZ/3\bbZ & \text{ if } H=\bbE_6^{\sc}.\\
\end{cases}
\end{equation}

From now on, we will consider the following

\un{Set-up}: Let $G$ be a reductive group such that $\scr D(G)$ (or equivalently $G^{\ss}$) is almost simple of type $E_*$. 
Equivalently, $G$ is the product of a torus and one of the following reductive groups  (see \cite[Example 2.3]{BLS98})
$$
\bbE_8, \bbE_7^{\sc}, \bbE_7^{\ad}, C_{\mu_2}(\bbE_7^{\sc}), \bbE_6^{\sc}, \bbE_6^{\ad}, C_{\mu_3}(\bbE_6^{\sc}).
$$


\begin{lem}\label{L:bilE}
Let $G$ be a reductive group as in the above set-up. Then we have that 
\begin{enumerate}[(i)]
\item \label{L:bilE1}  $\Bil^{s, \ev}(\Lambda(T_{G^{\sc}}))^{\scr W_G}=\langle (-,-)\rangle$.
\item  \label{L:bilE2} $\Bil^{s, \sc-\ev}(\Lambda(T_{\scr D(G)})\vert \Lambda(T_{G^{\ss}}))^{\scr W_G}=
\begin{cases}
\langle (-,-)\rangle & \text{ if } \scr D(G)=\bbE_8 \: \text{ or } \scr D(G)=\bbE_7^{\sc}\: \text{ or } \scr D(G)=\bbE_6^{\sc}, \\
\langle 2(-,-)\rangle & \text{ if  } \scr D(G)=\bbE_7^{\ad}, \\
\langle 3(-,-)\rangle & \text{ if } \scr D(G)=\bbE_6^{\ad}.\\
\end{cases}$
\item  \label{L:bilE3} $\Bil^{s,\ev}(\Lambda(T_{\scr D(G)})\vert \Lambda(T_{G^{\ss}}))^{\scr W_G}=
\begin{cases}
\langle (-,-)\rangle & \text{ if } \scr D(G)=\bbE_8 \: \text{ or } \scr D(G)=\bbE_7^{\sc}\: \text{ or } \scr D(G)=\bbE_6^{\sc}, \\
\langle 4(-,-)\rangle & \text{ if  } \scr D(G)=\bbE_7^{\ad}, \\
\langle 3(-,-)\rangle & \text{ if } \scr D(G)=\bbE_6^{\ad}.\\
\end{cases}$
In particular, 
$$
\coker(r_G)=
\begin{cases}
\bbZ/2\bbZ & \text{ if  } \scr D(G)=\bbE_7^{\ad}, \\
0 & \text{ otherwise.}
\end{cases}
$$
\end{enumerate}
\end{lem}
\begin{proof}
First of all, the three lattices have dimension one by Corollary \ref{C:rank-Bil}. 

Part \eqref{L:bilE1}: the symmetric bilinear form $(-,-)\in \Bil^s(\Lambda(T_{\bbE_*^{\sc}}))$ is $\scr W(E_*)$-invariant and, from  \eqref{E:basesE}, it follows that $(-,-)$ is even and it generates $ \Bil^s(\Lambda(T_{\bbE_*^{\sc}}))^{\scr W(E_*)}$.

 Parts \eqref{L:bilE2} and \eqref{L:bilE3} for  $\scr D(G)=\bbE_8, \bbE_7^{\sc}, \bbE_6^{\sc}$ follow from part \eqref{L:bilE1} and Lemma \ref{L:intGad}. 
 
 Parts \eqref{L:bilE2} and \eqref{L:bilE3} for  $\scr D(G)= \bbE_7^{\ad}$ follows, using \eqref{E:lattE}, from 
 $$
\begin{aligned}
&  \alpha(-,-) \text{ is integral on } \Lambda(T_{\bbE_7^{\ad}})\otimes \Lambda(T_{\bbE_7^{\ad}}) \Longleftrightarrow \bbZ \ni \alpha\langle \omega_7, \omega_7\rangle=\frac{3\alpha}{2} \Longleftrightarrow 2 \mid \alpha, \\
&  \alpha(-,-) \text{ is even on } \Lambda(T_{\bbE_7^{\ad}})\otimes \Lambda(T_{\bbE_7^{\ad}}) \Longleftrightarrow \bbZ \ni \alpha\langle \omega_7, \omega_7\rangle=\frac{3\alpha}{2} \Longleftrightarrow 4 \mid \alpha, \\
\end{aligned}
 $$
 
 Parts \eqref{L:bilE2} and \eqref{L:bilE3} for  $\scr D(G)= \bbE_6^{\ad}$ follows, using \eqref{E:lattE}, from 
 $$
\begin{aligned}
&  \alpha(-,-) \text{ is integral on } \Lambda(T_{\bbE_6^{\ad}})\otimes \Lambda(T_{\bbE_6^{\ad}}) \Longleftrightarrow \bbZ \ni \alpha\langle \omega_1, \omega_1\rangle=\frac{4\alpha}{3} \Longleftrightarrow 3 \mid \alpha, \\
&  \alpha(-,-) \text{ is even on } \Lambda(T_{\bbE_6^{\ad}})\otimes \Lambda(T_{\bbE_6^{\ad}}) \Longleftrightarrow \bbZ \ni \alpha\langle \omega_1, \omega_1\rangle=\frac{4\alpha}{3} \Longleftrightarrow 3 \mid \alpha. \\
\end{aligned}
 $$
\end{proof}

\begin{lem}\label{L:evE}
Let $G$ be a reductive group as in the above set-up. Consider an element $\delta\in \pi_1(G)$ with image $\delta^{\ss}\in \pi_1(G^{\ss})$.  Then we have that 
$$\coker(\wt \ev_{\scr D(G)}^{\delta})=\coker(\ev_{\scr D(G)}^{\delta})=
\begin{cases}
\bbZ/2\bbZ & \text{ if either } G^{\ss}=\bbE_7^{\sc},\\
& \text{ or } \scr D(G)=\bbE_7^{\sc} \text{ and }  G^{\ss}=\bbE_7^{\ad} \text{ and } \delta^{\ss}=0,\\
\bbZ/3\bbZ & \text{ if either } G^{\ss}=\bbE_6^{\sc},\\
& \text{ or } \scr D(G)=\bbE_6^{\sc} \text{ and }  G^{\ss}=\bbE_6^{\ad} \text{ and } \delta^{\ss}=0,\\
0& \text{otherwise.}
\end{cases}
$$
\end{lem}
\begin{proof}
If $\scr D(G)$ is of type $E_8$, then the conclusion is obviouos since  $\scr D(G)=G^{\ad}$ by \eqref{E:lattE}.

Assume that $\scr D(G)$ is of type $E_7$ or  $E_6$.  The map $\wt \ev_{\scr D(G)}^{\delta}$ (and hence also $\ev_{\scr D(G)}^{\delta}$) is obviously zero if either $\delta^{\ss}=0$ (which is always the case if $G^{\ss}=\bbE_7^{\sc}$ or $\bbE_6^{\sc}$  by \eqref{E:pi1E}) or $\scr D(G)$ is equal to $\bbE_7^{\ad}$ or $\bbE_6^{\ad}$.  In the remaining cases,  i.e. $(\scr D(G), G^{\ss})=(\bbE_7^{\sc}, \bbE_7^{\ad})$ or $(\bbE_6^{\sc}, \bbE_6^{\ad})$ and $\delta^{\ss}=\omega_7\in \pi_1(\bbE_7^{\ad})$ or $\delta^{\ss}=m \omega_1\in \pi_1(\bbE_6^{\ad})$ with $1\leq m\leq 2$ (by \eqref{E:pi1E}),   Lemma \ref{L:bilE} implies that 
$\Bil^{s, \sc-\ev}(\Lambda(T_{\scr D(G)})\vert \Lambda(T_{G^{\ss}}))^{\scr W_G}=\Bil^{s,\ev}(\Lambda(T_{\scr D(G)})=\langle (-,-)\rangle$ and hence we get that $\wt \ev_{\scr D(G)}^{\delta}= \ev_{\scr D(G)}^{\delta}$ is non-zero (and hence surjective), because 
$$\ev_{\scr D(G)}^{\delta}((-,-))(\omega_7)=(\omega_7,\omega_7)=\left(\frac{2\epsilon_6+\epsilon_7-\epsilon_8}{2},\frac{2\epsilon_6+\epsilon_7-\epsilon_8}{2}\right)=\frac{6}{4}\not\in \bbZ,$$
$$\ev_{\scr D(G)}^{\delta}((-,-))(\omega_1)=(m\omega_1,\omega_1)=m\left(\frac{2}{3}(\epsilon_8-\epsilon_7-\epsilon_6),\frac{2}{3}(\epsilon_8-\epsilon_7-\epsilon_6)\right)=m\frac{4}{3}\not\in \bbZ.$$
We conclude using \eqref{E:ZE}.
\end{proof}

\subsection{Types $G_2$, $F_4$}

For the root systems of type $G_2$ and $F_4$, the (co)root and (co)weight lattices (whose explicit definition we will not need to recall) satisfy
\begin{equation}\label{E:QP-GF}
\begin{aligned}
& Q(F_4)=P(F_4) \: \text{ and } Q(F_4^\vee)=P(F_4^\vee), \\
& Q(G_2)=P(G_2) \: \text{ and } Q(G_2^\vee)=P(G_2^\vee).
\end{aligned}
\end{equation}
This implies that there is a unique semisimple group $\bbF_4$ (resp. $\bbG_2$) which is almost symple of type $F_4$ (resp. $G_2$), and these two groups are both simply-connected and adjoint. 
In particular, we have that 
\begin{equation}\label{E:lattFG}
\begin{aligned}
&\Lambda(T_{\bbF_4})=Q(F_4^\vee)=P(F_4^\vee) \text{ and } \Lambda^*(T_{\bbF_4})=Q(F_4)=P(F_4),\\
&\Lambda(T_{\bbG_2})=Q(G_2^\vee)=P(G_2^\vee) \text{ and } \Lambda^*(T_{\bbG_2})=Q(G_2)=P(G_2).\\
\end{aligned}
\end{equation}

From now on, we will consider the following

\un{Set-up}: Let $G$ be a reductive group such that $\scr D(G)=\bbF_4$ or $\bbG_2$, or equivalently $G$ is  the product of a torus and either $\bbF_4$ or $\bbG_2$.


\begin{lem}\label{L:bilFG}
Let $G$ be a reductive group as in the above set-up. 
 Then we have that 
 \begin{enumerate}[(i)]
\item \label{L:bilFG1} $\Bil^{s,\ev}(\Lambda(T_{\scr D(G)})\vert \Lambda(T_{G^{\ss}}))^{\scr W_G}=\Bil^{s, \sc-\ev}(\Lambda(T_{\scr D(G)})\vert \Lambda(T_{G^{\ss}}))^{\scr W_G}=\Bil^{s, \ev}(\Lambda(T_{G^{\sc}}))^{\scr W_G}\cong \bbZ.$
In particular, $\coker(r_G)=0$. 
\item \label{L:bilFG2} $\coker(\wt \ev_{\scr D(G)}^{\delta})=\coker(\ev_{\scr D(G)}^{\delta})=0$ for any $\delta\in \pi_1(G)$. 
\end{enumerate}
\end{lem}
\begin{proof}
Part \eqref{L:bilFG1}: the equality of the three lattices follows from \eqref{E:lattFG} and Lemma \ref{L:intGad}, while the fact that they are of rank one follows from Corollary \ref{C:rank-Bil}. 

Part \eqref{L:bilFG2} follows from the fact that $\scr D(G)=G^{\ad}$ by \eqref{E:lattFG}. 
\end{proof}

\vspace{0.5cm}

\noindent {\bf Acknowledgments.} We are  grateful to Giulio Codogni, Martina Lanini and Johan Martens for useful conversations. We would like to thank the referee for carefully reading the paper and for giving us constructive comments which helped to improve the readability of the paper. The first author acknowledges PRIN2017 CUP E84I19000500006, and the MIUR Excellence Department Project awarded to the Department of Mathematics, University of Rome Tor Vergata, CUP E83C18000100006. The second author is a member of the CMUC (Centro de Matem\'atica da Universidade de Coimbra), where part of this work was carried over.

\bibliographystyle{alpha}
\bibliography{BiblioVec}
\end{document}